\documentclass{amsart}

\usepackage[utf8]{inputenc}
\usepackage{fullpage}
\usepackage[colorlinks,linkcolor=blue,citecolor=blue]{hyperref}
\usepackage{graphicx}
\usepackage{amsmath}
\usepackage{amsthm}
\usepackage{amsfonts}
\usepackage{amssymb}
\usepackage{mathtools}
\usepackage{array}   % for \newcolumntype macro
\newcolumntype{L}{>{$}l<{$}}

\usepackage{pgf,tikz}
\usetikzlibrary{
  knots,
  hobby,
  decorations.pathreplacing,
  shapes.geometric,
  calc
}

\usepackage{mathrsfs}
\usetikzlibrary{arrows}
\newtheorem{Thm}{Theorem}[section]
\newtheorem{thm}{Theorem}[subsection]
\newtheorem{Cor}[Thm]{Corollary}
\newtheorem{prop}[Thm]{Proposition}
\newtheorem{lem}[Thm]{Lemma}
\newtheorem{conj}[Thm]{Conjecture}

\newtheorem{defn}[Thm]{Definition}

\theoremstyle{definition}

\newtheorem{exmp}[Thm]{Example}

\theoremstyle{remark}
\newtheorem{ques}[Thm]{Question}

\newtheorem{rmk}{Remark}[section]

\usepackage{eucal}
\usepackage{graphicx}
\usepackage{multirow}
\usepackage[all,knot]{xy}

\DeclareMathOperator{\SL}{SL}

\DeclareMathOperator{\SMod}{SMod}

\DeclareMathOperator{\Irr}{Irr}

\DeclareMathOperator{\Id}{Id}

\calclayout
\newcommand{\RNum}[1]{\uppercase\expandafter{\romannumeral #1\relax}}

\newcommand{\BZ}{\mathbb{Z}}
\newcommand{\Zn}[1]{\BZ/{#1}\BZ}

\newcommand{\Sym}[6]{\Big\{\begin{matrix}
  #1 & #2 & #3 \\
  #4 & #5 & #6 
\end{matrix}\Big\}}
\usepackage{graphicx}
\graphicspath{ {./graph/} }
\renewcommand{\bot}{\boxtimes}

\newcommand{\PP}{\mathcal{P}}

\newcommand{\CC}{{\mathcal{C}}}
\newcommand{\DD}{\mathcal{D}}

\newcommand{\LL}{\mathcal{L}}
\newcommand{\ZZ}{\mathcal{Z}}

\def\FF{\mathcal{F}}

\newcommand{\Hom}{\operatorname{Hom}}
\newcommand{\bt}{\boxtimes}

\begin{document}
\title{On $\Zn2$ permutation gauging}
\author{Zhengwei Liu and Yuze Ruan}
\address{Z. LIU, Yau Mathematical Sciences Center and Department of Mathematics, Tsinghua University, Beijing, 100084, China}
\address{Yanqi Lake Beijing Institute of Mathematical Sciences and Applications, Huairou District,
Beijing, 101408, China}
\email{liuzhengwei@mail.tsinghua.edu.cn}

\address{Y. Ruan, Yau Mathematical Sciences Center and Department of Mathematics, Tsinghua University, Beijing, 100084, China}
\email{ryz22@mails.tsinghua.edu.cn}

\begin{abstract}
  We explicitly construct a (unitary) $\Zn2$ permutation gauging of a (unitary) modular category $\CC$. In particular, the formula for the modular data of the gauged theory is provided in terms of modular data of $\CC$,  which provides positive evidence of the reconstruction program. Moreover as a direct consequence, the formula for the fusion rules is derived, verifying the conjectured formula in \cite{EJP19}. Our construction explicitly shows the genus-$0$ data of the gauged theory contains higher genus data of the original theory. As applications, we obtain an identity for the modular data that does not come from modular group relations, and we prove that representations of the symmetric mapping class group (associated to closed surfaces) coming from weakly group-theoretical modular categories have finite images.    
\end{abstract}
\maketitle
\section{Introduction}
Given a unitary modular category $\CC$, one may ask if it can be realized as the representation category for a certain nice conformal field theory (CFT), known as the reconstruction program \cite{Jones17, EG22}. 
One possible way to approach this question is to build and verify the analogy between conformal field theory and modular category theory \cite{LX19}. 
Permutation orbifolds are important types of construction in conformal field theory \cite{BHS98, KLX05, Bantay98}. Given a completely rational conformal net or a $C_2$ co-finite vertex operator algebra, one considers the fixed-point theory under the permutation action on the $n$-fold tensor product of the original theory. The concrete constructions of the permutation orbifolds, for example \cite{Bantay98, BHS98, Bantay02, KLM01, LX04, KLX05}, suggest a relation between
genus-$0$ data of permutation orbifold CFT and higher genus data of original CFT (twisted/untwisted correspondence), see \cite{BS11, Gui21} for the topological/geometric interpretation of this correspondence. Moreover, the $S$-matrix (genus $1$ data) of the orbifold CFT was calculated (see \cite{Bantay98, KLX05, BHS98} for the $\Zn2$ permutation case and \cite{DRX21, DXY22} for general cyclic permutation case), which only involves the genus-$1$ data (the $S$- and $T$-matrices) of original CFT. Therefore the analogous concrete constructions and interpretations on the categorical side become particularly intriguing.

The categorical counterpart to the (permutation) orbifold construction is known as (permutation) gauging \cite{Muger05}; for the physical theory and examples of gauging, see, for instance, \cite{BM10, BBCW19, BB19}. The general gauging is a two-step process: given a modular category $\CC$ and a categorical action of a finite group $G$, the first step is to extend $\CC$ to a $G$-crossed braided category (with a $G$-action) then the second step is applying the $G$-equivariantization to obtain a new modular category \cite{DGNO, CGPW16}. For the special case, when we start with the $n$-fold Deligne tensor product $\CC^{\boxtimes n}$ and the action given by permuting the tensor factors, the extension and the whole process will be called permutation extension and permutation gauging respectively. % While the possible extensions can be abstractly characterized by certain group cohomological data \cite{ENO10}, the details structures of the gauged theory 
The abstract existence of permutation extension was shown in \cite{GJ19}, by computing certain cohomological obstructions developed in \cite{ENO10}. However, the detailed structures of the gauged theory and their precise relation to the original category remain largely unexplored.

There are few related works in this direction.  \cite{BS11} gives a topological construction of the permutation extensions that are weakly rigid, and demonstrates rigidity for the $\Zn2$ case by working out the detailed structure morphisms.  \cite{LX19} constructs the $m$-interval Jones-Wassermann subfactors and proves the self-duality of them, which gives rise to certain subcategory of the permutation extension. \cite{EJP19} calculates the fusion rules for the $\Zn2$-gauged theory when $\CC$ is unpointed and conjectures the validity in general, see the discussion after \cite[Lem.~2.4]{EJP19}. 

In this paper, we completely settle the case when $G=\Zn2$. In particular, the modular data (the $S$- and $T$-matrices) of the gauged theory is explicitly calculated which coincides with those in \cite{Bantay98}, \cite{BHS98} and \cite{KLX05}. As a direct consequence, the formula for the fusion rules is derived, verifying the conjectured formula in \cite{EJP19}. 

More specifically, given a (unitary) modular category $\CC$, inspired by the methods developed in \cite{LX19}, we explicitly construct a twisted category $\DD$ which is a spherical $\Zn2$ permutation extension of $\CC\boxtimes\CC$.
%More specifically, given a (unitary) modular category $\CC$, inspired by the methods developed in \cite{LX19}, we construct a (shaded) planar algebra whose $n$-box space is given by $Cf(n,\ZZ):=\Hom_{\CC^{\boxtimes2}}(1,\otimes^n_{i=1}\gamma Z_{2i-2}\boxtimes Z_{2i-1})$, where $\gamma=\bigoplus_{X\in irr(\CC)} X\bot \bar{X}$ and $Z_i\in obj(\CC)$. The tangle actions are defined only involving the categorical data of $\CC$. Our first main theorem is the following,
%\begin{Thm}\label{mainthm1}
   %The planar algebra is self dual, moreover it can be lifted to an unshaded planar algebra with $\Zn2$ action and braiding (Definition \ref{def:Zn2_braiding(unshaded)}).
%\end{Thm}
%All structures of this unshaded planar algebra are constructed explicitly using categorical data of $\CC$. Therefore by taking the category of projections we prove,

\begin{Thm}\label{mainthm2}
  We obtain a $\Zn2$-graded (unitary) spherical fusion category $\DD$ with a $\Zn2$-crossed braiding, where the trivial graded part is braided equivalent to $\CC\boxtimes\CC$.    
\end{Thm}
This category is generated by objects in the nontrivial graded part, which are labeled by $\hat{V}$ for $V\in obj(\CC)$. The fusion rules are as follows, with the $\Zn2$ action $\rho$ acting on objects by $\rho(\hat{V})=\hat{V}$ and $\rho(X\bot Y)=Y\bot X$.
\begin{itemize}
    \item $\hat{V}$ is simple in $\DD$ $\Leftrightarrow$ $V$ is simple in $\CC$.
    \item $\hat{V}\hat{W}=\hat{W}\hat{V}=\oplus_{X\in \Irr(\CC)}XV\boxtimes \bar{X}W$, 
    \item $\hat{V}X\boxtimes Y=X\boxtimes Y\hat{V}=\widehat{VXY}$.
\end{itemize}

More importantly, the construction relies solely on the categorical data of $\CC$, and we establish explicitly the relationship between genus-$0$ data of $\DD$ and the higher genus data of the original theory $\CC$ (the twisted/untwisted correspondence). This is illustrated in the following table, where $\gamma=\bigoplus_{X\in \Irr(\CC)} X\bot \bar{X}$, $Z_i\in obj(\CC)$ and $\SMod(\Sigma_g)$ denotes the symmetric mapping class group of the closed genus-$g$ surface \cite{BH71, BH73, MW21}.
%Moreover, the projections in the odd part ($P_n$ for $n$ odd) of the planar algebra are generated by a single cap (projections in $P_1$) labeled by $V$ for $V\in obj(\CC)$,  we denote them by $\{\hat{V}\}_{V\in obj(\CC)}$, we have 
%    \begin{itemize}
%        \item $\hat{V}$ is minimal $\Leftrightarrow$ $V$ is simple.
%        \item $\hat{V}X\boxtimes Y=X\boxtimes Y\hat{V}=\widehat{VXY}$
%        \item $\hat{V}\hat{W}=\hat{W}\hat{V}=\oplus_{X\in irr(\CC)}XV\boxtimes XW$. 
%    \end{itemize}  
% The $\Zn2$ action $\rho$ acts on objects by $\rho(\hat{V})=\hat{V}$ and $\rho(X\bot Y)=Y\bot X$
%In addition, this theorem generalizes the partial result in \cite{LX19} that the $2$-interval Jones-Wassermann subfactor is self-dual. 
\begin{table}[h]
    \centering
    \begin{tabular}{c|c}
      Original category  $\CC$ & Twisted category $\DD$  \\
      \hline
$Cf(n,\ZZ):=\Hom_{\CC^{\boxtimes2}}(1,\otimes^n_{i=1}\gamma Z_{2i-2}\boxtimes Z_{2i-1})$  & $\Hom_{\DD}(1,\otimes^{2n-1}_{i=0}\hat{Z}_i)$ \\
       \hline
       $\SMod(\Sigma_{n-1})$ action $T_i$ (Definition \ref{def:braiding})& Crossed braid action $\sigma_i$ \\
       \hline
       Fourier pairing/transform (Figure \ref{fig:Fourier_pairing}, \ref{fig:Fourier_transform})  &  $\sigma_{2n-2}\sigma_{2n-1}\cdots \sigma_1\sigma_0$\\
       $S$-matrix, $T$-matrix (when $n=2,\ Z_i=1$) &$\sigma_0\sigma_1\sigma_2,\ \ \sigma_0=\sigma_2$ \\
       \hline
       Hyperelliptic involution (Definition \ref{def:Z_2_action}) & $\Zn2$ action $\rho$\\
       \hline
      
       \end{tabular}
    \caption{The twisted/untwisted correspondence}
    \label{tab:t/unt}
\end{table}

%The key idea of the proof is to relate the certain Fourier pairing (see Figure \ref{fig:Fourier_pairing} and \cite[Sec. ~2.3]{LX19}) with the $\Zn2$ braiding structures: $T_k$ (see Definition \ref{def:braiding}).  Our next main result describes the relation among those braidings. 
Our next main theorem describes the relation among $T_i$ in table \ref{tab:t/unt}.
\begin{Thm}\label{mainthm3}
   $T_k\in End(Cf(n,\ZZ))\ (0\leq k\leq 2n-2)$ satisfy the following relations 
   \[
   \begin{aligned}
       T_kT_{k+2}&=T_{k+2}T_k,\\
       T_{k+1}T_{k}T_{k+1}&=\eta^{(-1)^k}T_{k}T_{k+1}T_{k},\\
       (T_{2n-2}T_{2n-1}\cdots T_1T_0)^{2n}&=(\prod^{2n-1}_{k=0}\theta^{\frac{1}{2}}_{Z_k})\Id,\\
       (T_0T_1\cdots T_{2n-2}T_{2n-2}\cdots T_1T_0)^2&=\theta^2_{Z_0}\Id.
   \end{aligned}
   \]
\end{Thm}
In the special case when $Z_i$ are all trivial, we show the $\SMod(\Sigma_n)$-representations are equivalent to those arising from Reshetikhin-Turaev topological quantum field theory \cite{RT91} associated to the original modular category $\CC$, see Theorem \ref{thm:Rel_to_RT}. in this case, the Fourier transform coincides with the one in \cite[Sec. ~2.3]{LX19}, which establishes the self duality of the $2$-interval Jones-Wassermann subfactor. In particular, when $n=2$, we recover the well-known projective representation of $\SL_2(\mathbb{Z})$. Hence, our results provide an interpretation of the twisted/untwisted correspondence.
%In particular, if we use the same setting as in \cite{LX19}, then the braidings on the $n$-box space give a unitary projective representation of symmetric mapping class group for the closed surface of genus $n-1$, where the $\Zn2$ action is given by a hyperelliptic involution. Moreover in this case, we show the representations are equivalent to ones coming from Reshetikhin-Turaev topological quantum field theory \cite{RT91} associated to the original modular category $\CC$, see Theorem \ref{thm:Rel_to_RT}. In particular, when $n=2$, we obtain the well-known projective representation of $\SL_2(\mathbb{Z})$. Hence, our results interpret the twisted/untwisted correspondence \cite{Bantay98, Gui21}.

As a result, the simple isotopes in the graphic calculus of $\DD$ may involve nontrivial higher genus data of $\CC$. %category $\DD$ contains the higher genus data of the original theory,by simple isotopies in $\DD$
This allows us to derive several interesting identities in modular categories. In particular, we obtain a non-modular-group-relation identity, see Proposition \ref{New_identity}, which generalizes the self-duality of $6j$-symbols \cite{Liu19}. Furthermore, we have the following obstruction (infinitely many obstructions through successive $\Zn2$ permutation gauging, see Remark \ref{rem:infinite_obstruc}) for modular category realization of $S$-matrix.
\begin{equation}\label{eq:intro_obstruct}
  Det|_{\Lambda}(S^{\otimes 6}P-I^{\otimes 6})=0.   
 \end{equation}
Where $P$ is the composition of a permutation matrix permuting the tensor factors by the cycle $(16)(25)(34)$ with $C^{\otimes 3}\times I^{\otimes3}$ ($C$ is the charge conjugation matrix). $Det|_\Lambda$ is the determinant of the matrix restricted in subspace $\Lambda$ spanned by admissible colours (see Theorem \ref{thm:obstruct}). We expect one can derive more obstructions by fully exploring the rich structures within the twisted category $\DD$.

Next, based on the concrete relation between $\CC$ and $\DD$ in table \ref{tab:t/unt}, we examine the equivariantization theory $\DD^{\Zn2}$ and derive the modular data in terms of the modular data of the original modular category $\CC$, where the simple objects in $\DD^{\Zn2}$ are of the form $(XY):=(X\bt Y\oplus Y\bt X,\  \Id)$, $(X,\pm):=(X\bt X,\pm\Id)$ and $(\hat{X},\pm):=(\hat{X},\pm\Id)$ ($X,Y\in \Irr(\CC)$, $1_{\DD^{\Zn2}}=(1,+)$).
\begin{Thm}\label{thm:mainthm4}
 Let $S', S^{eq}$ denote the $S (unnormalized)$ matrix for $\CC$ and the gauged theory $\DD^{\Zn2}$ respectively. We have $S^{eq}$ is symmetric, invertible and $(\epsilon_1,\epsilon_2, \epsilon\in\{\pm\})$
\[
\begin{aligned}
   S^{eq}_{(XY),(ZW)}&=2(S'_{X,Z}S'_{Y,W}+S'_{XW}S'_{YZ} )\\
   S^{eq}_{(XY),(Z,\epsilon)}&=2S'_{X,Z}S'_{Y,Z}\\
    S^{eq}_{(X,\epsilon_1),(Y,\epsilon_2)}&=(S'_{X,Y})^2\\
   S^{eq}_{(XY),(\hat{Z},\epsilon)}&=0\\
   S^{eq}_{(X,\epsilon_1),(\hat{Y},\epsilon_2)}&=\epsilon_1 \delta^2 S'_{X,Y}\\
   S^{eq}_{(\hat{X},\epsilon_1),(\hat{Y},\epsilon_2)}&=\epsilon_1\epsilon_2\eta^{-1}\theta^{1/2}_X\theta^{1/2}_Y(S'T^2S')_{X,Y}\\
\end{aligned}
\]   
And $T^{eq}$ is diagonal matrix with
\[
\begin{aligned}
T^{eq}_{(XY),(XY)}&=\theta_X\theta_Y\\
   T^{eq}_{(X,\epsilon),(X,\epsilon)}&=\theta^2_X\\
   T^{eq}_{(\hat{X},\epsilon),(\hat{X},\epsilon)}&=\epsilon\eta^{1/2}\theta^{1/2}_X    
\end{aligned}
\]
\end{Thm}
Our formulas generalize the formula in \cite{Bantay98, BHS98, KLX05} from $\Zn2$ permutation orbifold CFT to modular categories. In particular, we obtain the formulas for the fusion rules which verify the conjectured formula in \cite{EJP19} for any modular category $\CC$.  

Finally, as another application of the correspondence in table \ref{tab:t/unt}, we prove the $\SMod(\Sigma)$ version of Property $F$ conjecture \cite{NR11} for weakly group-theoretical modular categories.
\begin{Thm}\label{thm:ProF}
The symmetric mapping class group representations (associated with closed surfaces) for weakly group-theoretical modular categories have finite images. 
\end{Thm}

The structure of the paper is as follows. In section $2$ we review the preliminaries of unitary modular categories, the gauging process and planar algebras. In section $3$, we 
discuss the lifting problem of a shaded planar algebra to an unshaded planar algebra and provide a sufficient condition using certain braiding structures on the even part of the shaded planar algebras. In section $4$, we define the generalized configuration spaces $Cf(n,\ZZ)$, the operations among them and check the compatibility among these operations. In section $5$, we prove Theorem %\ref{mainthm1}
 \ref{mainthm2}. In section $6$, we give some applications, and in particular, we prove Theorem \ref{mainthm3} and derive many identities among modular data and beyond, including the identity \ref{eq:intro_obstruct}. In section $7$, we examine the equivariantization theory and prove Theorem \ref{thm:mainthm4}. Finally, In section $8$, we prove Theorem \ref{thm:ProF}, and discuss some questions and future directions. 

\section*{Acknowledgement}
The authors would like to thank  Y. Kawahigashi, I. Runkel,  F. Xu for discussions and E. Rowell for useful comments on the preprint.  This work was supported by Beijing Municipal Science \& Technology Commission [Z221100002722017 to Z.L. and Y.R]; Beijing National Science Foundation Key Programs [Z220002 to Z.L.]; China’s National Key R\&D Programmes [2020YFA0713000 to Z.L.].
\section{Some background}
\subsection{Preliminaries of modular category}
We assume the readers are familiar with the notion of unitary modular categories and the corresponding graphic calculus for them. One may refer to \cite{Tur10, BK01, Row05f, EGNO} for more details.

Let $\CC$ be a unitary modular category, now we fix some notation and review some useful graphic calculus identities in $\CC$. 
\subsubsection{Notations}
\begin{itemize}
    \item $\overline{\ \ }$: the dual functor, $\bar{X},\ \bar{f}$,
    \item $\dagger$:  Involutive antilinear contravariant endofunctor from the unitary structure,  
    \item $d_X$: quantum dimension of the object $X$,
    \item $\mu$: global dimension of $\CC$,  $\mu=\dim(\CC)=\sum_{V\in \Irr(\CC)}d^2_V$,
    \item $\delta$: the positive square root of $\mu$,
    \item $\Omega$: the Kirby colour $\sum_{V\in \Irr(\CC)}d_V V$,
    \item $\theta^{\pm}_X$: the twist for $X$,
    \item $\theta^{\pm1/2}_X$: We pick a family of square root of twists, such that $\theta^{1/2}_{V}=\theta^{1/2}_{\bar{V}}$ for $V\in \Irr(\CC)$ and define $\theta^{1/2}_{V\oplus W}=\theta^{1/2}_{V}Id_V\oplus\theta^{1/2}_{W}Id_W $.  In particular, $\theta^{1/2}_{\bar{V}}=\overline{\theta^{1/2}_{V}}$.
    \item $p^{\pm}$: $p^{\pm}=\sum_{V\in \Irr(\CC)}\theta^{\pm}_Vd_V^2$, $p^+p^-=\mu$,
    \item $\eta$: $\eta=\frac{p^+}{\delta},\  \eta^{-1}=\frac{p^-}{\delta}$.
    \end{itemize}

Next we introduce some graphic notations. The diagrams are read from top to bottom.
\begin{itemize}
    \item The positive (negative) twists are denoted by the unfilled (filled) circle. We also denote the corresponding square root of the twists by the unfilled (filled) diamond symbol for simplicity.
    \[
    \includegraphics[width=250pt]{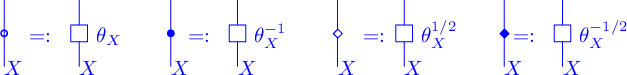}
    \]
    \item If the trivial object $1\in X$ with multiplicity one,  We use the unfilled square attached to the end of a strand (labeled by $X$) to denote the projection $X\mapsto 1$ or the inclusion $1\mapsto X$.
    \[
    \includegraphics[width=150pt]{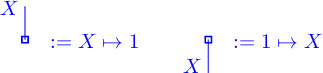}
    \]
    \item the $\dagger$ structure is given by the vertical reflection.
\end{itemize}
\subsubsection{Graphic calculus identities}
We use the red colored strands to indicate the strand is colored by $\Omega$. 
\begin{itemize}
    \item Twist property: $\theta_{X\otimes Y}=c_{Y,X}c_{X,Y}\theta_X\otimes\theta_Y$.
    \[
    \includegraphics[width=250pt]{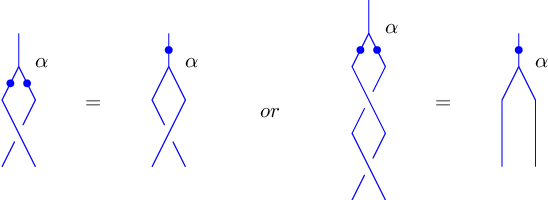}
    \]
    \item Cutting property of $\Omega$.
    \[
    \includegraphics[width=250pt]{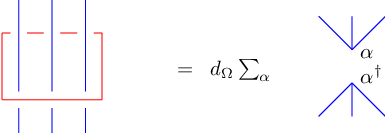}
    \]
    \item Handle slide property of $\Omega$
    \[
    \includegraphics[width=250pt]{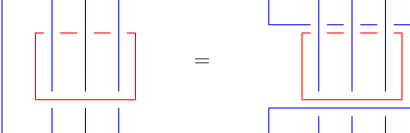}
    \]
\end{itemize}
\subsection{Preliminaries of gauging process}
We refer to \cite[Appendix~5]{TurHQFT}, \cite[Sec. 4]{DGNO} and \cite{CGPW16} for more details on the (unitary) $G$-crossed braided fusion category.

\begin{defn}
A (unitary) $G$-crossed braided fusion category $\mathcal{C}_G^{\times}$ is a (unitary) fusion category with
\begin{itemize}
\item (Unitary) action of $G$ on $\mathcal{C}_G^{\times}$.
\item Faithful $G$-grading $\mathcal{C}_G^{\times}=\oplus_{g\in G}\mathcal{C}_g$.
\item $G$-braiding: (unitary) natural isomorphisms 
$$
c_{X,Y}: X\otimes Y\rightarrow g(Y)\otimes X,\  g\in G, X\in \mathcal{C}_g, Y\in \mathcal{C}_G^{\times},
$$
\item Compatibility: 1. $g(\mathcal{C}_h)=\mathcal{C}_{ghg^{-1}}$, $\forall g,h\in G$,\\
2. $g(c_{X,Y})=c_{g(X),g(Y)}$, $\forall g,h\in G$,\\
3. some coherence, for example see \cite[Def.~4.41]{DGNO}
\end{itemize}
\end{defn}

\begin{defn}
Let $\mathcal{C}$ be a fusion category, $G$ be a finite group acting on $\mathcal{C}$. For any $g\in G$ let $\rho_g\in \underline{End(\CC)}$ be the corresponding action, and for any $g, h \in G$ let $\gamma_{g,h}$ be the isomorphism $\rho_g \circ \rho_h\cong \rho_{gh}$.The equivariantization of $\mathcal{C}$ is a category $\mathcal{C}^G$:
\begin{itemize}
\item \textbf{Object}: G-equivariant objects of $\mathcal{C}$, which is an object $X\in \CC$ together with isomorphisms $u_g: \rho_g(X)\cong X$ such that the diagram
$$
\begin{tikzpicture}[scale=1.5]
\node (A) at (-0.5,1) {$\rho_g(\rho_h(X))$};
\node (B) at (1.5,1) {$\rho_g(X)$};
\node (C) at (-0.5,0) {$\rho_{gh}(X)$};
\node (D) at (1.5,0) {$X$};
\path[->,font=\scriptsize,>=angle 90]
(A) edge node[above]{$\rho_g(u_h)$} (B)
(A) edge node[left]{$\gamma_{g,h}(X)$} (C)
(B) edge node[right]{$u_g$} (D)
(C) edge node[above]{$u_{gh}$} (D);
\end{tikzpicture}
$$
commutes for all $g,h\in G$.
\item \textbf{morphism}: morphisms in $\mathcal{C}$ commuting with $u_g$.
\end{itemize}  
\end{defn}

We refer to the \cite{Kir04, BN13} and \cite[Sec.~4]{DGNO} for more details and we summarize some known results about the $G$-equivariantization category.
\begin{itemize}
\item $\mathcal{C}^G$ is a fusion category.
\item $\CC$ is unitary, then $\CC^G$ is unitary by requiring $u_g$ are unitary isomorphisms \cite{CGPW16}.
\item If $\mathcal{C}$ is a $G$-crossed braided fusion category, then $\mathcal{C}^G$ is a braided fusion category, with the braiding $\tilde{c}$ defined as the composition
$$
X\otimes Y\xrightarrow{c_{X,Y}} g(Y)\otimes X\xrightarrow{u_g\otimes \Id_X} Y\otimes X
$$
in addition, $\mathcal{C}_1$ is non-degenerate $\Leftrightarrow\mathcal{C}^G$ is non-degenerate, where $\CC_1$ is the trivial grading part of $\CC$ 
\item $\dim(\mathcal{C}^G)=|G|\dim(\mathcal{C})$
\item The simple objects are parameterized by pairs $([X], \pi_X)$, where $[X]$ is an orbit of the $G$-action on simple objects of $\mathcal{C}$, and $\pi_X$ is an irreducible projective representation of $G_X$ (stabilizer group of $X$).
\item $\dim(([X], \pi_X))=\dim(\pi_X)\dim_{\mathcal{C}}(X)N_{[X]}$, where $N_{[X]}$ is the size of the orbit $[X]$.  
\end{itemize}

\begin{defn}\cite{CGPW16}
Let $\mathcal{C}$ be a unitary modular category with a global symmetry $(G,\rho)$, Gauging  is the two-step process:
\begin{itemize}
\item \textbf{G-extension}: Extend $\mathcal{C}$ to a $G$-crossed braided fusion category $\mathcal{C}^{\times}_{G}$.
\item \textbf{Equivariantization}: $\mathcal{C}^{\times}_{G}$ is equivariantized to a new unitary modular category $\mathcal{C}^{\times, G}_{G}$.  
\end{itemize}
\end{defn}

\subsection{Preliminaries of planar algebra}
The planar algebra was introduced by Vaughan Jones \cite{Jones21} as an axiomatization of standard 
invariant of subfactors. We refer to the following references for more details on the definition and relation to the category theory.
\begin{itemize}
    \item finite depth unshaded planar algebras and fusion categories. \cite[Sec.~ 2]{BHP12} \cite[Sec.~ 4.1]{MPS10}
    \item shaded planar algebra and 2-category \cite{Ghosh11}
\end{itemize}
 From now on, we will allow arcs to be colored and oriented, moreover we assume the coloring set has an involution which is compatible with reversing the orientations. In addition, as we are interested in the braiding structures, we allow arcs to intersect with each other transversely.
\begin{rmk}
It will be convenient to make the following simplifications and technical restrictions of our illustration of planar algebras.
\begin{itemize}
    \item Arcs intersect only the bottom boundary of input disks and output disks. And they cross the boundary of the output disk orthogonally.   
    \item The boundary of the disks and the arcs are all smooth. The corners appear in the diagram only for simplicity.
    \item We will often omit the output disks.
    \item We omit the $\$$ sign of the planar diagram if it is on the left. 
    \item The colors and orientations of the arcs will often be omitted if they are clear from the context.
\end{itemize}
\end{rmk}
It is noteworthy to present some special tangles, which give important structures of the planar algebra.
\begin{defn}\label{def:PA-composition}
 The algebra structure on $\{P_{2n}\}_{n\geq 0}$ is given by the following tangles (they are special cases of the general graded product defined in \cite{GJS10}), where $x_1\in P_{2(n_1+n_2)},\ x_2\in P_{2(n_2+n_3)},\ y\in P_{2n},\ z\in P_{2m}$. We use $x_1x_2$ to denote $x_1\wedge_{n_2} x_2$ when   $n_1=n_2=n_3=n$, and $y\otimes z$ for $y\otimes_{n_1}z$ when $n_1=n_2$.
  \[
 \includegraphics[width=350pt]{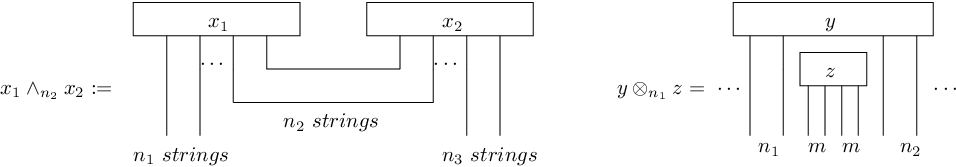}
 \]
\end{defn}
\begin{defn}
  A unitary structure on a (shaded) planar algebra is an antilinear involution $\Theta_1$ on the $P_n$, such that it is compatible with the reflection of the tangle. And the following pairing defines a positive definite inner product.
\[
  \includegraphics[width=200pt]{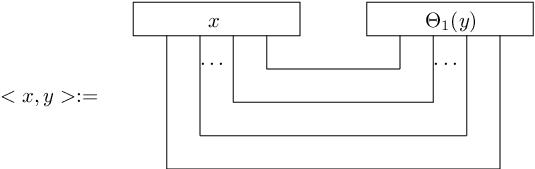}
\]
\end{defn}\label{def:Zn2_braiding(unshaded)}
The braiding on the certain planar algebra was discussed in \cite{MPS10}. They showed that braiding is a partial braiding in the sense that, away from certain input disks, diagrams in the planar algebra are equal up to framed three-dimensional isotopy. However the isotopy gives additional factors when going through these input disks \cite[Thm.~3.2]{MPS10}.

Here we will define a generalized notion of braidings on an unshaded planar algebra.
\begin{defn}
A $\Zn2$ braiding on an unshaded planar algebra is a braiding such that
\begin{itemize}
    \item there is an $\Zn2$ planar algebra action $\rho$ on $P_n$ (respect all the planar algebra structures including the braiding).
    \item Away from the input disks, the partition function $Z$ is invariant under framed three-dimensional isotopy.
    \item Strands can be isotopied above input disks. But being isotopied below input disks introduces the action $\rho$ on these input disks.
    \item The braiding is unitary when the underlying planar algebra is unitary.
\end{itemize}     
\end{defn}

Now to construct a planar algebra with a $\Zn2$ braiding, it is enough to define certain generating tangles and verify certain list of relations,  see for example the similar constructions in \cite{Jones21}, \cite{Tur10} and \cite{HPT23}. Here we will briefly summarize the construction.

\begin{figure}
    \centering
    \includegraphics[width=300pt]{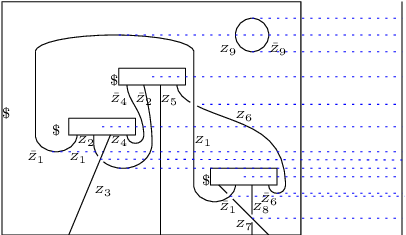}
    \caption{A generic planar tangle with braidings}
    \label{fig:PA-generic_planar_algebra}
\end{figure}

One defines the generic planar tangle such that the number of singular points (cap, cup, input disc, crossing) is finite and they are distinct with respect to the height function, see Figure \ref{fig:PA-generic_planar_algebra}. Now any planar tangle can be deformed to a generic one and thus can be assigned a map using defined generators. The isotopic (framed three-dimensional isotopy away from input disks) generic planar algebras can be related through the following moves. 
\begin{itemize}
    \item Framed three-dimensional isotopy in the class of generic planar algebra,
    \item Reidemeister moves, see Figure \ref{fig:Reidemeister},
     \begin{figure}
        \centering
        \includegraphics[width=350pt]{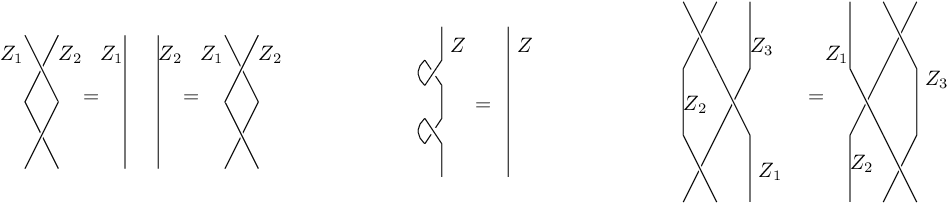}
        \caption{Reidemeister moves}
        \label{fig:Reidemeister}
    \end{figure}    
    \item Interchanging the order of two singular points with respect to
the height function, see Figure \ref{fig:exchagesingularity1}, \ref{fig:exchagesingularity2}
\begin{figure}
    \centering
    \includegraphics[width=0.8\linewidth]{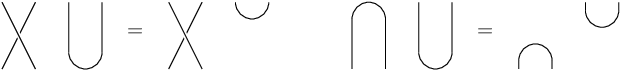}
    \caption{Interchanging the order of two singular points (without input disks)}
    \label{fig:exchagesingularity1}
\end{figure}

\begin{figure}
    \centering
    \includegraphics[width=1\linewidth]{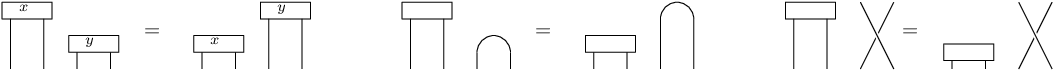}
    \caption{Interchanging the order of two singular points (with input disks)}
    \label{fig:exchagesingularity2}
\end{figure}
\item  Birth or annihilation of a pair of local extrema, see Figure \ref{fig:pa_zigzag},
\begin{figure}
    \centering
    \includegraphics[width=0.4\linewidth]{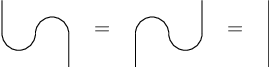}
    \caption{Zigzag relation}
    \label{fig:pa_zigzag}
\end{figure}
\item  Moving the braiding singularity along caps and cups, see Figure \ref{fig:braidingcapcup},
\begin{figure}
    \centering
    \includegraphics[width=0.7\linewidth]{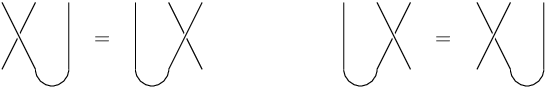}
    \caption{Moving the braiding singularity along cups}
    \label{fig:braidingcapcup}
\end{figure}
\item  $2\pi$-rotation invariance, see Figure \ref{fig:2pi_rot_inv}. 
\begin{figure}
        \centering
        \includegraphics[width=0.6\linewidth]{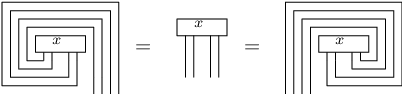}
        \caption{$2\pi$-rotation invariance.}
        \label{fig:2pi_rot_inv}
    \end{figure}
\end{itemize}
They will give the relations among generators to ensure the assigned maps are well-defined.

%\section{Orbifold theory}
%Now we consider one important examples coming from CFT %\cite{LX19}\cite{LMP20}\cite{JL17}\cite{BS11}
%Let $\CC$ be a modular category, consider the category $\DD:=\CC^{\bot n}$,
%We will mainly focus on the case when $n=2$. Now the natural $\mathbb{Z}_2$ action by permuting two factors gives an element in $\Breq(\BB)$, the corresponding module category is shown to be $\CC$ as a category with the various equivalent module structures, see \cite{BFRS10}, and the corresponding algebra is given by $\gamma=\bigoplus_X X\bot \bar{X}$, from \cite{KR08}, we also have $\gamma=M_1\otimes \bar{M}_1$.

\section{Lifting the shaded planar algebra}
In this section we will discuss the lifting problem of a shaded planar algebras to unshaded ones, the goal is to give a sufficient condition in terms of certain braiding structures on the even part of the shaded planar algebra. 
\begin{defn}\cite[Def.~1.3]{LMP20}\label{def_symmtric_self_dual}
A shaded planar algebra $(\PP_+,\PP_-)$ is called \textbf{self-dual}, if there are isomorphisms $\Phi_{\pm}$ ($(\Phi_{+},\Phi_{-}):(\PP_+,\PP_-)\to (\PP_-,\PP_+)$), satisfying
shaded planar algebra isomorphism: natural with respect to the tangle action (with the color flipped). Moreover the shaded planar algebra is \textbf{symmetrically self-dual} if $\Phi_{\mp}\circ\Phi_{\pm}=1_{\PP_{\pm}}$.
\end{defn}
In \cite{LMP20}, the following theorem is proved.
\begin{Thm}\cite[Thm.~ A]{LMP20}
Given a symmetrically self-dual shaded planar algebra $(\PP,\Phi_{\pm})$, it can be lifted to an unshaded planar algebra.
\end{Thm}
Next, we will define certain shaded $\Zn2$ braiding structures on the even part of shaded planar algebras.

\begin{defn}\label{def:Zn2_braiding_struc}
 The shaded $\Zn2$ braiding structure is defined to be a sequence of tangles.
 
 \includegraphics[width=350pt]{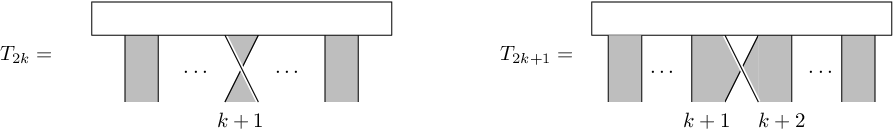}, 
 
 which satisfies the following relations (including all possible braidings and ways of shading).
 \[
 \includegraphics[width=400pt]{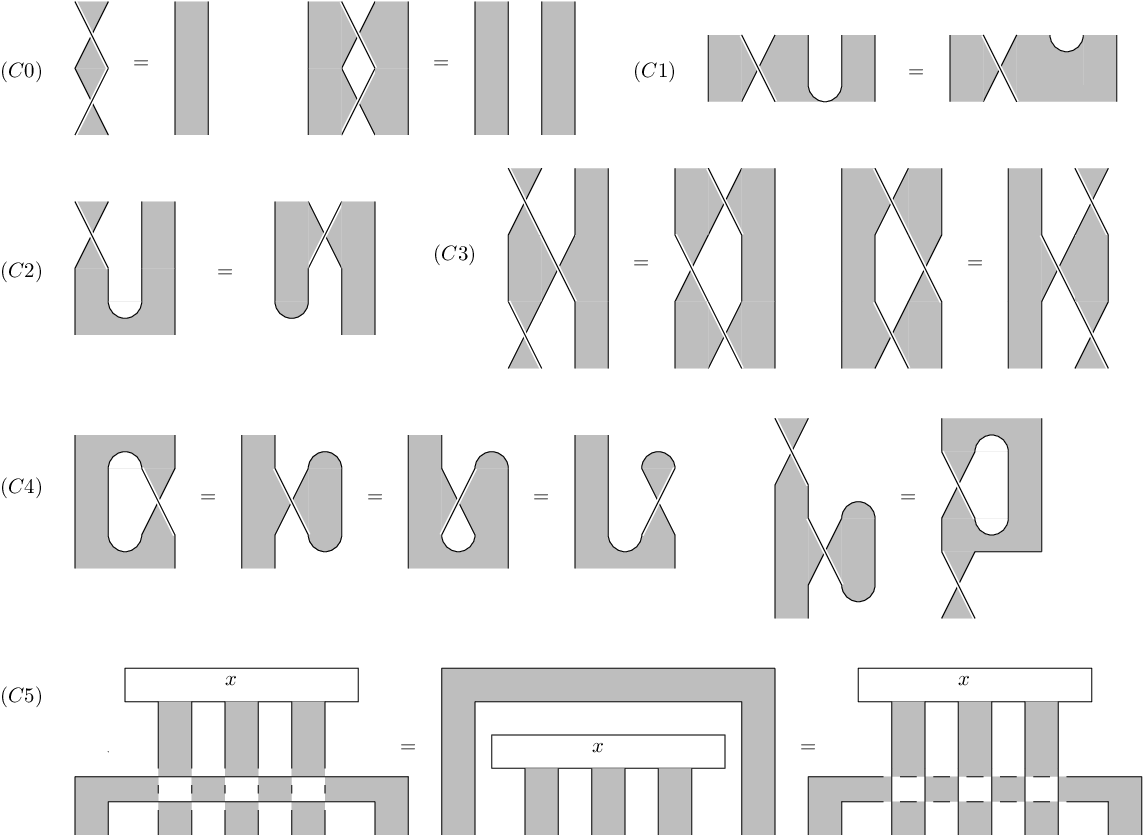}
 \]
\end{defn}

\begin{rmk}

\begin{itemize} 
    \item $(C1)$ indicates the compatibility of interchanging the order of braiding singularity and caps or cups with respect to the height function. For interchanging singularities of other types are inherited in the definition of shaded planar algebra. 
    \item $(C4)$ is the analogy of twists in the shaded version. We omit several relations there for simplicity: the similar equalities hold with the twist in the first strand, a positive twist and a negative twist in the same strand can be canceled. 
    \item $(C5)$ is well-defined relation once $(C2)$ is satisfied, since one can move the braiding singularities freely along the cap and cups. One can compare $(C5)$ with the symmetrically self-duality (Definition \ref{def_symmtric_self_dual}).
    \item All the strands can be labeled here, then the relations with all possible labels should be satisfied.
\end{itemize}
   
\end{rmk}

\begin{lem}\label{lem:Z_2_action}
The shaded $\Zn2$ braiding induces a $\Zn2$ action $\rho$ on $P_{n,+}$. In particular, $\rho(x)\otimes\rho(y)=\rho(x\otimes y)$ and $\rho(x)\rho(y)=\rho(xy)$.  
\end{lem}

\begin{proof}
   The action $\rho$ is given by the following tangle, the equality follows from $(C3)$. Where $c>0$ is the double loop value of the original shaded planar algebra \includegraphics[width=50pt]{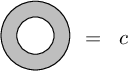}, which depends on the label of the loops.
\[
\includegraphics[width=350pt]{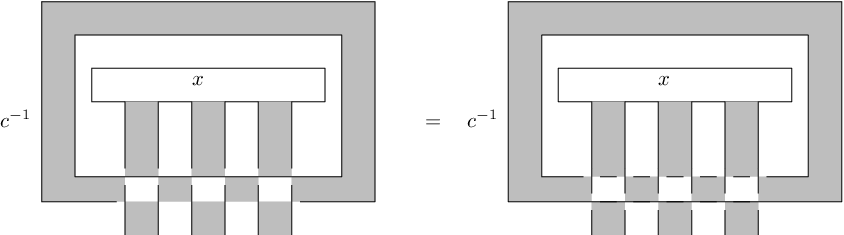}
\]

from $(C3)$ it is a planar algebra action (commute with all the planar tangle actions) and respects the braiding. Together with $(C5)$, one has it is a $\Zn2$ action.   
\end{proof}
\begin{Thm}\label{thm:lifting_thm}
If there exists a shaded $\Zn2$ braiding structure on the even part of a shaded planar algebra, then it can be lifted to an unshaded planar algebra together with a $\Zn2$ braiding.    
\end{Thm}
\begin{proof}

Now we define a unshaded planar algebra $\PP^{u}$ , the idea is similar to the \cite[Def.~1.6,1.7]{LMP20}:
\begin{itemize}
    \item $P^{u}_{2n}:=P_{n},\ P^u_{2n+1}:=\emptyset$
    \item Given an unshaded tangle $\mathbb{U}$, we give it checkerboard shading such that the region meeting the distinguished interval of the output disk is unshaded. Now if the distinguished interval of an input disk is in a shaded region, then we put a circle above or below as follows, scaled by $c^{-1/2}$.

    \includegraphics[width=350pt]{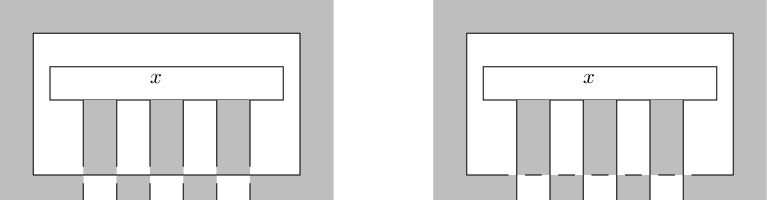}
\end{itemize}
Next, we prove this is indeed a planar algebra with $\Zn2$-braiding. The isotopy invariance follows from $1.$ The planar isotopy invariance of the shaded ones, $2.$ The isotopy involving braidings away from input disks: $(C0)-(C4)$ and $3.$ The isotopy through input disks: Corollary \ref{Cor:Z_2_action}. The proof of $Z(U\circ_{i} V)=Z(U)\circ_{i} Z(V)$ follows similarly to the proof of \cite[Thm~A]{LMP20}, one just replaces their $\Phi_{\pm}$ operator by our attaching circle operation, and the symmetrically self-duality is replaced by $(C5)$. 
\end{proof}
From now on, we fix a choice when constructing the unshaded planar algebra by attaching a circle from above. We will also denote the $n$-box space of the unshaded planar algebra $\PP^u$ by $P^u_n$ instead of $P^u_{2n}$ for simplicity. 
\begin{Cor}\label{Cor:Z_2_action}
In the unshaded planar algebra $\PP^u$, we have the following identities.
\[
\includegraphics[width=250pt]{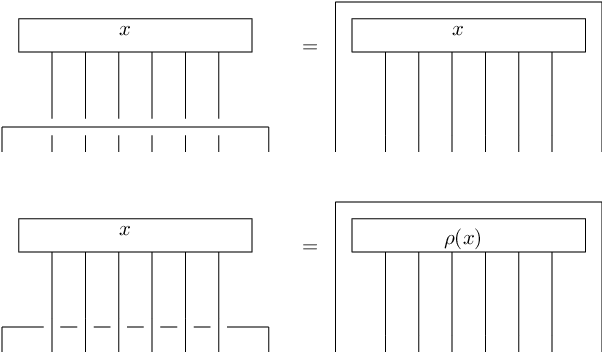}
\]

\end{Cor}
\begin{proof}
There are two cases depending on the shading of the input disk labeled by $x$ on the left-hand side, if it is shaded, then we attach a circle by definition, apply $(C5)$ and evaluate the unshaded disk (for the second equality, one need to use the diagrammatic description of the action in Lemma \ref{lem:Z_2_action}). If it is unshaded, then the input disk on the right-hand side is shaded, again we attach a circle, apply $(C5)$ and also Lemma \ref{lem:Z_2_action}. 
\end{proof}

%\section{CFT perspective}

%\section{BrPic(C) and Breq(Z(C))}

%\section{exact sequce}

\section{Graphic calculus in the configuration space}
In this section, we will define new configuration spaces inspired by those defined in \cite[Sec.~2]{LX19}, and several operations on them. 

Let $\CC$ be a unitary modular category, we simply use $ONB$ to denote an orthonormal basis in $\Hom$ spaces. Consider the category $\CC^{\bot 2}$, and the Frobnius algebra $\gamma=\bigoplus_X X\bot \bar{X}$. In \cite{LX19}, the authors identify the space $\Hom_{\CC^{\bot 2}}(1, \gamma^{\otimes n})$ with the configuration space $Conf(\CC)_{n,2}$, which is a Hilbert space and the orthonormal basis is given.

Now we consider the generalization to the case involving objects of the form (not symmetric) $Z_{2i-2}\boxtimes Z_{2i-1}$. Let $\ZZ$ denote the sequence $(Z_{k})_{0\leq k\leq 2n-1}$, We define the configuration space  $Cf(n,\ZZ):=\Hom_{\CC^{\boxtimes2}}(1,\otimes^n_{i=1}\gamma Z_{2i-2}\boxtimes Z_{2i-1})$, when all $Z_i's$ are trivial, we denote corresponding configuration space simply by $Cf(n)$. The configuration space is now spanned by the following vectors. (The blue vertical lines are labeled by the simple objects in $\CC$ similar as \cite[Sec.~2.2]{LX19})
\[
\includegraphics[width=350pt]{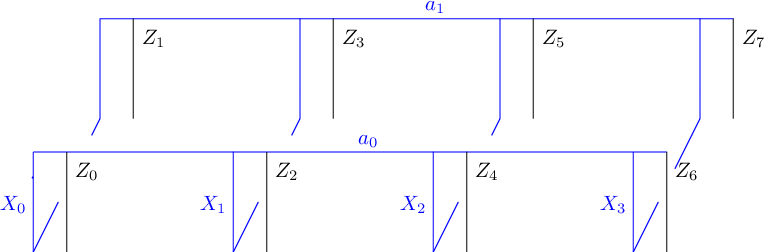}
\]
Since we only have two layers, we will simply denote $\overrightarrow{X}:=(X_k)_{0\leq k\leq n-1 }$ and we define $d_{\overrightarrow{X}}=\prod^{n-1}_{k=0}d_{X_k}$.
%We define several operations on the ordered set $\ZZ$, let $F(Z_i)=z_{2i}\boxtimes z_{2i+1}$ with $F(Z_n)=z_{2n}\boxtimes z_1$ and $\Theta_2{Z_i}=\bar{z}_{2i+1}\boxtimes \bar{z}_{2i}$. Now we define the action of $\Theta_1,\Theta_2,\rho_1,\rho_2$ on $\ZZ$ by the induced action from them on each component $Z_i$. 

%Fix $n,Z$, in the rest of paper we will focus mainly on the following spaces
%\[
%\begin{aligned}
%Cf_+(n,\ZZ)=&\Hom_{\CC^{\boxtimes2}}(1,\otimes^n_{i=1}\gamma Z_i),\\ 
%Cf_-(n,\ZZ)=&\Hom_{\CC^{\bt2}}(1,\otimes^n_{i=1}\gamma \tilde{Z_i}),\\
%Cf_{\pm}^1(n,\ZZ)=&Cf_{\pm}(n,\Theta_1(\ZZ)),\\
%Cf_{\pm}^2(n,\ZZ)=&Cf_{\pm}(n,\Theta_2(\ZZ)).
%\end{aligned}
%\]

Now we define some operations $\rho_1,\ \Theta_2$ in $\CC$ which is similar as in \cite{LX19}, $a^*$ is the corresponding morphism in the dual vector space. In the unitary case, it will simply be $a^{\dagger}$. 

%There are some operations among those spaces:
%\begin{itemize}
%\item $\Theta_1:Cf(n,\ZZ)\leftrightarrow Cf(n,\Theta_1(\ZZ))$, 
%\item $\Theta_2:Cf(n,\ZZ)\leftrightarrow Cf(n,\Theta_2(\ZZ))$, 
%\item $\rho_1: Cf(n,\ZZ)\to Cf(n,\rho_1(\ZZ)) $,
%\item $\rho_2: Cf(n,\ZZ)\to Cf(n,\rho_2(\ZZ))$.  
%\end{itemize}

\[
\includegraphics[width=350pt]{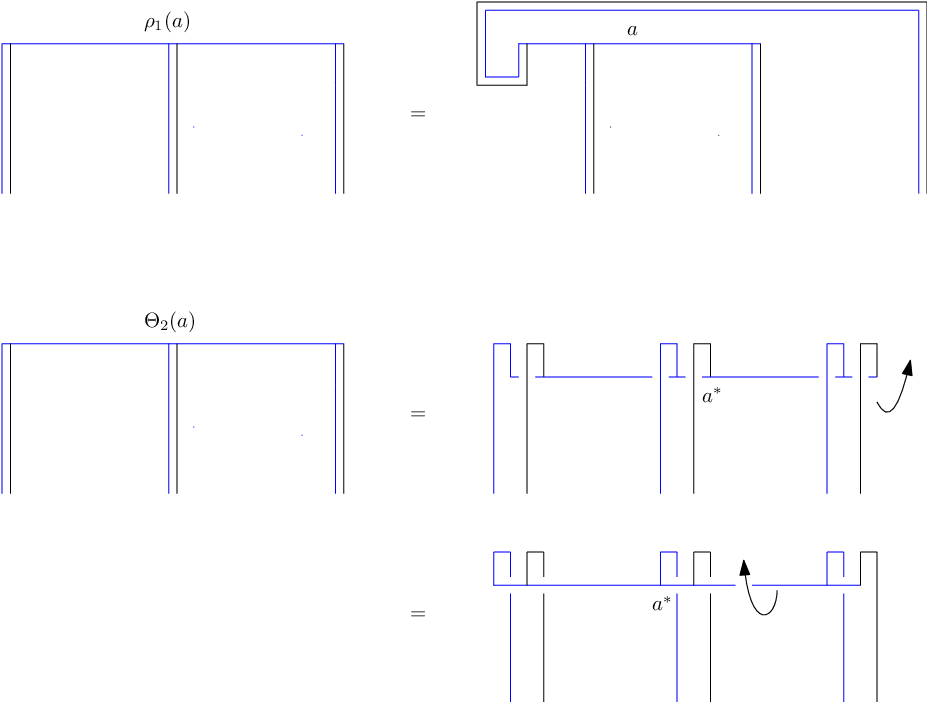}
\]
One can view the $\Theta_2$ action as bending the morphisms along the indicated directions.  
The next lemma will be used frequently and follows directly from the definition of $\Theta_2$ and the graphic calculus in $\CC$.
\begin{lem}\label{lem:Theta_2_lemma}
 The following morphisms in $\CC$ are equal (the red strand is labeded by $\Omega$)
 \[
 \includegraphics[width=400pt]{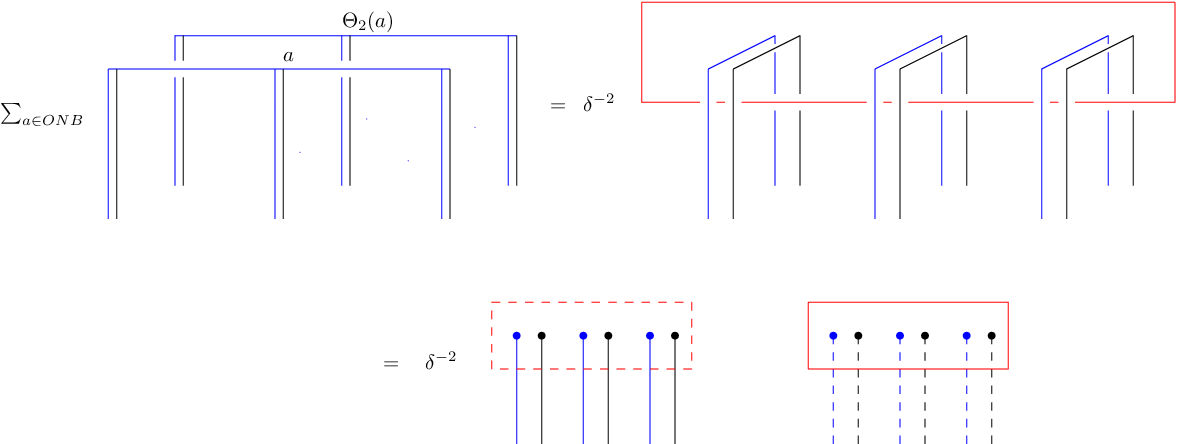}
 \]
\end{lem}
Here we also introduce the new diagrammatic notation for simplicity, the $3D$ diagram will now be interpreted by a pair of $2D$ diagrams, one is for the front layer and the other one is for the back layer. the small dot on the boundary of an edge indicates it connects with the other layer through $Y$-direction, the dashed strands are underneath other strands. We will often draw the red circle for both layers to indicate its position.
\begin{rmk}\label{rmk:graphic_cal_exp}
The inner product $<a,a'>$ for $a,a'$ in the configuration space is the obvious one given by the evaluation in $\CC\boxtimes\CC$ of $a$ composed with the vertical reflection of $a'$ ($a'^{\ \dagger}$), which equals to the evaluation in $\CC$ (view them as morphisms in $\CC$ in an obvious way indicated by our drawing). Therefore in most cases we can use Lemma \ref{lem:Theta_2_lemma} to simplify our calculations (though it is not a valid move in $\CC\boxtimes\CC$). We will simply denote an orthonormal basis in configuration space by $Cf_{ONB}$.   
\end{rmk}
The operations act among the configuration spaces as follows ($\Theta_2$ is antilinear, similarly as in \cite{LX19}).
\[
\includegraphics[width=350pt]{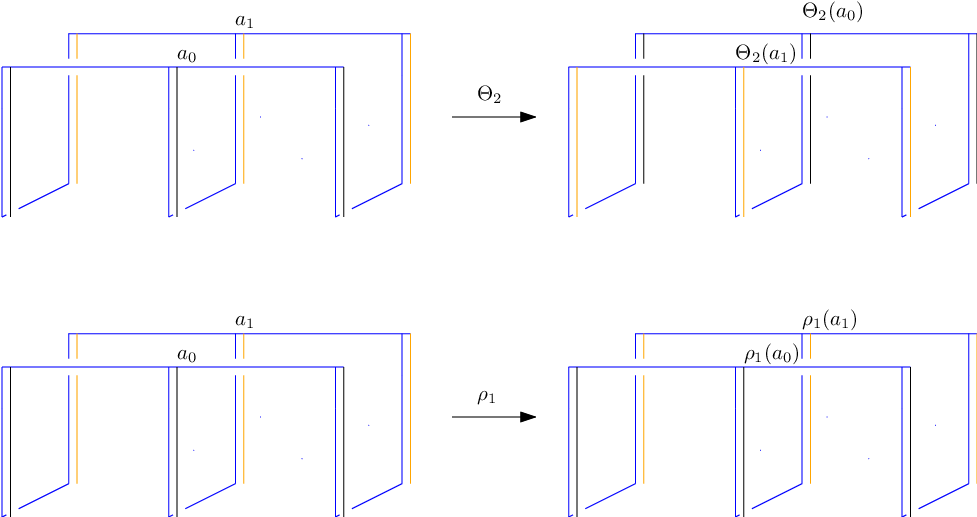}
\]
And we will simply use $Cf(n,\Theta_2(\ZZ)),\ Cf(n,\rho_1(\ZZ))$ to denote corresponding target spaces. 

Next we define a bilinear form generalizing pairing defined in \cite[Sec.~2.3]{LX19}, Let $F(\ZZ)=(Z_k)_{1\leq k\leq 2n}$ with $Z_{2n}:=Z_{0}$.

%The most important one we want is the Fourier transform $\FF$ from $Cf_+$ to $Cf_-$, in order to define $\FF$, we define a pairing between $Cf_+(n,\ZZ)$ and $Cf^2_-(n,\ZZ)$.

\begin{defn}
Now for $a\in Cf(n,\ZZ)$ and $a'\in Cf(n,F(\ZZ))$, we define the value of the pairing $\LL(a,\Theta_2(a'))$ to be the value of the diagram \ref{fig:Fourier_pairing} in $\CC$ (where $a$ is blue and $a'$ is puple), the coefficient is equal to $\delta^{1-n}\sqrt{d_{\overrightarrow{X}}}\sqrt{d_{\overrightarrow{Y}}}$. 
\end{defn}

\begin{figure}
    \centering
    \includegraphics[width=350pt]{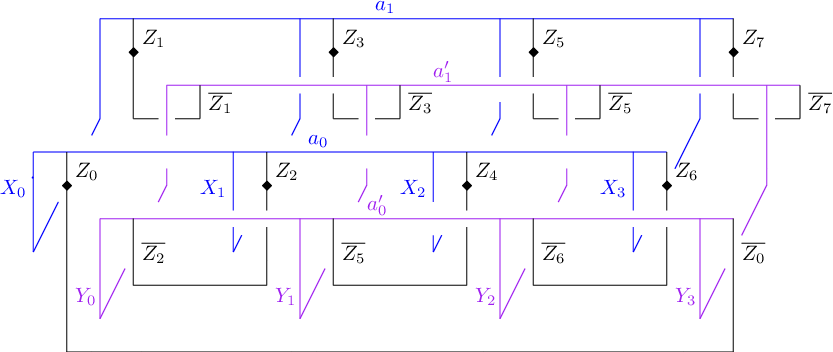}
    \caption{Fourier pairing}
    \label{fig:Fourier_pairing}
\end{figure}

The Fourier transform $\FF: Cf(n,\ZZ)\to Cf(n,F(\ZZ))$ is now defined by
\[
\FF(a)=\sum_{a'\in Cf_{ONB}}\LL(a,\Theta_2(a'))a'.
\]

\begin{exmp}
 the Fourier pairing on $Cf(2)$ with the canonical basis is the same as the $S$-matrix of $\CC$ \cite[Thm.~6.8]{LX19}.   
\end{exmp}

\begin{exmp}
The Fourier pairing between $Cf(2,\ZZ)$ and $Cf(2,F(\ZZ))$, where $\ZZ=(Z,1,1,1)$, is given by the following,

\[
\includegraphics[width=100pt]{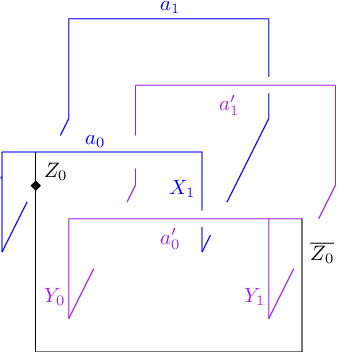}
\]

which can be viewed as a $S$-matrix of once punctured torus.
\end{exmp}

\subsection{Inclusion and contractions }

In this subsection, we will define certain operations on the configuration spaces, and study their compatibility with the Fourier transform. 

We first define the contractions and inclusions. 

\begin{defn}\label{def:contaction_and_inclusion}
 The contraction and inclusion $\phi_k: Cf(\CC)_{n,\ZZ}\to Cf(\CC)_{n-1,\phi_k(\ZZ)}$, $\iota_k:Cf(\CC)_{n,\ZZ}\to Cf(\CC)_{n+1,\iota_k(\ZZ)}$ ($\phi_k(\ZZ)$ and $\iota_k(\ZZ)$ are defined accordingly) are defined by stacking the following diagrams on the bottom. The sum is for all $0\leq i\leq n-1$, $X_i,Y,Y',Z\in \Irr(\CC)$ and all orthonormal basis $\alpha$ in corresponding $\Hom$ spaces. The global coefficient is $\delta^{1/2}$ for $\phi_{2k},\iota_{2k}$, $\delta^{-1/2}$ for $\phi_{2k+1}, \iota_{2k+1}$.

 \includegraphics[width=460pt]{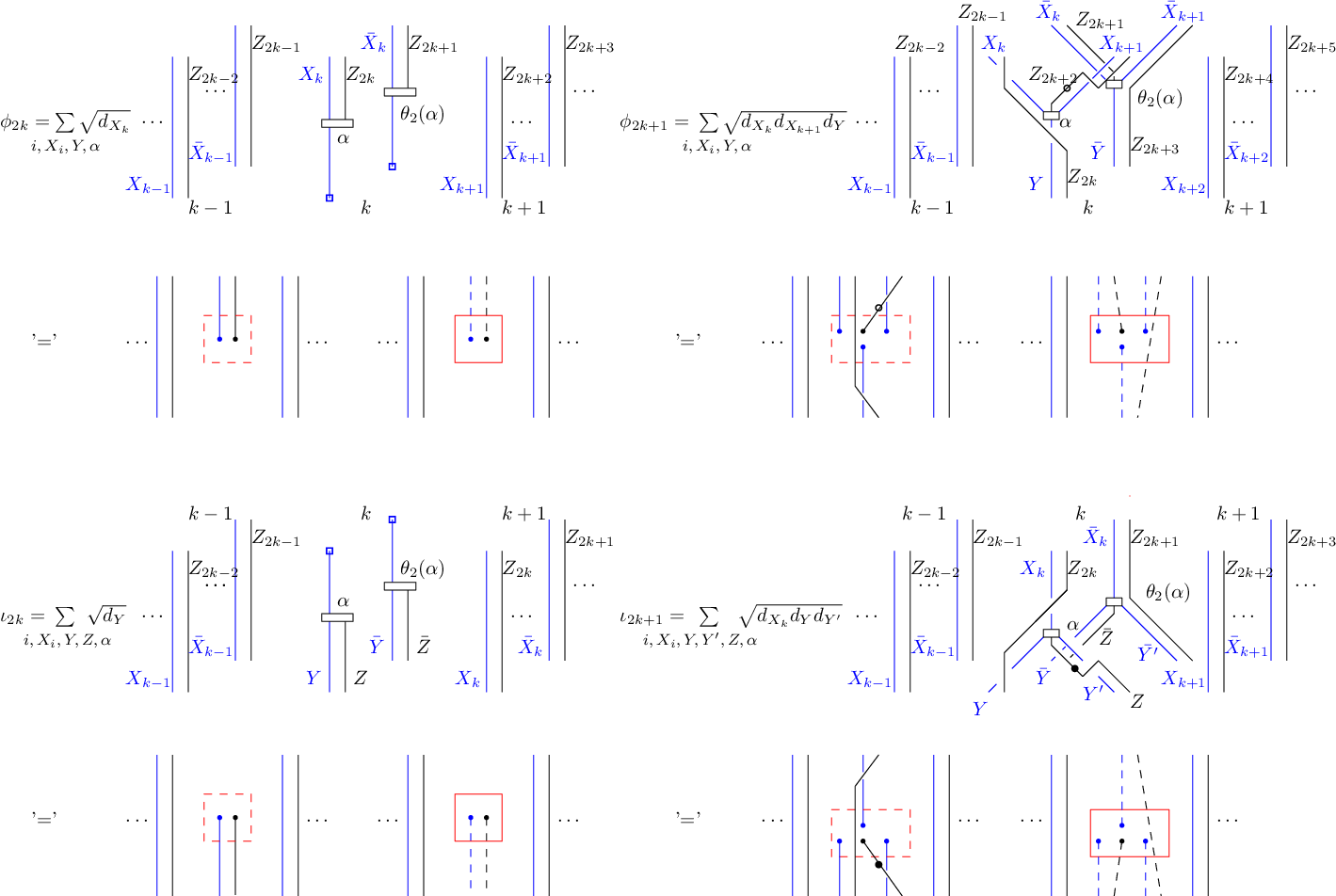}

\end{defn}

We also draw the simplified diagram using the newly introduced notation, see Lemma \ref{lem:Theta_2_lemma}. The new diagrammatic notation represents the morphisms in $\CC$, while the defined operations are indeed morphisms in $\CC\boxtimes\CC$. However, we consider these as equivalent by Remark \ref{rmk:graphic_cal_exp}.

Now we examine some properties of $\phi_k, \iota_k$.

First of all one verify $\phi_k,\iota_k$ satisfy "zig-zag" relations:
\begin{lem}\label{lem:zig-zag}
The following identities hold 
\[
\begin{aligned}
\phi_{2k-1}\circ(\Id\otimes \iota_{2k})&=(\Id\otimes \phi_{2k})\circ \iota_{2k-1}&=\Id\\
\phi_{2k-1}\circ(\iota_{2k-2}\otimes \Id )&=(\phi_{2k-2}\otimes \Id)\circ \iota_{2k-1}&=\Id
\end{aligned}
\]
\end{lem}
\begin{proof}
We will prove the first identity, and the second one follows similarly. The proof is given by the graphic calculus below and we omit some of the labels and summations.
\[
\includegraphics[width=425pt]{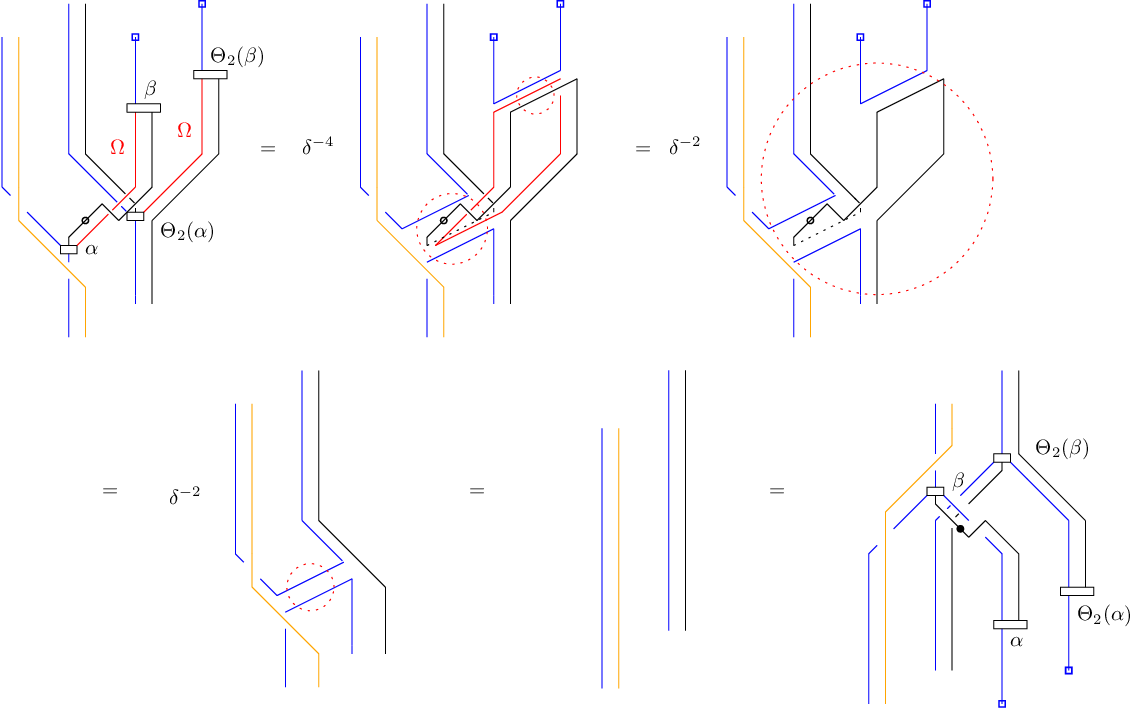}
\] 
The square root of quantum dimension coefficients for $\phi_k,\ \iota_k$ make the pair of blue lines connecting two boxes in the first diagram become a red circle (Kirby color $\Omega$), combined with Lemma \ref{lem:Theta_2_lemma}, the first equality now follows, the next ones follow from using isotopy and the cutting property of the Kirby color.   
\end{proof}

\begin{rmk}\label{rem:Kirby_color_creat}
 As in the proof of Lemma \ref{lem:zig-zag}, apart from the Lemma \ref{lem:Theta_2_lemma}, the step of changing a pair of blue-colored lines (with suitable coefficients) connecting two boxes into a red-colored circle will also be used frequently.    
\end{rmk}

\begin{lem}\label{lem:contraction_inclusion_equality}
Let $\hat{Z}:=\oplus_{Z\in \Irr(\CC)}Z$ and $\ZZ$ be the sequence whose elements are all $\hat{Z}$ , we have for $x\in Cf(n,\ZZ)$ and $y\in Cf(n-1,\ZZ)$
\begin{equation}
\begin{aligned}
\phi_k\iota_k&=\delta d_{\hat{Z}}, \\
\LL(\phi_{2k}x,\Theta_2(y))&=\LL(x,\Theta_2\iota_{2k-1}(y)),  \\
\LL(\phi_{2k+1}(x),\Theta_2(y))&=\LL(x,\Theta_2 \iota_{2k}(y)).
\end{aligned}
\end{equation}
\end{lem}
\begin{proof}
The first identity follows from direct graphic calculus. For other identities, one first observes the action of $\Theta_2$ will change the braiding and the way of connections. Now the proof follows from the following graphic calculus: we use again Lemma \ref{lem:Theta_2_lemma}, the summation is replaced by several $\Omega$-colored circles, now applying the cutting property and the equalities can be seen from isotopy of the orange strands. 

\[
\includegraphics[width=300pt]{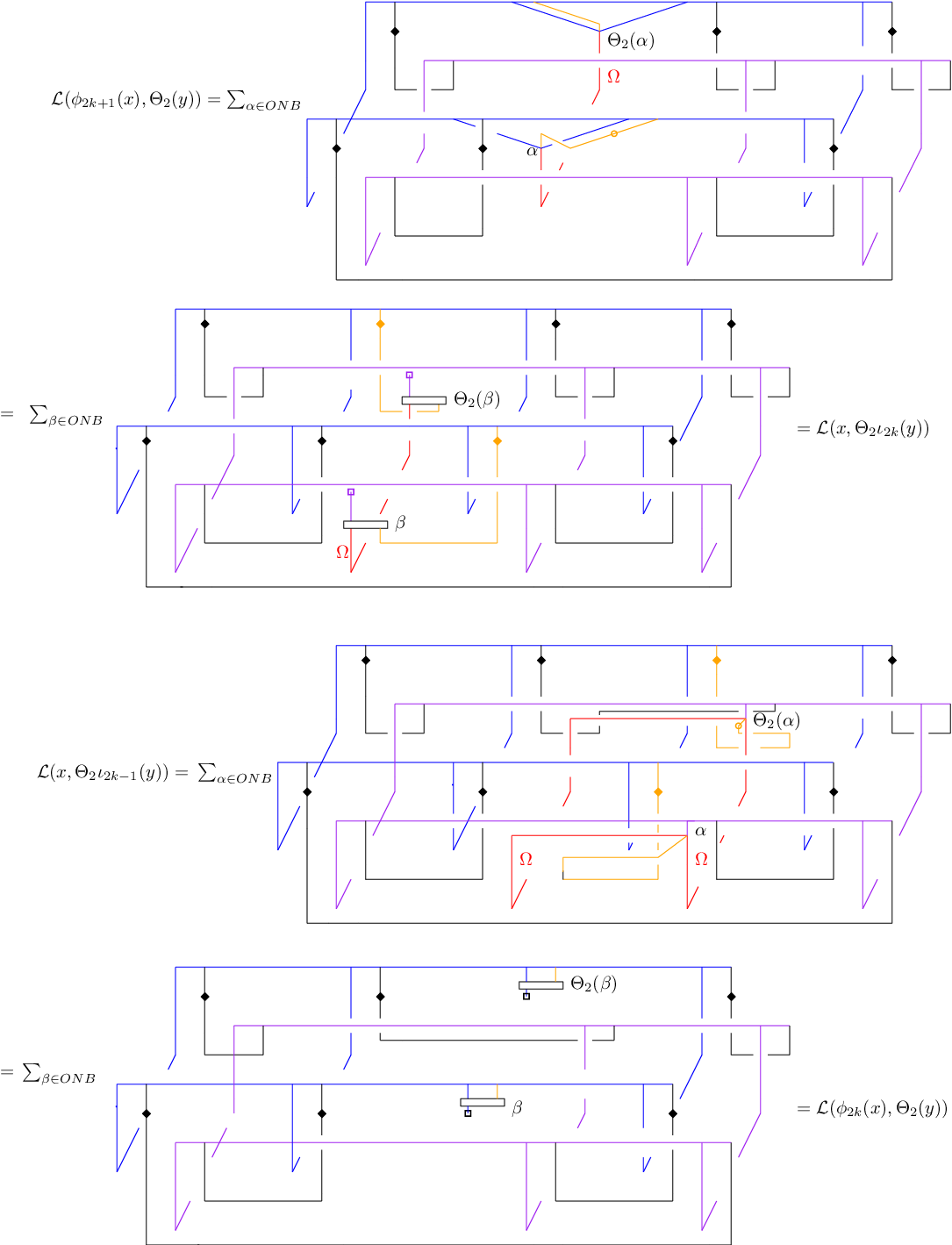}
\]

\end{proof}

From Lemma \ref{lem:contraction_inclusion_equality}, we have the following compatibility with respect to the Fourier transform.
\begin{prop}
We have the following identity
\[
\FF(\phi_{k+1}(x))=\phi_{k}\FF(x).
\] 
\end{prop}
\begin{proof}
 \[
 \begin{aligned}
  \FF(\phi_{k+1}(x))&=\sum_y\LL(\phi_{k+1}x,\Theta_2(y))y\\
&=\sum_y\LL(x,\Theta_2\iota_{k}(y))y\\
&=\sum_{(\delta d_{\hat{Z}})^{-1/2}\iota_k y}\LL(x,\Theta_2 (\delta d_{\hat{Z}})^{-1/2}\iota_{k}(y))\phi_k(\delta d_{\hat{Z}})^{-1/2}\iota_ky \\
&=\phi_{k}\FF(x)
 \end{aligned}
 \]   
\end{proof}

%\begin{rmk}
% The tensor product among configuration spaces is defined naturally by simply placing vectors next to each other.   
%\end{rmk}

%\begin{defn}
%    the tensor product of $P_{\pm,n}$ is given by usual one, for the graded product, we define for $x,y\in P{+,n},P_{+,m}$, $x*y\in P_{+,n+m}$ is given by the following diagram. 
%\[
%\includegraphics[width=300pt]{Graded_product.eps}
%\]   
%\end{defn}
%It is not hard to show the compatibility of this two products with respect to the Fourier transform.
%\begin{lem}
% For $z\in Cf^2_-(n+m)$, and $x'\otimes y'$ denotes the image of projection to the $Cf^2_-(n)\otimes Cf^2_-(m)$, we have the following identity.  
%\[
%\LL(x*y,z)=\LL(x,x')\LL(y,y')
%\]

%\end{lem}

%In particular, we have 
%\[
%\FF(x*y)=\FF(x)\otimes\FF(y)
%\]

\subsection{Braiding Structure}
in this section, we will define the braiding structures using the data of $\CC$. 

% notation for the positive (right) twist---------------------

%\[ \Theta_X \operatorname{Id}_X=
%\raisebox{-0.5cm}{
%    \begin{tikzpicture}[scale=0.5]
%    \draw [black] (0, 1) node [above] {$X$} .. controls +(0, -1) and +(0, -1)..(0.5, 0); 
%    \fill[white] (0.15,0) circle (4pt);
%   \draw [black] (0.5,0) .. controls +(0, 1) and +(0,1) ..(0, -1);
%  \end{tikzpicture}}
    
%\raisebox{-0.5cm}{
%    \begin{tikzpicture}[scale=0.5]
%    \draw [black] (0, -1) .. controls +(0, 1) and +(0, 1)..(-0.5, 0); 
%    \fill[white] (-0.15,0) circle (4pt);
%    \draw [black] (-0.5,0) .. controls +(0, -1) and +(0,-1) ..(0, 1) node [above] {$X$};
%    \end{tikzpicture}}
%    =
%\raisebox{-0.5cm}{
%    \begin{tikzpicture}[scale=0.5]
%    \draw [black] (0, 1) node [above] {$X$} -- (0, -1); 
%    \fill[red] (-0,0) circle (4pt);
%    \end{tikzpicture}}
%\]

% Notation for negative (left) twist-----------------------------------
%Similarly for the left twist
%\[
%\Theta^{-1}_X \operatorname{Id}_X=
%\raisebox{-0.5cm}{
%    \begin{tikzpicture}[scale=0.5]
%    \draw [black] (0, 1) node [above] {$X$} -- (0, -1); 
%    \fill[black] (-0,0) circle (4pt);
%    \end{tikzpicture}}
%\]  
\begin{defn}\label{def:braiding}
    We define the braiding morphism as follows, where $T_{2k+1}=(\phi_{2k+1}\otimes \Id)( \Id\otimes T^{-1}_{2k+2} \otimes \Id )( \Id\otimes \iota_{2k+3})$ and the coefficient for $T_{2k+1}$ is $\delta^{-1}$ ($\delta^{-1/2}$ each for $\phi_{2k+1}$ and $\phi_{2k+3}$). 
\[
\includegraphics[width=450pt]{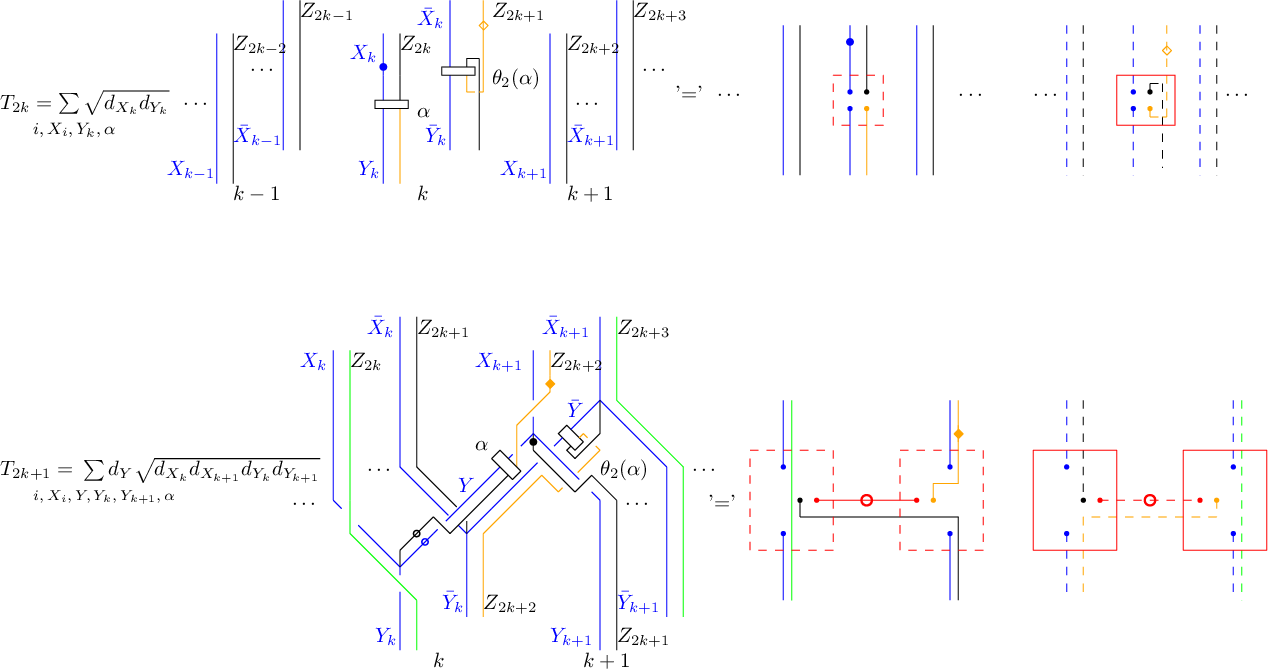}
\]
\end{defn}

Now it is straightforward to prove operators $T_{k}$ are unitary with $T_k^{\dagger}=T_k^{-1}$ given by the vertical reflection with corresponding dagger operation on morphisms.
\begin{lem}\label{lem:braiding&Fourier_comp}
 We have the following identities,
 \[
 \FF(T_{2k+1}(x))=\eta T_{2k}\FF(x)
 \]
\end{lem}
\begin{proof}
   We will first prove 
\begin{equation}\label{eq:braiding_Fourier_comp}
\LL(T_{2k+1}x,\Theta_2(y))=\eta\LL(x,\Theta_2 T^{-1}_{2k}(y)),    
\end{equation}
recall $\eta=\frac{p^+}{\delta}$ is a global constant.
By the definition of the $\Theta_2$, we have the following identities, the second equality follows from Lemma \ref{lem:Theta_2_lemma}, Remark \ref{rem:Kirby_color_creat} and using the twist property. 

\[
\includegraphics[width=400pt]{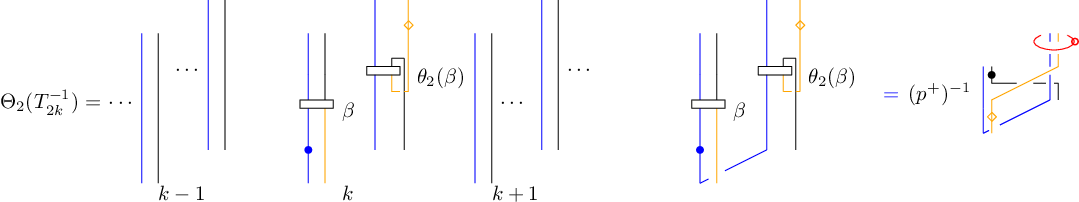}
\]

Now the identity (\ref{eq:braiding_Fourier_comp}) follows from transferring two diagrams below by the direct isotopy and cancellation of the twists. (Here we only draw the local diagrams for the pairing $\LL$, since we already did a similar proof in Lemma \ref{lem:contraction_inclusion_equality} by drawing full diagrams). 

\[
\includegraphics[width=400pt]{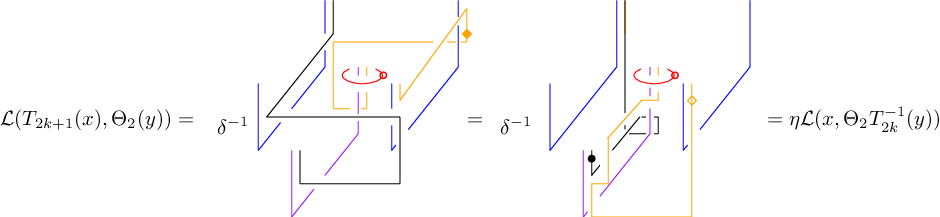}
\]

Now we have,
\[
\begin{aligned}
\FF(T_{2k+1}(x))=&\sum_y\LL(T_{2k+1}x,\Theta_2(y))y\\
=&\sum \eta\LL(x,\Theta_2 T^{-1}_{2k}(y))T_{2k}T^{-1}_{2k}(y)\\
=&\eta T_{2k}\FF(x)
\end{aligned}
\]
\end{proof}
Next we show the braiding is compatible with contractions and inclusions in the sense of the next lemma.
\begin{lem}\label{Lem:braiding_capcup_comp}
 We have the identities,
 \[
 \begin{aligned}
   \phi_{2k}\phi_{2k+1}\circ (T_{2k}\otimes \Id)=&
   \phi_{2k}\phi_{2k+1} \circ (\Id\otimes T_{2k+2}),\\
   (T_{2k}\otimes \Id)\circ \phi_{2k+1}\phi_{2k}=&
   (\Id \otimes T_{2k+2})\circ \phi_{2k}\phi_{2k+1}.
 \end{aligned}
 \]
\end{lem}
\begin{proof}
We only prove the first statement, the second one follows similarly
\[
\includegraphics[width=350pt]{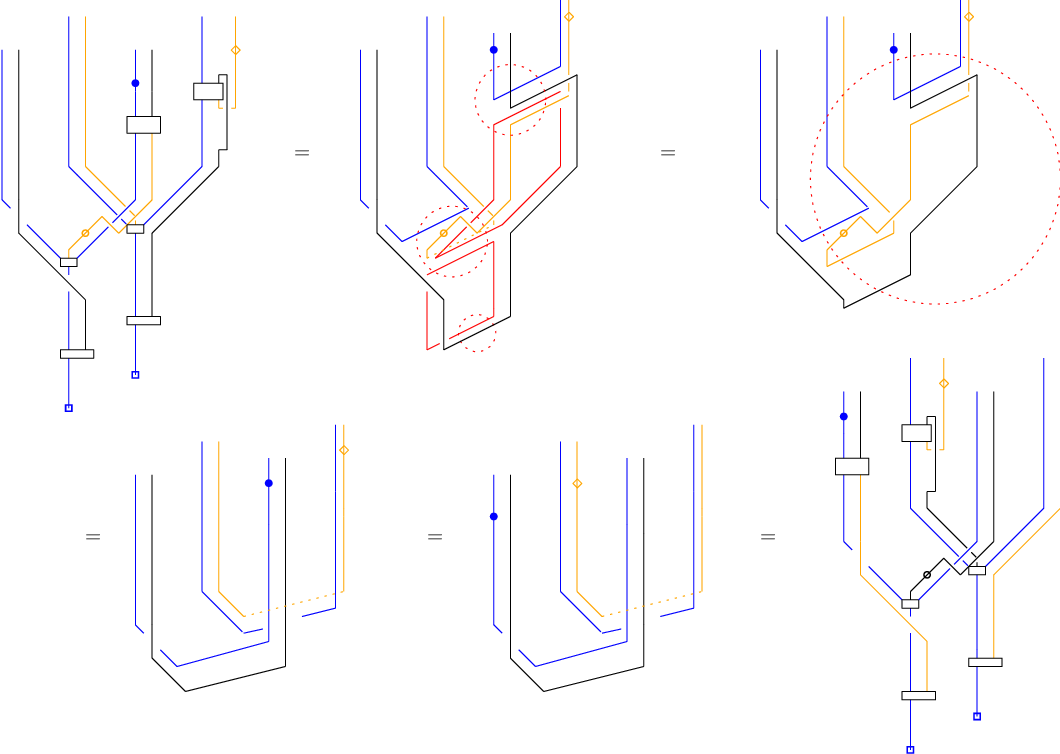}
\]
\end{proof}

Now we give another diagrammatic interpretation of the Fourier transform $\FF$. 
\begin{defn}
We define an operator $F_{+,Z}$ in configuration space from $Cf(n,\ZZ)$ to $Cf(n,F(\ZZ))$ given by the Figure \ref{fig:Fourier_transform} with a $\delta^{-1}$ factor, where $Z\in obj(\CC)$ is the color on the black strand. We denote $F_{+,1}$ by $F_{+}$, and we also denote the operator with braiding reversed by $F_{-,Z}\  (F_{-}:=F_{-,1})$.
\begin{figure}
    \centering
    \includegraphics[width=350pt]{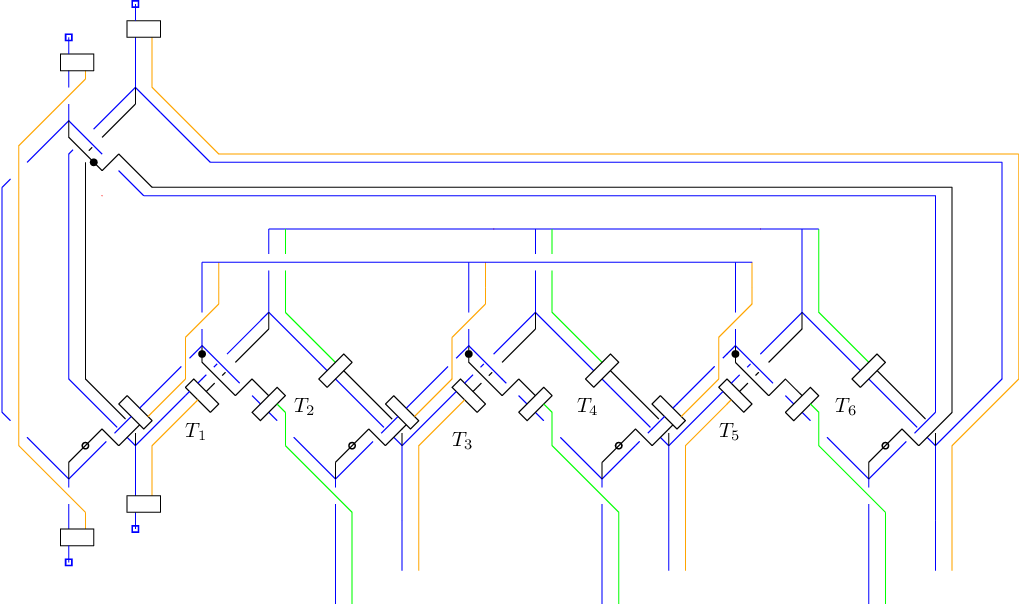}
    \caption{Fourier transform as an operator between configuration spaces}
    \label{fig:Fourier_transform}
\end{figure}  
\end{defn}

\begin{defn}\label{def:Z_2_action}
 The action $\rho_2$ on the configuration space is defined as follows. 

\[
\includegraphics[width=450pt]{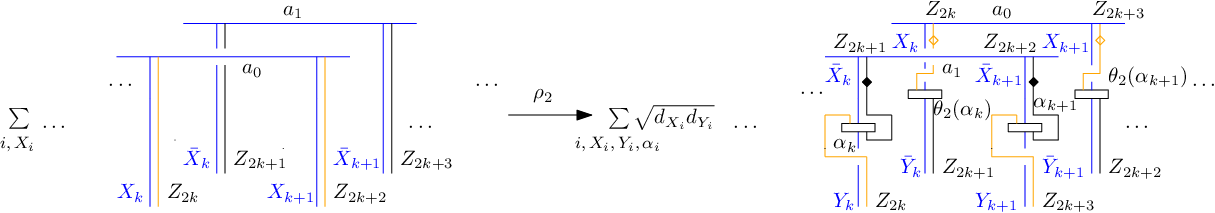}
\]   
\end{defn}
By direct graphic calculus, one gets the following lemma.

\begin{lem}
 $\rho_2$ is an isometry, $\rho_2^2=1$ and $\rho_2(x) \otimes \rho_2(y)=\rho_2(x\otimes y)$.   
\end{lem}

The next theorem is one of the main theorems that relates two definitions of the Fourier transform (Figure \ref{fig:Fourier_pairing} and \ref{fig:Fourier_transform}). 

\begin{Thm}\label{thm_flatness}
We have the following identities,
\[
\begin{aligned}
    <F_{+,Z}(x),y>&=d_Z\LL(x,\theta_2(y))\\
    <F_{-,Z}(x),y>&=d_Z\LL(\rho_2(x),\theta_2(y))\\
\end{aligned}
\]
\end{Thm}
\begin{proof}
The diagram \ref{fig:flatness1} is read as discussed before, they all come from the front views, while the left ones are first layers, and the right ones are second layers, the small dot on the boundary of an edge indicates it connects with the the other layer through $Y$-direction. surrounded by a Kirby color (dotted right circle). We emphasize that this simplification is purely notational; one could instead draw the 3D diagrams, which represent genuine morphisms in $\CC$ (see Lemma \ref{lem:Theta_2_lemma}), the graphical calculus used here is entirely derived from the graphical calculus of $\CC$. Here, unfilled (filled) diamond symbols are placed on each green (orange) line (from Definition \ref{def:braiding}), but they have been omitted for clarity in the initial steps.  

\begin{figure}
    \centering
    \includegraphics[width=400pt]{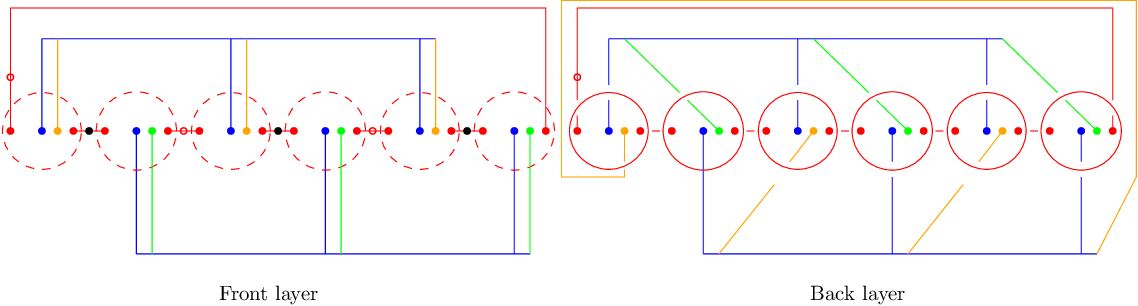}
    \caption{Two-layer presentation of $<F_{+,Z}(x),y>$}
    \label{fig:flatness1}
\end{figure}

First, observe there are two types of red circles, we call them vertical and horizontal red circles respectively, the number of each of them is $2n$,  and the twists are all on the horizontal ones. One uses definitions of the braiding and the black string is pulled out which gives the $d_{Z}$ factor. By Lemma \ref{lem:Theta_2_lemma} and the definition of the $\phi,\iota$, we have another $(\delta^{3/2-2-2})^{2n}\delta^{-1}=\delta^{-5n-1}$ as the coefficient.

Now we use the twist property to link horizontal red circles (Kirby color) with blue and orange ones, such that the twists on the red strings are all resolved. Next using the cutting property to reduce the number of both the vertical and horizontal red circles to $n$, which gives a $\delta^{2n}$ factor. By sphericality and the property of the configuration space (two layers are separated), one removes the rightmost vertical red circle, which gives a $\delta^2$ factor. The small horizontal black line segments connecting two vertical lines indicate the full left twists on two lines. These steps are illustrated in Figure \ref{fig:flatness2}.

\begin{figure}
    \centering
    \includegraphics[width=400pt]{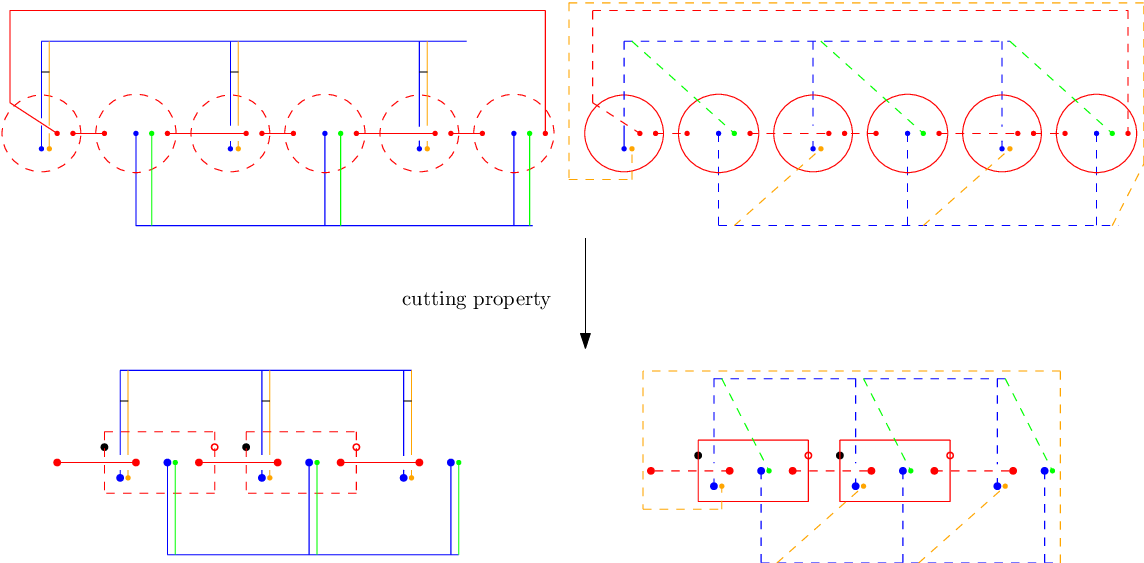}
    \caption{Graphic calculus}
    \label{fig:flatness2}
\end{figure}

Next one considers the strands that go through the horizontal red circles and applies the twist property again. Next using the handle slide property of the Kirby color, we link the $x,y$ together and all red circles are resolved, which introduces a factor of $\delta^{2n}$. And the coefficient now is equal to $d_Z\delta^{-5n-1+2n+2+2n}=d_Z\delta^{1-n}$ and the terms of the square roots of the object quantum dimensions appear as in the definition of the contraction and inclusion (Definition \ref{def:contaction_and_inclusion}), which is the same as $d_Z$ times the coefficient in the definition of $\LL$. These steps are illustrated in Figure \ref{fig:flatness3}.

\begin{figure}
    \centering
    \includegraphics[width=400pt]{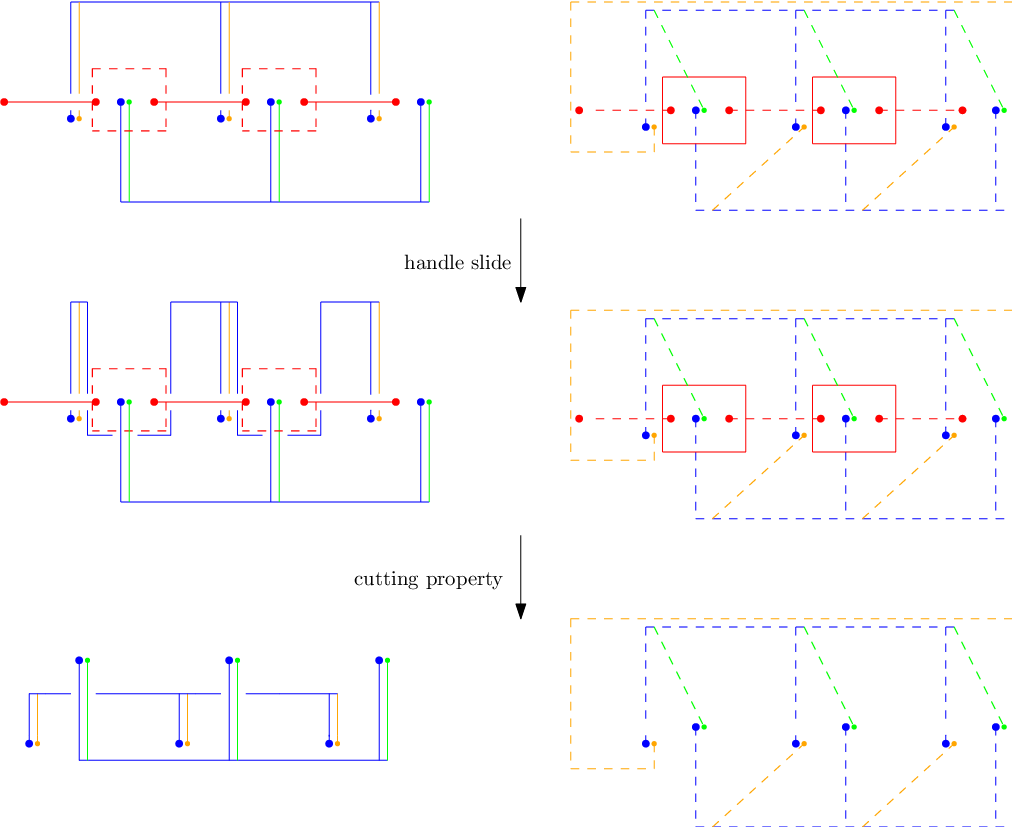}
    \caption{Graphic calculus}
    \label{fig:flatness3}
\end{figure}

Now we have the following 3D diagram and the equality follows after doing bending moves as described in the definition of $\Theta_2$. The rotations introduce the left twists to green strings and $\Theta_2$ action to $y$ (purple). 
\[
\includegraphics[width=400pt]{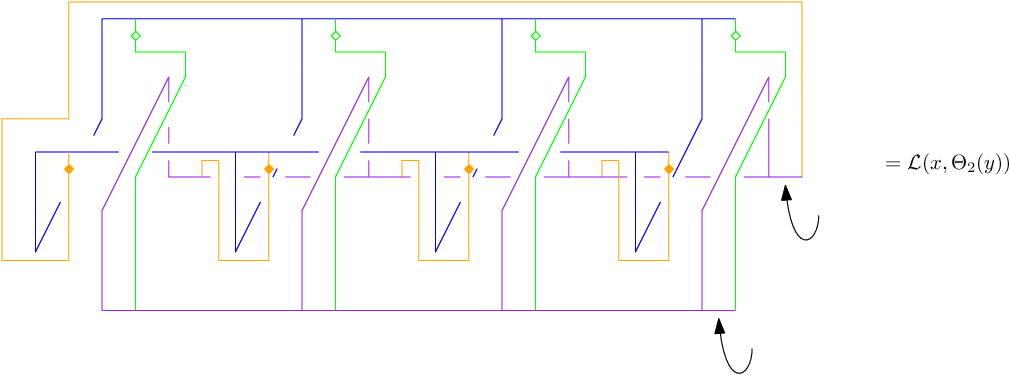}
\]
For the reverse braiding, one can imagine that we read diagrams "from the other side", then the twists change to the case of the previous situation and proof follows similarly except the position of two basis in the end gets changed. Now the rotation is here to transform it to the standard position. Here we only draw the last steps of the proof.

\[
\includegraphics[width=400pt]{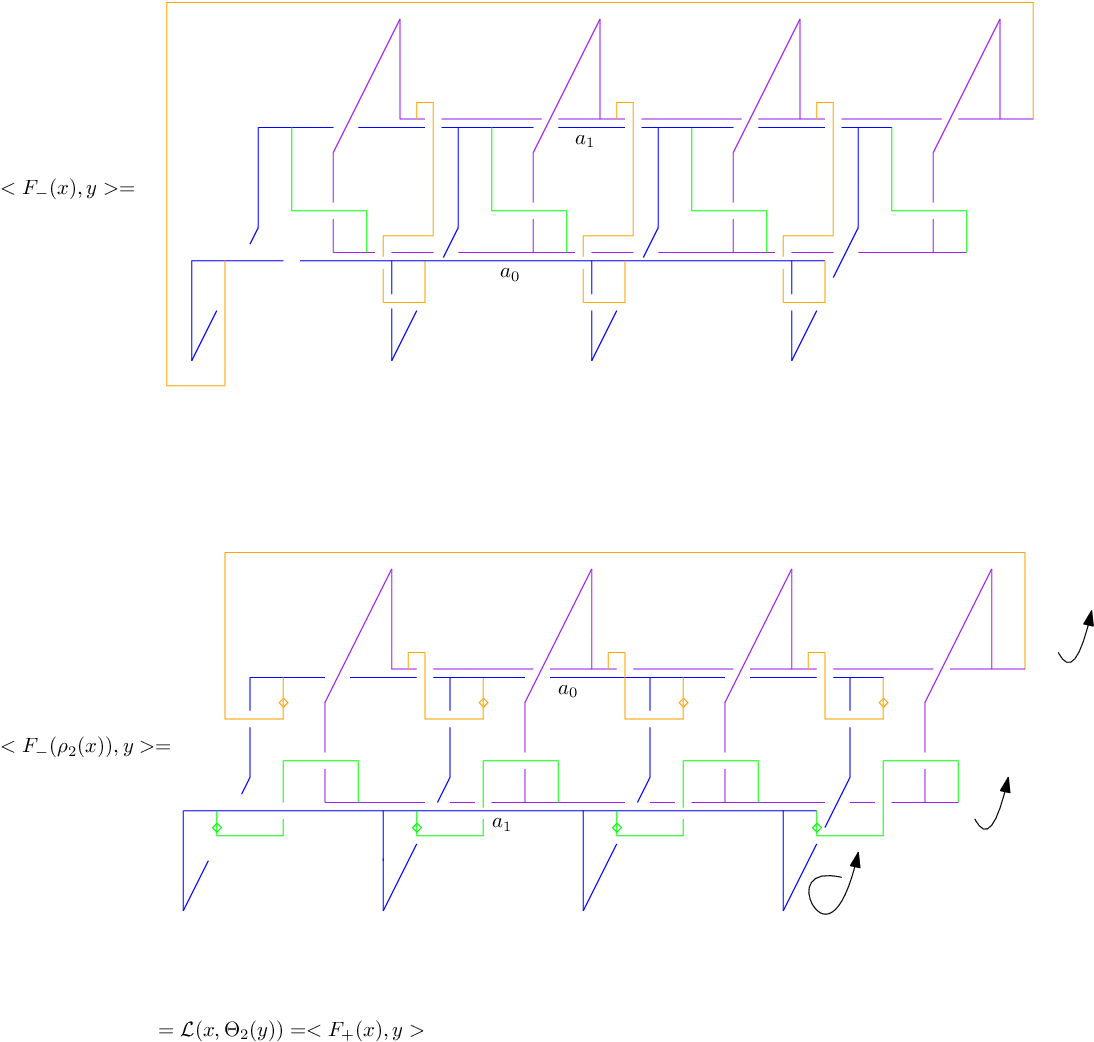}
\]

The rotations introduce the left twists to both green and orange strings and $\Theta_2$ action to $y$ (purple). 
\end{proof}

Now from Theorem \ref{thm_flatness}, we see $F_+=\FF$ as an operator from $Cf(n,\ZZ)$ to $Cf(n,F(\ZZ))$.
We will use this to prove $\FF$ is an isometry.
\begin{thm}\label{Thm:isomety_square_to_rot}
 The Fourier transform $\FF$ is an isometry from $Cf(n,\ZZ)$ to $Cf(n,\FF(\ZZ))$. Moreover we have $\FF^2=\rho_1$.  The same statements also hold for $F_-$ 
\end{thm}
\begin{proof}
 $<\FF(x),\FF(y)>$ is equal to the following diagrams, see Figure \ref{fig:isometry1}, with a $(\delta^{-1/2-2})^{4n}\delta^{-2}=\delta^{-10n-2}$ factor.

\begin{figure}
    \centering
    \includegraphics[width=350pt]{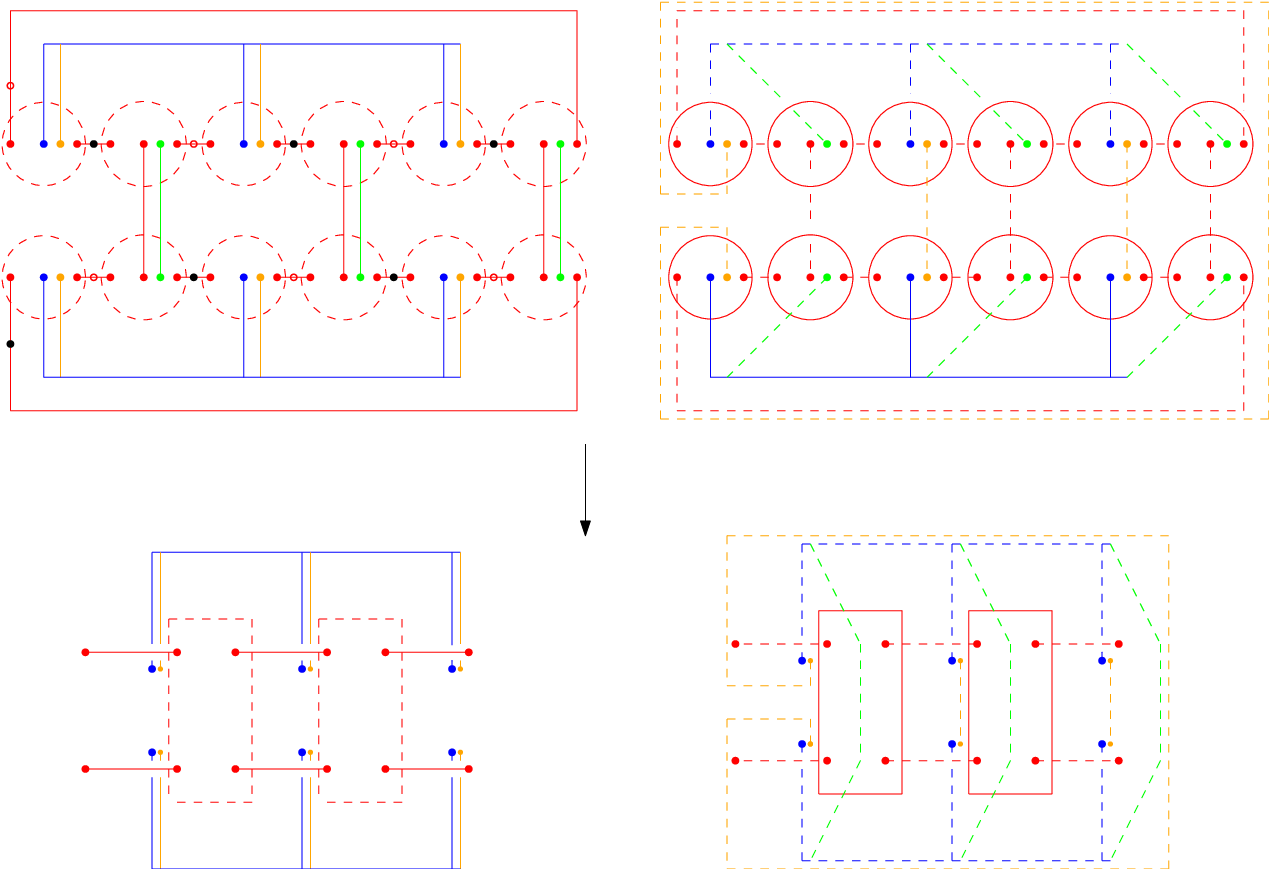}
    \caption{Graphic calculus}
    \label{fig:isometry1}
\end{figure}

Then we use similar moves both in the top and bottom part as in the proof of the theorem \ref{thm_flatness} to simplify the diagrams (the bottom is just a reflection of the top hence all the moves are the same), we get following diagrams, see Figure \ref{fig:isometry2}, with a factor of $\delta^{-10n-2+2n+4n+2+2n}=\delta^{-2n}$.
\begin{figure}
    \centering
    \includegraphics[width=350pt]{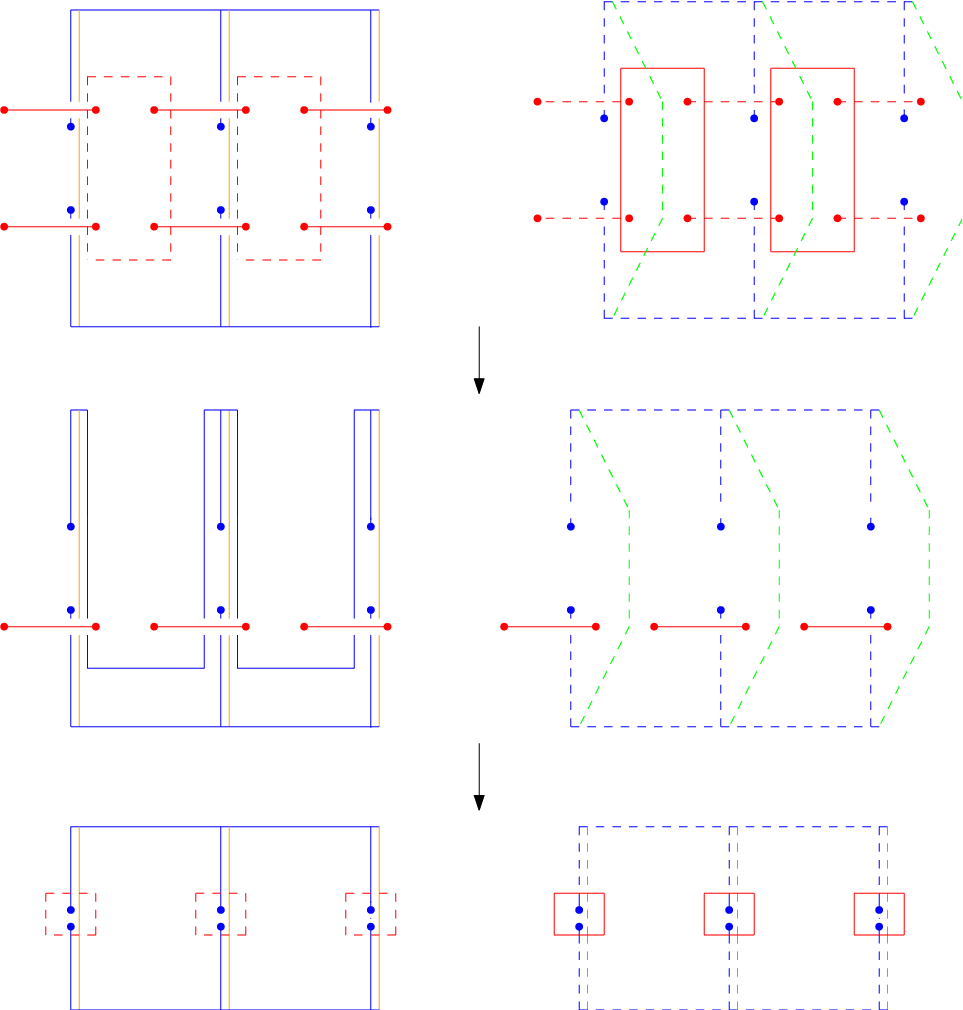}
    \caption{Graphic calculus}
    \label{fig:isometry2}
\end{figure}
The last diagram in Figure \ref{fig:isometry2} gives $\delta^{-2n}\delta^{2n}<x,y>=<x,y>$.

As for the second statement, $<F^2_+x,y>$ are given by the following diagrams, see Figure \ref{fig:square_to_rot}. 
\begin{figure}
    \centering
    \includegraphics[width=350pt]{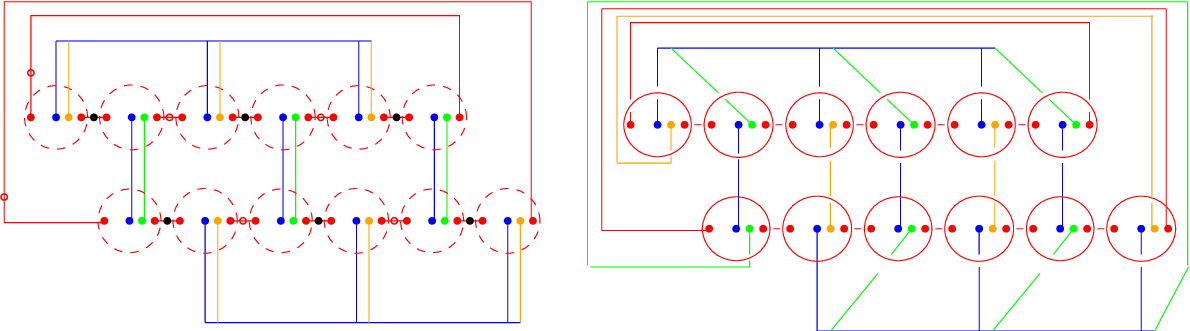}
    \caption{Two-layer presentation of $<F^2_+x,y>$}
    \label{fig:square_to_rot}
\end{figure}
One observes the isotopy given by rotating clockwise the diagram in the right lower corner to the left lower corner, one gets exactly the diagram for the inner product of $F_+(x)$ and $F_+{(\rho_1^{-1}y)}$ which is equal to $<x,\rho_1^{-1} y>$ by the proof of the first statement, which implies $F^2_+=\rho_1$. 

The statement for $F_{-}$ now follows directly from Theorem \ref{thm_flatness}.
\end{proof}

\begin{Cor}\label{Cor:full_comp}
We have the following identities,
\[
\begin{aligned}
\FF(T_{k+1}(x))&=\eta^{(-1)^k}T_{k}(\FF(x)),\\
\FF(\iota_{k+1}(x))&=\iota_{k}\FF(x).
\end{aligned}
\]
\begin{proof}
From Theorem \ref{thm_flatness}, \ref{Thm:isomety_square_to_rot} and Lemma \ref{lem:contraction_inclusion_equality}, \ref{lem:braiding&Fourier_comp}, we have, 
\[
\FF T_{2k+1}\FF^{-1}=\eta T_{2k}\implies \FF T_{2k+2}\FF^{-1}=\eta^{-1}\rho_1 T_{2k+3}\rho^{-1}_{1}=\eta^{-1}T_{2k+1}.
\]
For the first equality, we have,
\[
\FF \phi_{k+1}\FF^{-1}=\phi_k \implies \FF \phi_{k+1}\FF^{-1}\iota_k=\phi_k\iota_k=\delta d_{\hat{Z}}\implies \iota_k=\FF \iota_{k+1}\FF^{-1}.
\]
\end{proof}
\end{Cor}

Next we verify the braiding we defined here is indeed the 'square root' of the braiding on $\CC\boxtimes \CC$ ($c_{x,y}\boxtimes c_{\bar{x},\bar{y}}$).

\begin{prop}\label{braiding_squareroot}
  We have the following identity (as operators on the configuration space). 
\[
T_{2k+1}\circ T_{2k}\otimes T_{2k+2} \circ T_{2k+1}=\bigoplus_{X,Y\in \Irr(\CC)}c_{XZ_{2k},Y Z_{2k+2}}\boxtimes c_{\bar{X}Z_{2k+1},\bar{Y}Z_{2k+3}}
\]
\end{prop}
\begin{proof}
As usual, using Lemma \ref{lem:Theta_2_lemma} (Here we indeed show they are equal by evaluating in $\CC$ the inner product with basis vectors,  hence the lemma can be applied), we draw the simplified diagrams as follows. Now we first apply the twist property inside vertical red circles at the right upper and left lower corners, then apply the twist property again inside the upper and lower horizontal red circles and apply the cutting property. We have     
\[
\includegraphics[width=350pt]{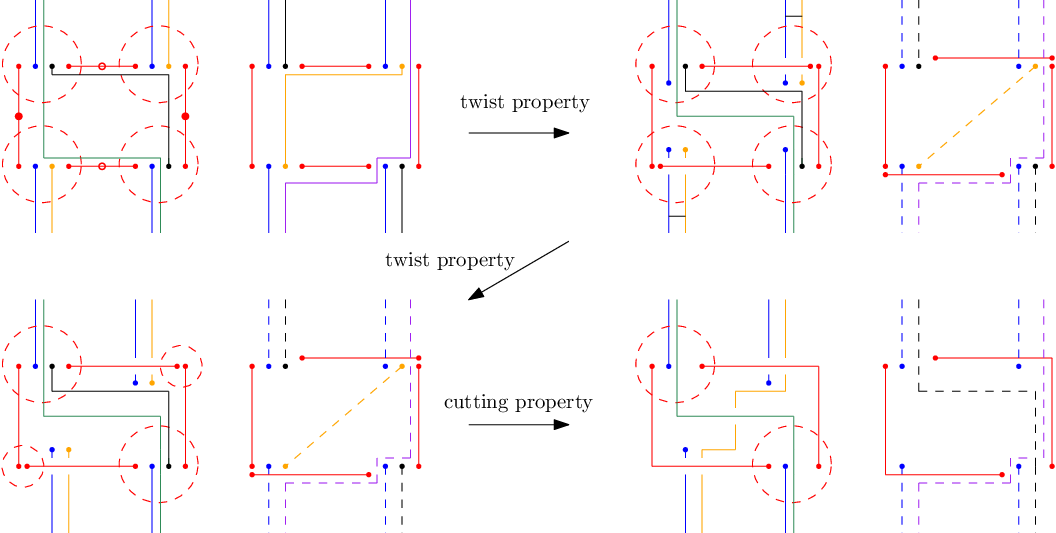}
\]
Now we use the fact that we are working in the configuration space and apply the cutting property, therefore we have,
\[
\includegraphics[width=350pt]{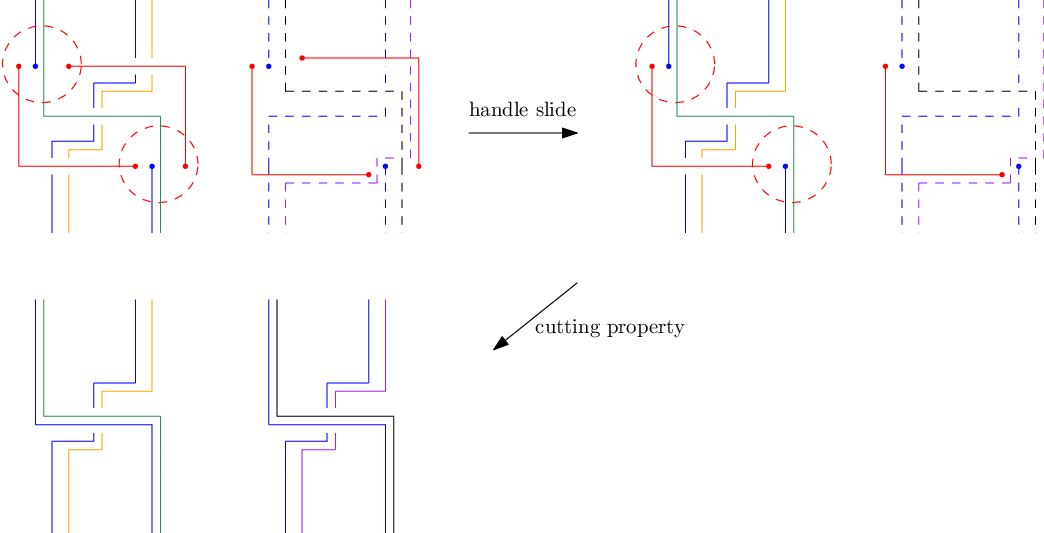}
\]
\end{proof}

%\begin{prop}
%   $\gamma$ is a commutative Frobnies algebra in $\CC\boxtimes \CC^{op}$, but an Azumaya (maximally non-commutative) algebra in $\CC\boxtimes \CC$
%\end{prop}

\section{$\Zn2$ permutation extension}
Now we define an even part of the shaded planar algebra with a shaded $\Zn2$ braiding structure from the data of $\CC$.

\begin{defn}
We define a colored and oriented shaded planar algebra $\PP^{\CC}$, the coloring set is given by the set of objects in $\CC$, with the involution given by the duality. Moreover $P^{\CC}_{n,+,\ZZ}=Cf(n,\ZZ)$ and the generating tangles are defined as follows.
\end{defn}

\includegraphics[width=350pt]{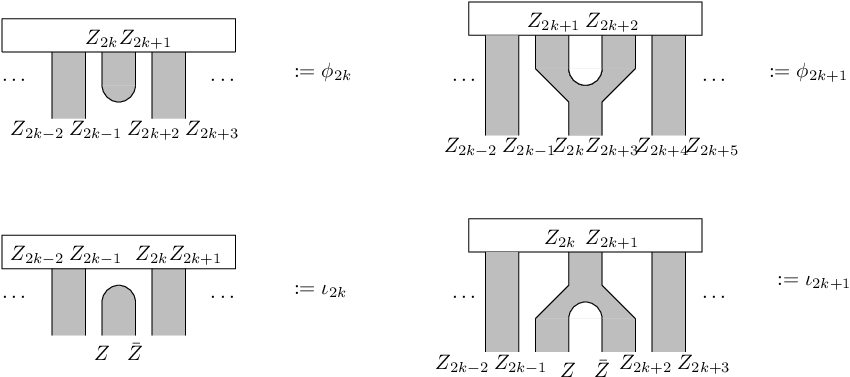}

\includegraphics[width=350pt]{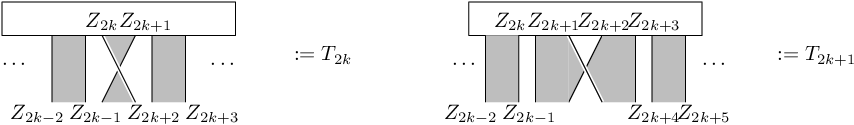}
 %\begin{itemize}
     %\item $P_{n,+,\ZZ}=Cf(n,\ZZ)$ as Hilbert spaces.
     %\item Tensor product, contractions, inclusions, braidings are defined. Moreover they are defined to be the morphsims in $\CC\boxtimes\CC$ (instead of the morphisms in $\CC$). 
     %\item Fourier transform as an operator from $P_+\to P_- $ and $P_-\to P_+$ is given by $\FF$, we give two interpretations of $\FF$, one is from pairing as shown in Figure \ref{fig:Fourier_pairing}, the other one is give by morphism in $\CC\boxtimes\CC$ diagrammatically see Figure \ref{fig:Fourier_transform}. Moreover we proved in \ref{Thm:isomety_square_to_rot} it is an finite order isometry with order equal to $2n$.
     %\item Moreover we have proved $\FF$ intertwines many inner operations. 
     %\item we show that graphic calculus of double strings are compatible with the graphic calculus in $\CC\boxtimes \CC$. Moreover change of height for odd part follows from those moves in the even part by applying Fourier transformations. (which will end up in the right space thanks to Figure \ref{fig:Fourier_transform})  
% \end{itemize}

\begin{Thm}\label{PA-Zn2_braiding_structures_verify}
 They define an even part of the shaded planar algebra with a shaded $\Zn2$ braiding structure (Definition \ref{def:Zn2_braiding_struc}).   
\end{Thm}

\begin{proof}
The following relations can be verified. By conjugating with $F_+$, we have $(R2)-(R6)$ with reversed shading are also satisfied (the relation holds for any labeling, hence we omit labels for simplicity).

\[
\includegraphics[width=350pt]{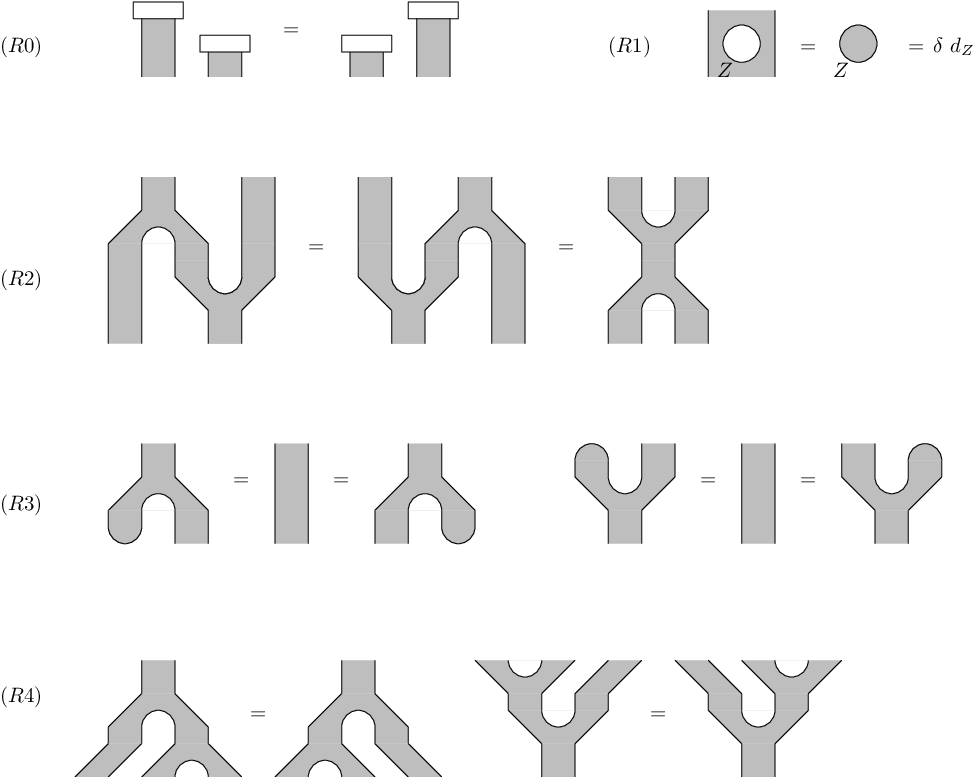}
\]

\[
\includegraphics[width=350pt]{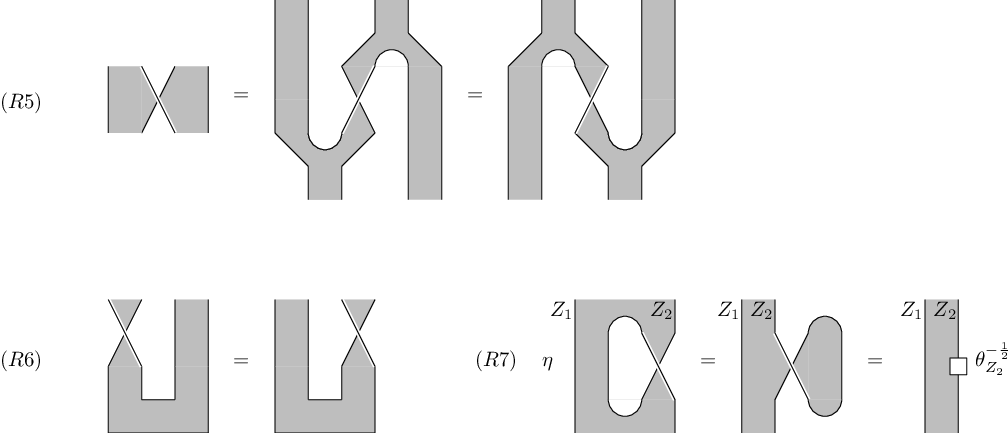}
\]

Here the braiding in $(R5),\ (R6)$ can be either positive or negative. 
$(R0)$ is obvious, $(R1)$ is the first identity in Lemma \ref{lem:contraction_inclusion_equality}, $(R2),(R4),(R7)$ follow from direct graphic calculus using Lemma \ref{lem:Theta_2_lemma} and Remark \ref{rem:Kirby_color_creat}, $(R3)$ follows from "zig-zag" relation which is proved in Lemma \ref{lem:zig-zag}.

The first equality of $(R5)$ is the definition \ref{def:braiding}, the second one follows from direct graphic calculus. $(R6)$ is proved in Lemma \ref{Lem:braiding_capcup_comp}. 

Now $(R0)-(R4)$ (also with shading reversed) gives the structure of an even part of planar algebra. $(C0)$ can be verified easily (Definition \ref{def:braiding}), $(C1)$ follows from $(R5),\ (R4)$. $(C2)$ can be verified using $(R5)$ on the right-hand side and applying $(R2),\ (R3),\ (R6)$. $(C3)$ will be proved later (Lemma \ref{Lem:C3verify}). $(C4)$ follows from $(R7)$ and other equalities are direct consequence of applying $(R5)$ and $(C3)$. And $(C5)$ follows from Proposition \ref{braiding_squareroot}.
\end{proof}

\begin{rmk}
We will abuse the notation to denote curls or the square root of twist in $(R7)$ by again a diamond symbol on the corresponding strand.
\[
\includegraphics[width=100pt]{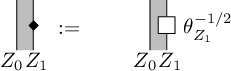}
\]
\end{rmk}

The immediate Corollary of $(C5)$ is the following 
\begin{Cor}\label{Cor:F_braiding_rels}
For $x\in P_{n,\ZZ}$, we have 
\[
F_{\pm,Z}(x)=d_Z(\theta^{\mp 1/2}_{Z_0}\otimes Id)\circ\prod^{2n-2}_{k=0}T^{\pm 1}_{2n-2-k}(x)
\]
\end{Cor}
\begin{proof}
 By comparing various definitions, The operators $F_{\pm,(Z)}$ are given by the following tangles, and the results follows from $(C5)$ (Proposition \ref{braiding_squareroot}) and $(R7)$. Here we present the proof for $F_{+,Z}$. 
   
\[
   \includegraphics[width=400pt]{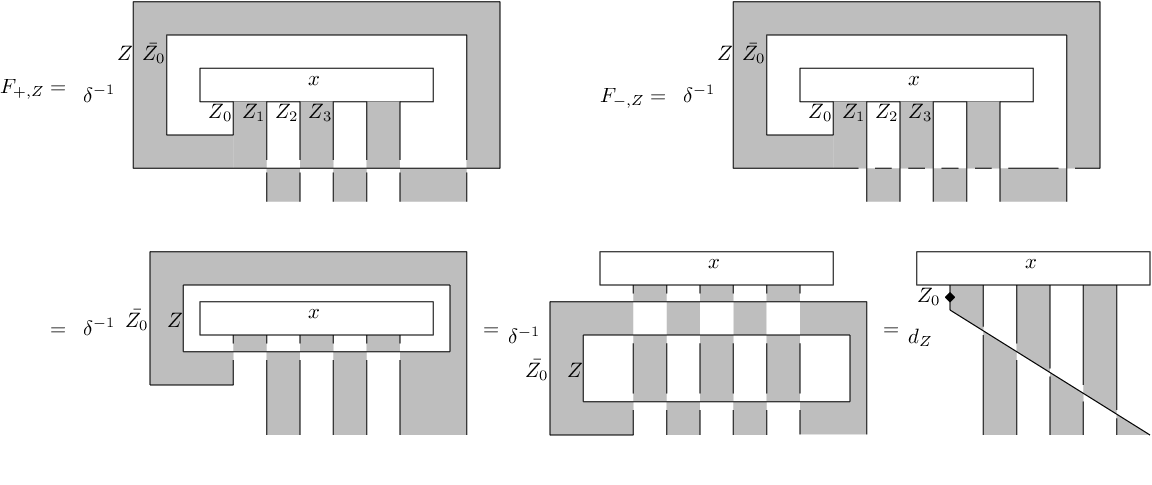}
\]  
\end{proof}

Now the $(C3)$ (see the next lemma) follows from Corollary \ref{Cor:F_braiding_rels}, Lemma \ref{lem:braiding&Fourier_comp} and Corollary \ref{Cor:full_comp}.

\begin{lem}\label{Lem:C3verify}
 The braidings we defined satisfy the following
 \[
 T_{k+1}T_{k}T_{k+1}=\eta^{(-1)^k}T_{k}T_{k+1}T_{k}.
 \]
\end{lem}
\begin{proof}

the identity is equivalent, by conjugating $F_{+,Z}$, to $T_1T_0T_1=\eta T_0T_1T_0$.

Now from Corollary \ref{Cor:F_braiding_rels} and \ref{Cor:full_comp}, we have $\prod^{2n-2}_{m=0}T_{2n-2-m}T_{1}=\eta T_0\prod^{2n-2}_{m=0}T_{2n-2-m}$. Since $T_mT_{m+2}=T_{m+2}T_{m}$, after cancellations we have $T_1T_0T_1=\eta T_0T_1T_0$
%We prove the first relation in $(R6)$ where the braidings are all negative, the other cases are similar. It's not hard to see by tensoring identity (double strings ) on the right and precomposing an inclusion, it is equivalent to proving the identity 
%\includegraphics[width=50pt]{PA-YBequation.eps}. Note the two operators are both unitary (up to a constant).

%Now we use Lemma \ref{lem:braiding&Fourier_comp} and Corollary \ref{Cor:full_comp} by picking special $x\in P_{2,\ZZ}$, we have,
%\[
%\includegraphics[width=400pt]{PA-YBproof.eps}
%\]
%Here we use $(R5)$ and the unitarity of the braidings to move $Z$-circle through caps and cups. And $(R0)-(R3)$ to resolve caps and cups. 
\end{proof}

Now in order to resolve the global constant $\eta$, one can easily check it is enough to let $\tilde{T}_{k}=\eta^{\frac{(-1)^k} {2}}T_k$. Then in favor of Theorem \ref{thm:lifting_thm}, we obtain an unshaded planar algebra $\PP$ with a $\Zn2$ braiding (the braiding structure now comes from $\tilde{T_k}$). The corresponding $\Zn2$ action is given by placing the input disks in between two $1$-labeled circles, see Lemma \ref{lem:Z_2_action}. Next lemma shows this action is the same as $\rho_2$ (Definition \ref{def:Z_2_action}).
\begin{lem}\label{lem:rel_rho_2andPA_action}
  The $\Zn2$ action is given by $\rho_2$.  
\end{lem}
\begin{proof}
From Theorem \ref{thm_flatness}, we have $F^{-1}_-F_+=\rho_2$, and it is easy to see the corresponding tangle of the left-hand side is isotopic to the desired form (see the proof of Corollary \ref{Cor:F_braiding_rels}).  
\end{proof}

Now by detailed analysis of the minimal projections and the $\Zn2$ action in the unshaded planar algebra, one gets the following.
\begin{Thm} \label{cat_of_Z_2-extension}
  We obtain a $\Zn2$-graded (unitary) spherical fusion category $\DD:=(\CC\bot\CC)^{\times}_{\Zn2}$ with a (unitary) $\Zn2$-crossed braiding. Moreover the projections in the odd part are given by a single cap labeled by $V$ for $V\in obj(\CC)$,  we denote them by $\{\hat{V}\}_{V\in obj(\CC)}$, we have 
    \begin{itemize}
        \item $\hat{V}$ is minimal $\Leftrightarrow$ $V$ is simple,
        \item $\hat{V}X\boxtimes Y=X\boxtimes Y\hat{V}=\widehat{VXY}$,
        \item $\hat{V}\hat{W}=\hat{W}\hat{V}=\oplus_{X\in \Irr(\CC)}XV\boxtimes \bar{X}W$. 
    \end{itemize}  
 The $\Zn2$ action $\rho$ (given by $\rho_2$ )acts on objects by $\rho(\hat{V})=\hat{V}$ and $\rho(X\bot Y)=Y\bot X$   
\end{Thm}

\begin{proof}
Recall the algebra structure of $P_n$ is given in Definition \ref{def:PA-composition}. $\theta_1$ is given by the following tangle,
\[
\includegraphics[width=150pt]{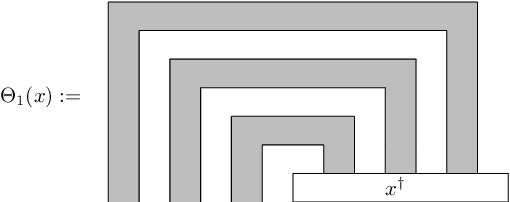}
\]
and the unitarity follows from the planar isotopy and that of $\CC$. 
We first describe the category coming from the even part. When $n$ is even, it is not hard to show the following two $\Hom$ space is isomorphic, as $\mathbb{C}^*$ algebras:
 \[
 \includegraphics[width=200pt]{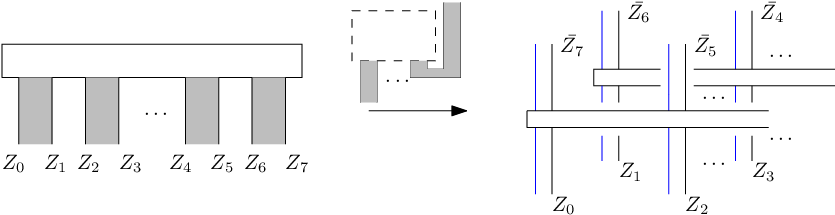}
 \] Hence the minimal projections in $P_n$ for $n$ even is equivalent to those in $P_2$ which can be described by the simple objects in $\CC\boxtimes \CC$, we denoted them by $p_{X\boxtimes Y},\ X,Y\in \Irr(\CC)$. The following representative will be used in further calculations
 \[
 \includegraphics[width=250pt]{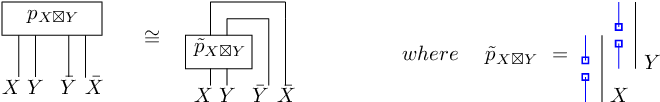}
 \]
 The corresponding fusion rings are isomorphic. Note the corresponding evaluation and coevaluation maps for $p_{X\boxtimes Y}$ are given by the diagram in $(R6)$ and its vertical reflection without the square root of twists (see also in Lemma \ref{Lem:braiding_capcup_comp}), and the rigidity follows from $(R6)$ and the planar isotopy.  Moreover thanks to Proposition \ref{braiding_squareroot}, the braiding structures agree. Therefore the category of projections from the even part is equivalent, as a unitary braided fusion category to $\CC\boxtimes\CC$, in particular, $P_{2n}$ is semisimple with a non-degenerate trace. 

 Now the semisimplicity of $\PP$ (or $P_{2n-1}$) follows from the unital inclusion $\iota_{2n-1}:P_{2n-1}\to P_{2n} $ and the trace preserving conditional expectation $\phi_{2n-1}:P_{2n}\to P_{2n-1}$, see \cite[Lem.~3.14]{muger03}.

 When $n=1$, let $\ZZ=\{V,W\}$, by computing the dimension of the configuration space, we have $\dim(P_{1,\ZZ})=\dim(\Hom_{\CC}(V,W^*))$. Hence the projections in $1$-box space are given by the single cap labeled by $V,\bar{V}$, and the projection is minimal when $V$ is simple. The dual projection is given by the cap labeled $\bar{V},V$ (rigidity follows from the planar isotopy of the single string), we denote them by $\hat{V},\ \hat{\bar{V}}$ respectively. Using the structure of the even part described previously, the fusion rule for $\hat{V}\hat{W}$ is immediate. ($\hat{V}\hat{W}=\hat{W}\hat{V}$ can be shown either by direct arguing $\oplus_{X\in \Irr(\CC)}XV\boxtimes XW=\oplus_{X\in \Irr(\CC)}XW\boxtimes XV$ or by establishing equivalence of projections simply using braidings)  
 
 Now for any projections $Q$ in $P_n$ for $n$ is odd, pick any $V\in  \Irr(\CC)$, we have $Q\otimes \hat{V}\in P_{n+1}$, hence  it can be written as sums of projections in $\CC\boxtimes \CC$. Using rigidity and the fusion rule for $\hat{V}\boxtimes \hat{\bar{V}}$, we have the projections in $P_n$ for $n$ odd are the sum of projections of the form $(X\boxtimes Y)\hat{V}$. Next we prove they are equivalent to the projections in the $P_1$.

 $(X\boxtimes Y)\hat{V}=\hat{V}(X\boxtimes Y)$ follows from the equivalence of two projections given by braidings: 
\[
\includegraphics[width=200pt]{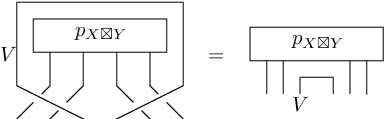}
\]
Next by direct calculation of $\oplus_{Z\in \Irr(\CC)}XZV\boxtimes Y\bar{Z}W=\oplus_{Z\in \Irr(\CC)}XYV\bar{W}Z\boxtimes\bar{Z}$, we have (again use rigidity and fusion rule):
\begin{equation}
 \Hom_{\DD}((X\boxtimes Y)\hat{V},\hat{W})\cong \Hom_{\CC\boxtimes \CC}(X\boxtimes Y\hat{V}\hat{\bar{W}},1\otimes 1)\cong\Hom_{\CC}(XYV,W)
\end{equation}
And now $(X\boxtimes Y)\hat{V}=\widehat{VXY}$ follows from comparing the quantum dimension.

Next, we examine the $\Zn2$ actions on the objects, $\rho_2(\hat{V})=(\hat{V})$ follows directly from the planar algebra description of the action (Lemma \ref{lem:Z_2_action}) and the objects in odd part. The action on the even part needs a little bit of calculation. By observing the similarities between more explicit definition (Definition \ref{def:Z_2_action}) of $\rho_2$ and braidings action on the $\tilde{p}_{X\bot Y}$, the following identity holds (we also use the proof of Lemma \ref{Lem:braiding_capcup_comp} to make the flipping action involved in $\rho_2$ more direct).
\[
\includegraphics[width=250pt]{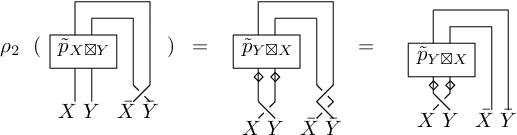}
\]
Now $\rho_2(X\bot Y)\cong Y\bot X$ follows from the next identities.
\[
\includegraphics[width=250pt]{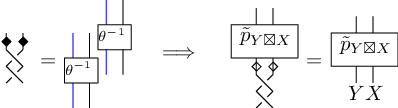}
\]
Where the first identity can easily proved by first isotopy the left-hand side to a full twist on double strings, then using Proposition \ref{braiding_squareroot} (also see Remark \ref{rmk:doulestrand_braiding}).  

Finally, the coherence conditions of $\Zn2$-crossed braiding can be checked easily from Corollary \ref{Cor:Z_2_action} and $(C3)$. 
\end{proof}
\begin{rmk}
The sphericality can also be derived simply using the braiding that strands can be isotopied above everything.   
\end{rmk}

\begin{rmk}\label{rmk:doulestrand_braiding}
 Proposition \ref{braiding_squareroot} has the following graphic presentation.
 \[
 \includegraphics[width=100pt]
 {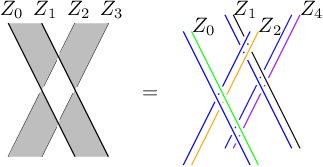}
 \]
\end{rmk}

\begin{rmk}
Even though the unitarity is crucial for the reconstruction program, it is not necessary in our construction, the whole proof works very similarly. Now instead of picking the orthonormal basis in the configuration space, we take the basis and the dual basis under the non-degenerate bilinear form given by $<,>$. The place we also use the unitarity is when we take the square root in $\mathbb{C}$ of the quantum dimension that appears in the coefficient of many operations ($\phi_k,\ \iota_k,\ T_k$). These can be adjusted to any field by redefining the vertical composition with additional $\frac{1}{d}$ factors ($d$ is the quantum dimension of the corresponding simple object on blue strands) so that the coefficient only involves the product of the quantum dimensions. 
\end{rmk}

\section{Some applications}
In this section, we will first prove the relations satisfied by the $\Zn2$ braidings in the twisted category $\DD:=(\CC\boxtimes\CC)^{\times}_{\Zn2}$ (Theorem \ref{mainthm3}), then we will discuss the twisted/untwisted correspondence (see Table \ref{tab:t/unt}). Finally, several identities for modular data of $\CC$ will be derived including equation \ref{eq:intro_obstruct} and its generalization using simple isotopes in $\DD$. 
\begin{Thm}\label{thm:relations}
  The $\Zn2$ braidings $T_k\ (0\leq k\leq 2n-2)$ satisfy the following relations 
   \[
   \begin{aligned}
       T_kT_{k+2}&=T_{k+2}T_k,\\
       T_{k+1}T_{k}T_{k+1}&=\eta^{(-1)^k}T_{k}T_{k+1}T_{k},\\
       (T_{2n-2}T_{2n-1}\cdots T_1T_0)^{2n}&=(\prod^{2n-1}_{k=0}\theta^{\frac{1}{2}}_{Z_k})\Id,\\
       (T_0T_1\cdots T_{2n-2}T_{2n-2}\cdots T_1T_0)^2&=\theta^2_{Z_0}\Id.\\
       T_0T_1\cdots T_{2n-2}T_{2n-2}\cdots T_1T_0&\ commutes\ with\ T_k\ (0\leq k\leq 2n-2)
   \end{aligned}
   \]
\end{Thm}
 \begin{proof}
     The first one is obvious, the second one is Lemma \ref{Lem:C3verify}. The third and fourth ones  follow from rewriting $(F_+)^{2n}=\Id$ and $F_+=F_-\rho_2$ using Corollary \ref{Cor:F_braiding_rels}, and $\rho_2^2=\Id$. The last one follows from Lemma \ref{lem:rel_rho_2andPA_action} and the fact that the planar algebra action respects braidings (Lemma \ref{lem:Z_2_action}).
 \end{proof}
Now consider the configuration space with all black strands labeled by the same simple object $V$. We denote it by $Cf^{V}(n)$.  We refer to \cite[Thm.~8]{BH71} for the presentation of symmetric mapping class group $\SMod(\Sigma_{n-1})$,  we thus have the following Corollary,
 \begin{Cor}\label{cor:SMod_rep}
  $Cf^{V}(n)$ gives a unitary projective representation of $\SMod(\Sigma_{n-1})$.     
 \end{Cor}

Now we consider the case that $V$ is the trivial object $1$, then the space $Cf^{1}(n)$ is the same as the configuration space in \cite{LX19} for $m=2$. And the planar algebra we constructed here corresponds to the 2-interval Jones-Wassermann subfactor. As a result of Theorem \ref{Thm:isomety_square_to_rot}, we have a different proof for the following Theorem.

\begin{prop}[\cite{LX19}]
 The 2-interval Jones-Wassermann subfactor is self-dual.     
\end{prop}

And the braiding in this case is easier to describe, for example, 
\[T_{2k}=\cdots \Id\otimes(\bigoplus_{V\in \Irr(\CC)}\theta^{-1}_{V}\Id_V\boxtimes \Id_{\bar{V}})\otimes \Id\cdots\]
When working in the $2$-box space, we have $T_0=T_2$ is the $T$-matrix of $\CC$, the Fourier transformation $F_{+}$ corresponds to $S$-matrix (\cite{LX19}) and the charge conjugation matrix $C$ is given by the $Z_2$-action $\rho_2$. See Figure \ref{fig:basic_identity} for more details. 

\begin{figure}
    \centering
    \includegraphics{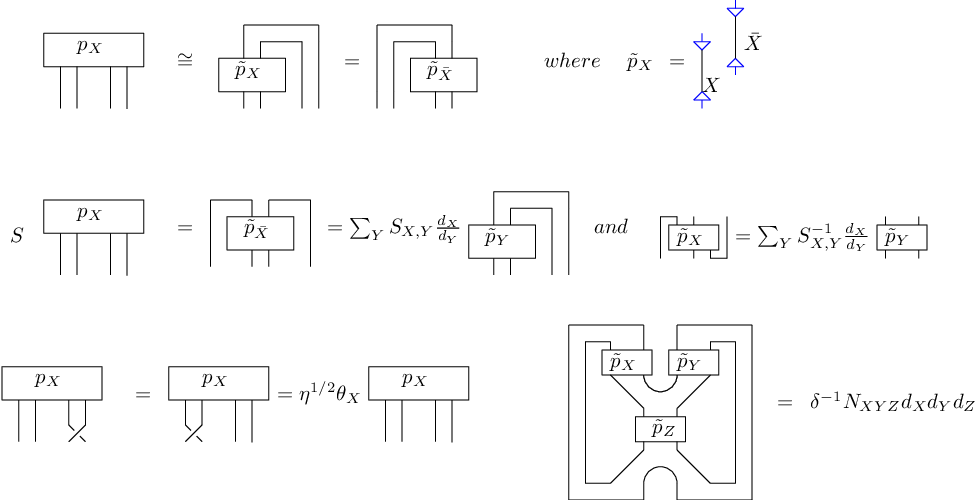}
    \caption{Some basic identities when all strands are labeled by $1$ (or $\hat{1}$)}
    \label{fig:basic_identity}
\end{figure}

The following Corollary follows from Theorem \ref{thm:relations} and Lemma \ref{Cor:F_braiding_rels}, which gives the well-known projective representation of $\SL_2(\mathbb{Z})$
\begin{Cor}\label{Cor:SL2Z_relation}
The following identities hold
\[
\begin{aligned}
S^4=\Id & (S^2=C)\\
(T^{-1}S)^3&=\eta C\\
(ST)^3&=\eta^{-1}
\end{aligned}
\]
   
\end{Cor}

\begin{rmk} \label{rmk_twist_untwist}
We will show in Section $8$, in general our $\SMod(\Sigma_{n-1})$ representation given by $Cf^1(n)$ is equivalent to the representation coming from Reshetikhin-Turaev TQFT state space associated with closed surface of genus $n-1$, see Theorem \ref{thm:Rel_to_RT}, which indicates the genus-$0$ data (braidings) of the extension theory contains the higher genus data of the original theory. It interprets the twisted/untwisted correspondence \cite{Bantay98, Gui21}. 
\end{rmk}

Previous identities are derived only from the isotopy of strands.
One may expect there are more identities derived from the isotopy in $\DD$. Indeed we obtain numerous interesting identities. For example in the next proposition, we list here only three of them involving merely one isotopy through input disks of trivial labeled $2$-box space ($Cf(2,\ZZ)$ with all $Z_i$ being trivial objects).
\begin{prop}
  We have the following identities. 
\[
\begin{aligned}
 \sum_{V\in \Irr(\CC)}\frac{\theta_V}{\theta_X}S_{YV}N^Z_{VX}&=\sum_{V\in \Irr(\CC)}\frac{\theta_V}{\theta_W}S_{XV}N^Z_{VY},\\
 \sum_{V,W\in \Irr(\CC)}\delta_{Y,Z}\frac{\theta_V}{\theta_W}S_{XV}N^Z_{VW}d_Xd_W&=\sum_{V,W\in \Irr(\CC)}\delta_{X,Z}\frac{\theta_V}{\theta_W}S_{YV}N^Z_{VW}d_Yd_W,\\
 \sum_{V\in \Irr(\CC)}\delta_{Y,Z}\frac{\theta_V}{\theta_X}N^Z_{VX}d_V&=\sum_{V,W\in \Irr(\CC)}\frac{\theta_V}{\theta_W}S_{XZ}S_{Y\bar{W}}N^Z_{WV}d_V.
\end{aligned}
\]
\end{prop}
\begin{proof}
    
\end{proof}
Here we list all three isotopy identities we used, and the proof of the first one is established, the others can be derived similarly.
\[
\includegraphics[width=400pt]{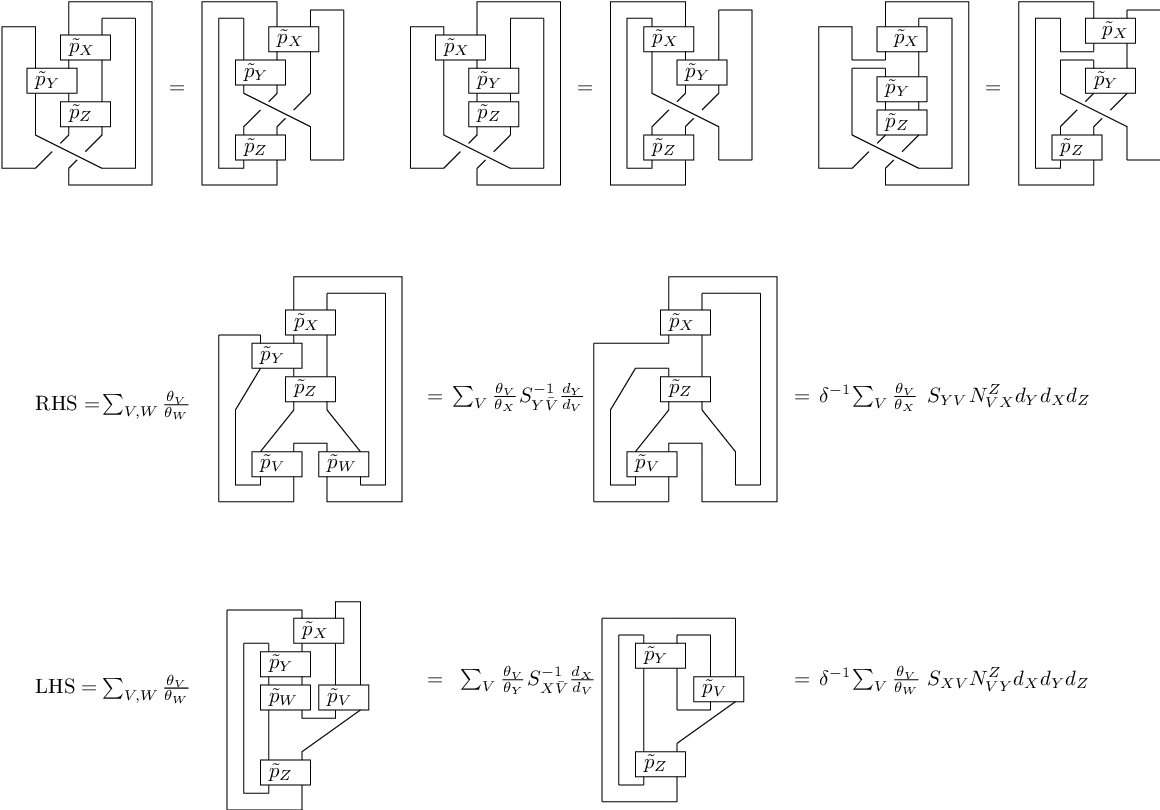}
\]

We don't know, at this stage, whether all the identities we derived can be generated by the Verlinde formula, balancing equation or $\SL_2(Z)$ relations. If not, we will have new obstructions for the realization of the modular data. We conjecture one can deduce some new identities using isotopy only involving trivial labeled $2$-boxes .

\begin{rmk}
 The Verlinde formula and the balancing equation can be easily derived from isotopy in $\DD$. For instance, the Verlinde formula is obtained by evaluating the last diagram in Figure \ref{fig:basic_identity} after shifting the bottom strand to the top using sphericality (braiding).
\end{rmk}

The next identity comes from considering only planar isotopy among nontrivial labeled $2$-box spaces, which are not generated by the known identities.
\begin{prop}\label{New_identity}    
We have the following identity,
\[
\includegraphics[width=325pt]{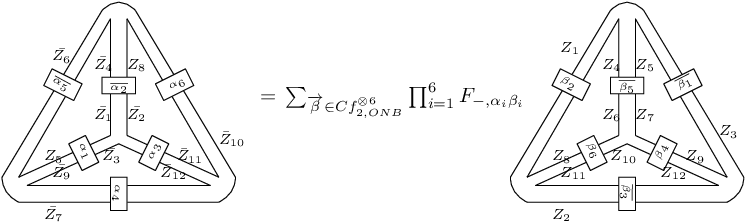}
\]
Here instead of making strings zigzag, we simply rotate the boxes (one can retrieve everything to a standard form), and $\beta_i,\alpha_i$ are the orthonormal basis in the corresponding configuration space. We denote $\overline{\alpha}:=\rho_1(\alpha)$. $F_{-,\alpha_i\beta_i}:=<F_-(\alpha),(\beta)>$.
\end{prop}
\begin{proof}
 The right-hand side is given by applying (inverse) Fourier transform $F_{-}$ to every $2$-box (rotate $\frac{\pi}{2}$ counter-clockwise), see the following diagram.
\[
\includegraphics[width=100pt]{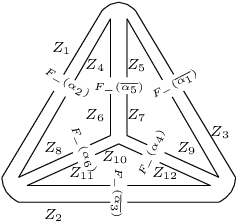}
\]

Direct isotopy then gives the left-hand side.
\end{proof}

When all $Z_i's$ are trivial, then we replace $\alpha_i,\beta_i$ by $\tilde{p}_{X_i}$ and $\tilde{p}_{Y_i}$ respectively, and $F_{-,\alpha_i\beta_i}=S^{-1}_{X_iY_i}$. The previous proposition reduces to the identity known as 6j-symbol self-duality \cite{Liu19}. Now we use $\Lambda$ to denote the subspace spanned by admissible colors $\{X_1\otimes X_2\otimes X_3\otimes X_4\otimes X_5\otimes X_6\}_{\overrightarrow{X}}$ such that $N_{\bar{X}_1\bar{X}_5X_2}N_{\bar{X}_1X_4X_3}N_{X_2X_6\bar{X}_3}N_{X_4X_5X_6}\neq 0$ . Now we have the following non-modular-group-relation identity. 
%\begin{Cor}(\cite{Liu19})
%Let $\big|\Sym{X_1}{X_2}{X_3}{X_4}{X_5}{X_6}\big|^2$ be the basis-free $6j$-symbol ($F$-symbol) as described in \cite{Liu19}. We denote the vector $\{\big|\Sym{X_1}{X_2}{X_3}{X_4}{X_5}{X_6}\big|^2\}_{\overrightarrow{X}}$ by $v$. We have
%\[
%det(S^{\otimes 6}P-I^{\otimes 6})=0,
%\]
%where $P$ is the composition of a permutation matrix permuting the tensor factors by the cycle $(16)(25)(34)$ and $C^{\otimes 3}\times I^{\otimes3}$. ($C$ is the charge conjugation matrix) 
%\end{Cor}
\begin{Thm}\label{thm:obstruct}
 We have the following identity,
 \begin{equation}\label{eq:Obstruction}
  Det|_{\Lambda}(S^{\otimes 6}P-I^{\otimes 6})=0.   
 \end{equation}
Where $P$ is the composition of a permutation matrix permuting the tensor factors by the cycle $(16)(25)(34)$ and $C^{\otimes 3}\times I^{\otimes3}$ ($C$ is the charge conjugation matrix). $Det|_\Lambda$ is the determinant of the matrix restricted in $\Lambda$. 
\end{Thm}
\begin{proof}
Let $\Big|\Sym{X_1}{X_2}{X_3}{X_4}{X_5}{X_6}\Big|^2$ be the basis-free $6j$-symbol ($F$-symbol) as described in \cite{Liu19} and we denote the vector $\biggl\{\Big|\Sym{X_1}{X_2}{X_3}{X_4}{X_5}{X_6}\Big|^2\biggr\}_{\overrightarrow{X}}$ by $v$.
Now $\Big|\Sym{Y_1}{Y_2}{Y_3}{Y_4}{Y_5}{Y_6}\Big{|}^2$ is given by the value of the diagram on the right-hand side in the previous proposition, while the left-hand side is $\Big|\Sym{\bar{X_6}}{\bar{X_5}}{\bar{X_4}}{X_3}{X_2}{X_1}\Big|^2$. The result follows by observing that $v$ is an eigenvector of the matrix $S^{\otimes 6}P$ with eigenvalue $1$ and $v|_{\Lambda^c}=0$.
\end{proof}

\begin{rmk}\label{rem:infinite_obstruc}
Identity \ref{eq:Obstruction} is an obstruction for modular category realization of $S$-matrix. This identity can be applied to the $S$-matrix of the gauged theory discussed in the next section, and obtain another obstruction. As one can repeatedly perform the $\Zn2$ permutation gauging, there will be infinitely many obstructions!
 % In this theorem,  whether such eigenvector $v$ (with many $0$'s) with eigenvalue $1$ exists is already an obstruction  ($v$ needs to be a positive vector if one considers unitary case).     
\end{rmk}
\section{Equivariantization theory}

In this section we will work on the $\Zn2$-equivariantization of the category $\DD$ obtained in Theorem \ref{cat_of_Z_2-extension}. The resulting category is denoted by $\DD^{\Zn2}$.

Here in our case, the action $\rho$ is explicit and moreover it is strict ($\rho^2=\Id$), one obtains the simple objects in $\DD^{\Zn2}$ are of the form $(XY):=(X\bt Y\oplus Y\bt X,\  \Id)$, $(X,\pm):=(X\bt X,\pm\Id)$ and $(\hat{X},\pm):=(\hat{X},\pm\Id)$ ($X,Y\in \Irr(\CC)$, $1_{\DD^{\Zn2}}=(1,+)$). Moreover, thanks to the properties of the braiding and the topological description of the action. It is straightforward to see the structure morphisms (ev, coev, braiding) commute with the action, thus giving rise to the  structure morphism in $\DD^{\Zn2}$. Therefore we can do calculations in the planar algebra $\PP$ for $\DD^{\Zn2}$.    

Now let $S',T, N, S^{eq}, T^{eq}, \mathcal{N} $ denote the $S (unnomalized),T$ matrix and fusion coefficient for $\CC$ and $\DD^{\Zn2}$ respectively. 
\begin{Thm}\label{thm:S_data}
We have $S^{eq}$ is symmetric and 
\[
\begin{aligned}
   S^{eq}_{(XY),(ZW)}&=2(S'_{X,Z}S'_{Y,W}+S'_{XW}S'_{YZ} )\\
   S^{eq}_{(XY),(Z,\epsilon)}&=2S'_{X,Z}S'_{Y,Z}\\
    S^{eq}_{(X,\epsilon_1),(Y,\epsilon_2)}&=(S'_{X,Y})^2\\
   S^{eq}_{(XY),(\hat{Z},\epsilon)}&=0\\
   S^{eq}_{(X,\epsilon_1),(\hat{Y},\epsilon_2)}&=\epsilon_1 \delta S'_{X,Y}\\
   S^{eq}_{(\hat{X},\epsilon_1),(\hat{Y},\epsilon_2)}&=\epsilon_1\epsilon_2\eta^{-1}\theta^{1/2}_X\theta^{1/2}_Y(S'T^2S')_{X,Y}\\
\end{aligned}
\]
In addition, $T^{eq}$ is diagonal matrix with
\[
\begin{aligned}
T^{eq}_{(XY),(XY)}&=\theta_X\theta_Y\\
   T^{eq}_{(X,\epsilon),(X,\epsilon)}&=\theta^2_X\\
   T^{eq}_{(\hat{X},\epsilon),(\hat{X},\epsilon)}&=\epsilon\eta^{1/2}\theta^{1/2}_X    
\end{aligned}
\]
\end{Thm}
\begin{proof}
As described in the preliminary section. It suffices to evaluate the following four diagrams in the planar algebra. Here $\epsilon_i=\pm 1$, indicates the morphism we pick for the $u_g$.
\[
\includegraphics[width=400pt]{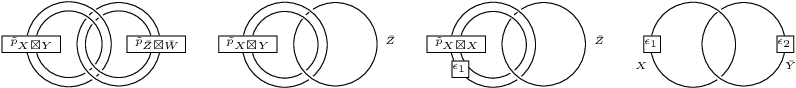}
\]
The first three identities follow from the evaluation of the first diagram by using Proposition \ref{braiding_squareroot}. The fourth and fifth ones come from the evaluation of the second and third diagrams, the detailed calculations are as follows,
\[
\includegraphics[width=400pt]{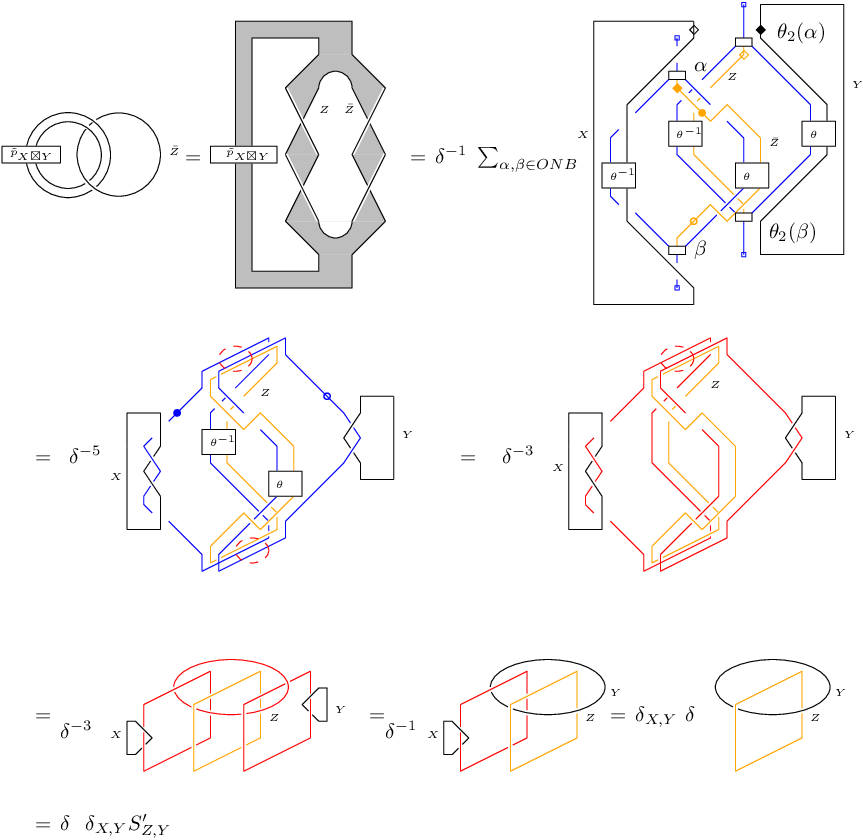}
\] 

Finally, the last identity follows from the following evaluation of the last diagram,
\[
\includegraphics[width=300pt]{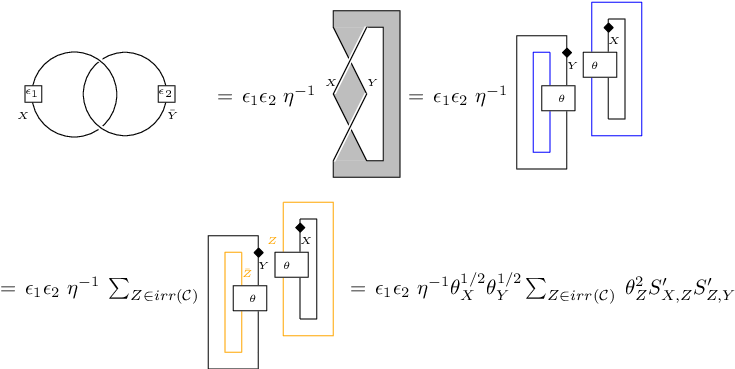}
\]
\end{proof}

Now we simply denote the normalized $S$-matrix for $\CC$ by $S$, and define the matrix $P:=\eta^{-1}T^{1/2}ST^2ST^{1/2}$(this is the same $P$ defined in \cite{Bantay98, BHS98, KLX05}, since their $T$ matrix is nomalized by $\eta^{-1/3}$). The following corollary is obtained by direct computation.
\begin{Cor} \label{Cor:fusionrule}
  $\DD^{\Zn2}$ is modular and the fusion rules are as follows
\[
\begin{aligned}
 \mathcal{N}_{(X_1Y_1),(X_2Y_2),(X_3Y_3)}&=N_{X_1X_2X_3}N_{Y_1Y_2Y_3}+N_{X_1X_2Y_3}N_{Y_1Y_2X_3}+N_{X_1Y_2X_3}N_{Y_1X_2Y_3}+N_{Y_1X_2X_3}N_{X_1Y_2Y_3},\\
 \mathcal{N}_{(X_1Y_1),(X_2Y_2),(Z,\epsilon)}&=N_{X_1X_2Z}N_{Y_1Y_2Z}+N_{X_1Y_2Z}N_{Y_1X_2Z},\\
 \mathcal{N}_{(XY),(Z_1,\epsilon_1),(Z_2,\epsilon_2)}&=N_{XZ_1Z_2}N_{YZ_1Z_2},\\
 \mathcal{N}_{(Z_1,\epsilon_1),(Z_2,\epsilon_2),(Z_3,\epsilon_3)}&=\frac{1}{2}N_{Z_1Z_2Z_3}(N_{Z_1Z_2Z_3}+\epsilon_1\epsilon_2\epsilon_3),\\
\mathcal{N}_{(XY),(\hat{Z}_1,\epsilon_1),(\hat{Z}_2,\epsilon_2)}&=\sum_{W\in \Irr(\CC)}\frac{S_{XW}S_{YW}S_{Z_1W}S_{Z_2W}}{S^2_{1W}},\\
\mathcal{N}_{(X,\epsilon),(\hat{Z}_1,\epsilon_1),(\hat{Z}_2,\epsilon_2)}&=\frac{1}{2}\sum_{W\in \Irr(\CC)}\frac{S^2_{XW}S_{Z_1W}S_{Z_2W}}{S^2_{1W}}+\frac{1}{2}\epsilon\epsilon_1\epsilon_2\sum_{W\in \Irr(\CC)}\frac{S_{XW}P_{Z_1W}P_{Z_2W}}{S_{1W}}.\\
\end{aligned}
\]
  
\end{Cor}
\begin{proof}
One shows the normalized $S^{eq}$ is unitary by using Theorem \ref{thm:S_data} and the unitarity of $S$ . The fusion rules follow from the Verlinde formula.   
\end{proof}
\begin{rmk}
 The above results coincide with those in \cite{Bantay98}, \cite[Sec.~9.3]{KLX05} and \cite[Sec.~4.4]{BHS98},  which are in the setting of conformal field theory.
\end{rmk}

\begin{exmp}
 Metaplectic categories $SO(N)_2$ ($N=2r+1$) are the unitary modular categories with the same fusion rule as the type $B$ quantum group category at level $2$,  see \cite{HNW13,HNW14,ACRW16} for the characterization of inequivalent theories. Combined with Corollary \ref{Cor:fusionrule}, we have the following result.   
\end{exmp}

\begin{Cor}
 $SO(N)_2$  ($N=2r+1$)  is equivalent to the subcategory of a $\Zn2$ permutation gauging of $\CC(\Zn N,q)$ generated by $(\hat{1},+)$.    
\end{Cor}
\begin{proof}
 One can directly check the fusion rule using Corollary \ref{Cor:fusionrule}, therefore it suffices to show we can obtain all $2^{s+1}$ inequivalent categories \cite{ACRW16} ($s$ is the number of prime factors of $N$). The choices of non-degenerate quadratic form $q$ give $2^s$ inequivalent theories \cite{ACRW16}, and the last two choices come from picking a different square root of the global dimension (as the value of a single strand circle with the trivial label), which, in particular, change $\eta$ to $-\eta$ and the value $T_{(\hat{1},\pm),(\hat{1},\pm)}$ will be different.  (One can also compare to the calculations in \cite[Section~5]{GNN09})  
\end{proof}

\begin{rmk}
In our construction of the twisted category $\DD$, we make two choices. The first involves lifting the shaded planar algebra, as discussed after Theorem \ref{thm:lifting_thm}. The second is selecting a square root of the global dimension $\mu$ of $\CC$. We expect that these choices will account for all inequivalent $\Zn2$ permutation extensions as $\Zn2$-crossed braided fusion categories . 
\end{rmk}

%\begin{proof}
 %One can directly check the fusion rule using Corollary \ref{Cor:fusionrule}, Let $N=\prod_i p_i^{s_i}$ and $s=\sum_i s_i$, it is known there are $2^{s+1}$ inequivalent theories and the choices of non-degenerate quadratic form $q$ gives $2^s$ inequivalent theories  \cite{ACRW16}.   
%\end{proof}

\section{Some observations and questions} 

In this section, we present additional applications, problems, and future prospects related to the $\Zn2$ permutation gauging. We will first demonstrate the equivalence of the 0-genus data in the twisted category $\DD$ (specifically, the configuration space $Cf^1(n)$ and the action of the crossed braiding) with the Reshetikhin-Turaev TQFT state space of $\CC$ associated with closed $(n-1)$-genus surfaces, as well as the isomorphism concerning the action of the symmetric mapping class group. Consequently, the finiteness of braid group representations in $\DD$ is equivalent to the finiteness of representations in the corresponding higher genus surface $\SMod(\Sigma)$. Based on this, we discuss the conjecture proposed by Naidu and Rowell in 2011, raise several questions, and prove Theorem \ref{thm:ProF}. Finally, we outline some future research directions including more general cyclic permutation gauging and extensions of Quon languages in quantum information.
\subsection{Relation to the Reshetikhin-Turaev TQFT}
We refer to \cite{RT91, Tur10} for the basic definition of Reshetikhin-Turaev TQFT and the corresponding mapping class group actions.  

The state space associated to the closed genus $g$ surface $\Sigma_g$ is given by (here we adopt the notation in \cite[Chapter~IV]{Tur10}, $I$ is the index set and $\{V_{i}\}_{i\in I}$ is the set of all simple objects)
$$
\Psi_g=\bigoplus_{i\in I^{g}}\Psi^i_g=\bigoplus_{i\in I^{g}}\Hom(1,\bigotimes_{r=1}^g(V_{i_r}\otimes \overline{V}_{i_r})).
$$
Here we denote the vector space $\Hom(1,\bigotimes_{r=1}^g(V_{i_r}\otimes \overline{V}_{i_r}))$ by $\Psi^i_g$,

%Next we define a subspace $\mathcal{S}$ of $\Psi_g^{\otimes 2}$
%$$
%\mathcal{S}=\bigoplus_{i\in I^{g}}\Psi^i_g\otimes\Psi^{i^*}_g=\bigoplus_{i\in I^{g}}(\Hom(1,\bigotimes_{r=1}^g(V_{i_r}\otimes V_{i_r}^*))\otimes \Hom(1,\bigotimes_{r=1}^g(V_{i_r}^*\otimes V_{i_r})))
%$$
Now we define the map $\Phi_g: \Psi_g\to Cf^1(g+1)$ by 
$$
\Phi_g(\bigoplus_{i\in I^g} f_i)=\sum_{i\in I^g, \alpha, W_r\in \Irr(\CC)} (Id_{V_{i_1}}\otimes \bigotimes^{g-1}_{r=1}\alpha_{\overline{V}_{i_r},V_{i_{r+1}}}^{W_r}\otimes \Id_{\overline{V}_{i_g}})\circ f_i\otimes\Theta_2\big((Id_{V^*_{i_1}}\otimes \bigotimes^{g-1}_{r=1}\alpha_{\overline{V}_{i_r},V_{i_{r+1}}}^{W_r}\otimes \Id_{V_{i_g}})\circ \bigotimes^g_{r=1}(coev_{\overline{V}_{i_r}})\big) 
$$
See the following diagram.
\[
\includegraphics[width=400pt]{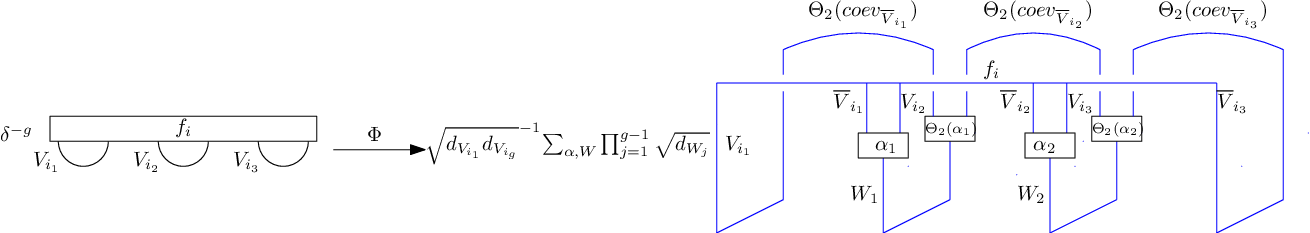}
\]
\begin{lem}\label{lem:samedim}
The dimensions of the two vector spaces are equal,
 \[
 \dim(\Psi_g)=\dim(Cf^1(g+1)).
 \]
\end{lem}
\begin{proof}
 $\dim(\Psi_g)=\sum_{i\in I^g}\dim(\Psi^i_g)$, and we have
 \[
 \begin{aligned}
  \dim(\Psi^i_g)&= \dim(\Hom(1,\bigotimes_{r=1}^g(V_{i_r}\otimes \overline{V}_{i_r})))\\
  &=\dim(\Hom(\bigotimes_{r=1}^g V_{i_r}, \bigotimes_{r=1}^g V_{i_r}))\\
  &=\sum_{W\in \Irr(\CC)}\dim (\Hom(\bigotimes_{r=1}^g V_{i_r},W))\dim (\Hom(W,\bigotimes_{r=1}^g V_{i_r}))\\
  &=\sum_{W\in \Irr(\CC)}\dim (\Hom(1,W\otimes\bigotimes_{r=1}^g \overline{V}_{i_r}))\dim (\Hom(1,\overline{W}\otimes \bigotimes_{r=1}^g V_{i_r}))
 \end{aligned} 
 \]
\end{proof}

\begin{Thm}\label{thm:Rel_to_RT}
 $\Phi_g$ is an (isometric) isomorphism of $\SMod(\Sigma_g)$ representations. Here we also denote the $\SMod(\Sigma_g)$ action on $\Psi_g$ by $T_j(0\leq j\leq 2g+1  )$ for simplicity.
\end{Thm}

\begin{proof}
 We first prove $\Phi_g$ is an (isometric) isomorphism using the non-degenerated bilinear form (or inner product in the unitary case) for these two spaces.
\[
\includegraphics[width=400pt]{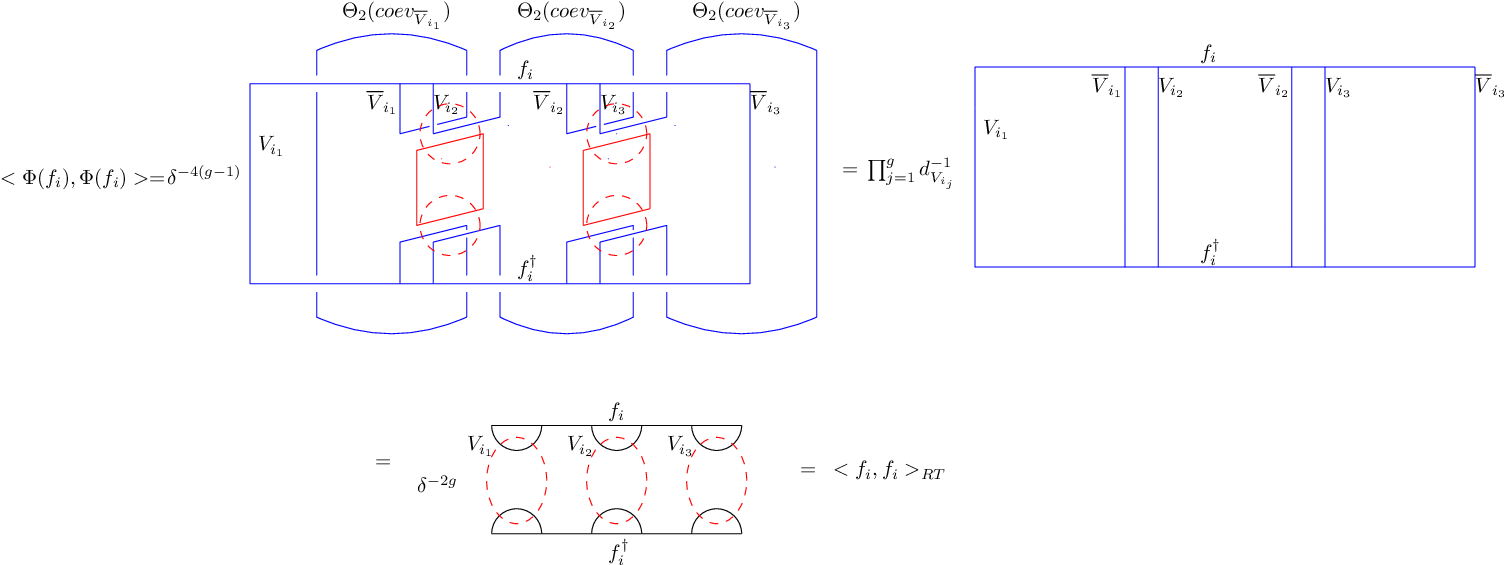}
\]
Now it follows from Lemma \ref{lem:samedim}.
Then we prove they are $\SMod(\Sigma_g)$-equivariant.
\[
\includegraphics[width=300pt]{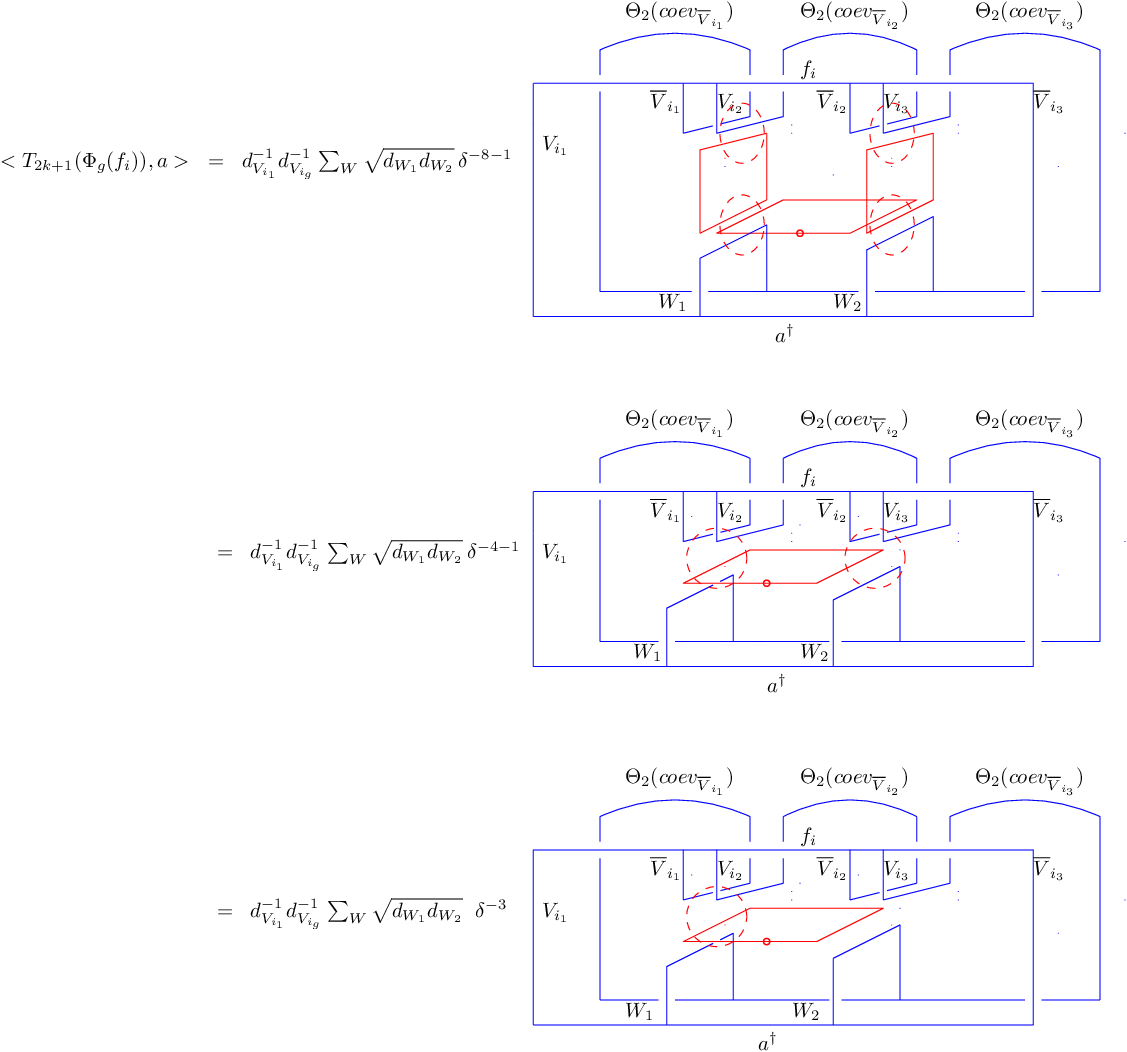}
\]
\[
\includegraphics[width=325pt]{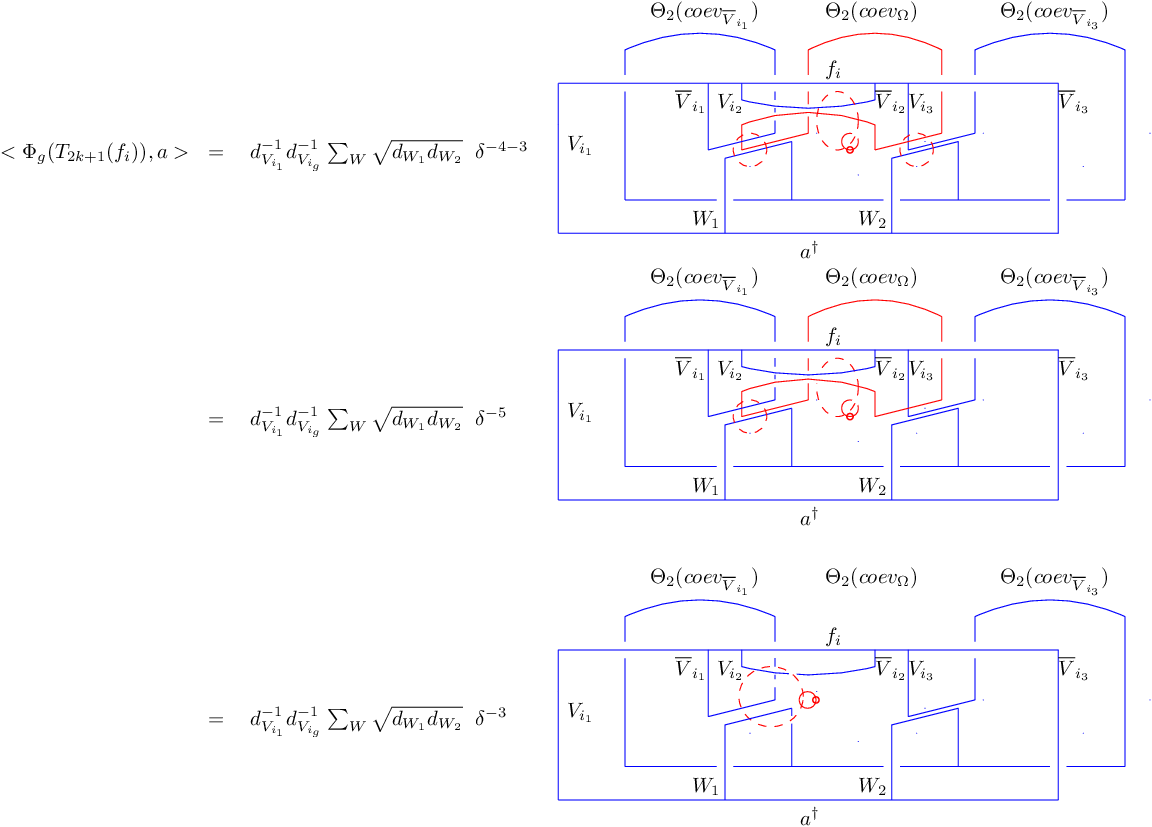}
\]
The equivalence of $T_{2k}(0\leq k\leq g )$ action is straightforward, hence the theorem is proved.
\end{proof}
\begin{rmk}
 The equivalence is easier to see when working with Temperley-Lieb-Jones modular categories, see \cite{Ruan22} for more calculations and formulas of the similar Fourier pairing in this basis.  
\end{rmk}

\begin{rmk}
 It will be interesting to see the relations with the general orbifold constructions, developed in \cite{CMRSS24, CRS19, CRS20}, for (defect) TQFTs.
 
\end{rmk}
\subsection{On Property F conjecture}
A conjecture of Naidu and Rowell (known as Property $F$ conjecture) states as follows
\begin{conj}(\cite{NR11})
The braid group representations associated with any object in weakly integral braided fusion categories have finite images. (The property that the braid group image is finite is called property $F$)
\end{conj}
the conjecture is verified for metaplectic categories \cite{RW17, GRR20}, and for general weakly group-theoretical categories \cite{GN21}. Since gauging preserves integrality. We ask the following related question,
\begin{ques}\label{ques:propertyF}
Does $\Zn2$ permutation gauging preserve property $F$?    
\end{ques}
Our detailed construction has the following indication, the braid group representation associated with object $\hat{1}$ is the same as the symmetric mapping class group representation (Corollary \ref{cor:SMod_rep} and Theorem \ref{thm:Rel_to_RT}) of the original category. In particular, the space $\Hom_{D^{\Zn2}}((1,+)\oplus (1.-), (\hat{1},+)^{\otimes 2n})$ is naturally isomorphic to $Cf^1(n)$. Hence the braid group representation is equivalent (as we work in the same planar algebra). Now the question \ref{ques:propertyF} leads to the following one,

\begin{ques}
  Given a modular category $\CC$, if $\CC$ has property $F$, does symmetric mapping class group representation (described in Corollary \ref{cor:SMod_rep}) associated with $\CC$ have a finite image?  
\end{ques}
Since the property being weakly group-theoretical is invariant under gauging \cite{ENO11}, together with the results of \cite{GN21}, we can obtain the following result,
\begin{Cor}
The symmetric mapping class group representation of a weakly group-theoretical category has a finite image.    
\end{Cor}
\begin{proof}
Suppose the modular category $\CC$ is weakly group-theoretical, then the category $\DD=(\CC\boxtimes \CC)^{\Zn2}_{\times, \Zn2}$ of its $\Zn2$ permutation gauging is again weakly group-theoretical \cite{ENO11}, hence has property $F$ by \cite{GN21}. Therefore we have the braid group image obtained from the
$\Hom_{\DD}((1,+)\oplus (1.-), (\hat{1},+)^{\otimes 2n})$ is finite which is equivalent to the finite image of the $\SMod(\Sigma_{n-1})$ representation by previous argument and Theorem \ref{thm:Rel_to_RT}.
\end{proof}
Moreover, our symmetric mapping class group representation associated with $\CC(\Zn {(2r+1)},q)$ is equivalent (again by Theorem \ref{thm:Rel_to_RT}) to ones calculated in \cite{GN21} and \cite{BWrep}, hence the image is finite, as a direct consequence, we have 

\begin{Cor}\cite{GRR20}\label{cor:meta_prop_f}
$SO(N)_2$ for $N=2r+1$ have property $F$. 
\end{Cor}
\begin{proof}
It is clear that $SO(N)_2$ is generated by the object $(\hat{1},+)$, hence it suffices to consider the braid group representation on  $\Hom_{SO(N)_2}((1,+), (\hat{1},+)^{\otimes 2n})$, which is an invariant subspace of $\Hom_{SO(N)_2}((1,+)\oplus (1,-), (\hat{1},+)^{\otimes 2n})$. And the latter space is naturally isomorphic to the space $Cf^1(n)$.
\end{proof}

\subsection{On Cyclic Permutation Gauging}
In \cite{DRX21} and \cite{DXY22}, the formula of the $S$-matrix for the $\Zn n$ permutation orbifold is given in the setting of vertex operator algebras. We ask the following natural question.
\begin{ques}
Can our techniques be generalized to the case of $\Zn n$ permutation gauging?   
\end{ques}
We believe the general idea is the same, and one needs to consider $n$-layer generalized configuration space instead.  Hence the main difficulty here will be defining all these delicate structures, in particular, the braiding structures.

\subsection{On Quantum Information}

If we take modular category $\CC=\CC(\Zn d,q)$, the $\Zn2$ permutation extension $\DD$ has simple objects of the form $(j_1, j_2)$, $\hat{j}$ where $0\leq j,j_1,j_2\leq d-1$ (see Theorem \ref{cat_of_Z_2-extension}). 
Let $q(j)=e^{\frac{2\pi ij^2}{d}}$.
Its planar subalgebra, generated by the object $\tau=\hat{1}$, along with its braiding, corresponds to the parafermion planar para algebras introduced in \cite{JL17}. The charge can be interpreted as an object connecting a local $\tau$-color string and a $\tau$-color at infinity on the left, which is given by objects $(j,\bar{j}=d-j)$ in our case.

For example, the square root of phase in \cite[Proposition~2.15]{JL17} is given by the normalization factor $\eta^{\frac{1}{2}}$ in our case, also see \cite[Section~8]{JL17} for the corresponding braiding structure and $\Zn2$ action. Moreover, the image group of braid group representation is calculated in \cite[Proposition~9.1]{JL17}, which can also be applied to prove Corollary \ref{cor:meta_prop_f}.

This parafermion planar para algebras has been used to construct the Quon language for qudits to study quantum information \cite{LWJ17}. 
It will be interesting to extend the Quon language using the full $\Zn2$ twisted category $D$ as string-net on the surface, together with the recent development of the $2+1$ alterfold theory \cite{LMWW23a,LMWW23b}.

\bibliographystyle{abbrv}
\bibliography{Reference}

\begin{thebibliography}{10}

\bibitem{ACRW16}
E.~Ardonne, M.~Cheng, E.~C. Rowell, and Z.~Wang.
\newblock Classification of metaplectic modular categories.
\newblock {\em J. Algebra}, 466:141--146, 2016.

\bibitem{BK01}
B.~Bakalov and A.~Kirillov, Jr.
\newblock {\em Lectures on tensor categories and modular functors}, volume~21
  of {\em University Lecture Series}.
\newblock American Mathematical Society, Providence, RI, 2001.

\bibitem{Bantay98}
P.~Bantay.
\newblock Characters and modular properties of permutation orbifolds.
\newblock {\em Phys. Lett. B}, 419(1-4):175--178, 1998.

\bibitem{Bantay02}
P.~Bantay.
\newblock Permutation orbifolds.
\newblock {\em Nuclear Phys. B}, 633(3):365--378, 2002.

\bibitem{BBCW19}
M.~Barkeshli, P.~Bonderson, M.~Cheng, and Z.~Wang.
\newblock Symmetry fractionalization, defects, and gauging of topological
  phases.
\newblock {\em Phys. Rev. B}, 100:115147, Sep 2019.

\bibitem{BM10}
M.~Barkeshli and X.-G. Wen.
\newblock
  $u(1)\ifmmode\times\else\texttimes\fi{}u(1)\ensuremath{\rtimes}{Z}_{2}$
  chern-simons theory and ${Z}_{4}$ parafermion fractional quantum hall states.
\newblock {\em Phys. Rev. B}, 81:045323, Jan 2010.

\bibitem{BS11}
T.~Barmeier and C.~Schweigert.
\newblock A geometric construction for permutation equivariant categories from
  modular functors.
\newblock {\em Transform. Groups}, 16(2):287--337, 2011.

\bibitem{BH71}
J.~S. Birman and H.~M. Hilden.
\newblock {\em On the mapping class groups of closed surfaces as covering
  spaces}.
\newblock Ann. of Math. Stud., No. 66. Princeton Univ. Press, Princeton, NJ,
  1971.

\bibitem{BH73}
J.~S. Birman and H.~M. Hilden.
\newblock On isotopies of homeomorphisms of {R}iemann surfaces.
\newblock {\em Ann. of Math. (2)}, 97:424--439, 1973.

\bibitem{BWrep}
W.~Bloomquist and Z.~Wang.
\newblock On topological quantum computing with mapping class group
  representations.
\newblock {\em J. Phys. A}, 52(1):015301, 23, 2019.

\bibitem{BHS98}
L.~Borisov, M.~B. Halpern, and C.~Schweigert.
\newblock Systematic approach to cyclic orbifolds.
\newblock {\em Internat. J. Modern Phys. A}, 13(1):125--168, 1998.

\bibitem{BHP12}
A.~Brothier, M.~Hartglass, and D.~Penneys.
\newblock Rigid {$C^*$}-tensor categories of bimodules over interpolated free
  group factors.
\newblock {\em J. Math. Phys.}, 53(12):123525, 43, 2012.

\bibitem{BB19}
D.~Bulmash and M.~Barkeshli.
\newblock Gauging fractons: Immobile non-abelian quasiparticles, fractals, and
  position-dependent degeneracies.
\newblock {\em Phys. Rev. B}, 100:155146, Oct 2019.

\bibitem{BN13}
S.~Burciu and S.~Natale.
\newblock Fusion rules of equivariantizations of fusion categories.
\newblock {\em J. Math. Phys.}, 54(1):013511, 21, 2013.

\bibitem{CMRSS24}
N.~Carqueville, V.~Mulevi\v{c}ius, I.~Runkel, G.~Schaumann, and D.~Scherl.
\newblock Reshetikhin--{T}uraev {TQFT}s {C}lose {U}nder {G}eneralised
  {O}rbifolds.
\newblock {\em Comm. Math. Phys.}, 405(10):Paper No. 242, 2024.

\bibitem{CRS19}
N.~Carqueville, I.~Runkel, and G.~Schaumann.
\newblock Orbifolds of {$n$}-dimensional defect {TQFT}s.
\newblock {\em Geom. Topol.}, 23(2):781--864, 2019.

\bibitem{CRS20}
N.~Carqueville, I.~Runkel, and G.~Schaumann.
\newblock Orbifolds of {R}eshetikhin-{T}uraev {TQFT}s.
\newblock {\em Theory Appl. Categ.}, 35:Paper No. 15, 513--561, 2020.

\bibitem{CGPW16}
S.~X. Cui, C.~Galindo, J.~Y. Plavnik, and Z.~Wang.
\newblock On gauging symmetry of modular categories.
\newblock {\em Comm. Math. Phys.}, 348(3):1043--1064, 2016.

\bibitem{DRX21}
C.~Dong, L.~Ren, and F.~Xu.
\newblock {$S$}-matrix in orbifold theory.
\newblock {\em J. Algebra}, 568:139--159, 2021.

\bibitem{DXY22}
C.~Dong, F.~Xu, and N.~Yu.
\newblock {$S$}-matrix in permutation orbifolds.
\newblock {\em J. Algebra}, 606:851--876, 2022.

\bibitem{DGNO}
V.~Drinfeld, S.~Gelaki, D.~Nikshych, and V.~Ostrik.
\newblock On braided fusion categories. {I}.
\newblock {\em Selecta Math. (N.S.)}, 16(1):1--119, 2010.

\bibitem{EJP19}
C.~Edie-Michell, C.~Jones, and J.~Y. Plavnik.
\newblock Fusion rules for {$\Bbb{Z}/2\Bbb{Z}$} permutation gauging.
\newblock {\em J. Math. Phys.}, 60(10):102302, 15, 2019.

\bibitem{EGNO}
P.~Etingof, S.~Gelaki, D.~Nikshych, and V.~Ostrik.
\newblock {\em Tensor categories}, volume 205 of {\em Mathematical Surveys and
  Monographs}.
\newblock American Mathematical Society, Providence, RI, 2015.

\bibitem{ENO10}
P.~Etingof, D.~Nikshych, and V.~Ostrik.
\newblock Fusion categories and homotopy theory.
\newblock {\em Quantum Topol.}, 1(3):209--273, 2010.
\newblock With an appendix by Ehud Meir.

\bibitem{ENO11}
P.~Etingof, D.~Nikshych, and V.~Ostrik.
\newblock Weakly group-theoretical and solvable fusion categories.
\newblock {\em Adv. Math.}, 226(1):176--205, 2011.

\bibitem{EG22}
D.~E. Evans and T.~Gannon.
\newblock Reconstruction and local extensions for twisted group doubles, and
  permutation orbifolds.
\newblock {\em Trans. Amer. Math. Soc.}, 375(4):2789--2826, 2022.

\bibitem{GJ19}
T.~Gannon and C.~Jones.
\newblock Vanishing of categorical obstructions for permutation orbifolds.
\newblock {\em Comm. Math. Phys.}, 369(1):245--259, 2019.

\bibitem{GNN09}
S.~Gelaki, D.~Naidu, and D.~Nikshych.
\newblock Centers of graded fusion categories.
\newblock {\em Algebra Number Theory}, 3(8):959--990, 2009.

\bibitem{Ghosh11}
S.~K. Ghosh.
\newblock Planar algebras: a category theoretic point of view.
\newblock {\em J. Algebra}, 339:27--54, 2011.

\bibitem{GN21}
J.~Green and D.~Nikshych.
\newblock On the braid group representations coming from weakly
  group-theoretical fusion categories.
\newblock {\em J. Algebra Appl.}, 20(1):Paper No. 2150210, 20, 2021.

\bibitem{Gui21}
B.~Gui.
\newblock Genus-zero permutation-twisted conformal blocks for tensor product
  vertex operator algebras: The tensor-factorizable case.
\newblock {\em arXiv:2111.04662}, 2021.

\bibitem{GJS10}
A.~Guionnet, V.~F.~R. Jones, and D.~Shlyakhtenko.
\newblock {\em Random matrices, free probability, planar algebras and
  subfactors}, volume~11 of {\em Clay Math. Proc.}
\newblock Amer. Math. Soc., Providence, RI, 2010.

\bibitem{GRR20}
P.~Gustafson, E.~C. Rowell, and Y.~Ruan.
\newblock Metaplectic categories, gauging and property {$F$}.
\newblock {\em Tohoku Math. J. (2)}, 72(3):411--424, 2020.

\bibitem{HNW13}
M.~B. Hastings, C.~Nayak, and Z.~Wang.
\newblock {Metaplectic Anyons, Majorana Zero Modes, and their Computational
  Power}.
\newblock {\em Phys. Rev.}, B87(16):165421, 2013.

\bibitem{HNW14}
M.~B. Hastings, C.~Nayak, and Z.~Wang.
\newblock On metaplectic modular categories and their applications.
\newblock {\em Comm. Math. Phys.}, 330(1):45--68, 2014.

\bibitem{HPT23}
A.~G. Henriques, D.~Penneys, and J.~Tener.
\newblock Planar algebras in braided tensor categories.
\newblock {\em Mem. Amer. Math. Soc.}, 282(1392):vi+106, 2023.

\bibitem{JL17}
A.~Jaffe and Z.~Liu.
\newblock Planar para algebras, reflection positivity.
\newblock {\em Comm. Math. Phys.}, 352(1):95--133, 2017.

\bibitem{Jones17}
V.~Jones.
\newblock Some unitary representations of {T}hompson's groups {$F$} and {$T$}.
\newblock {\em J. Comb. Algebra}, 1(1):1--44, 2017.

\bibitem{Jones21}
V.~F.~R. Jones.
\newblock Planar algebras, {I}.
\newblock {\em New Zealand J. Math.}, 52:1--107, 2021 [2021--2022].

\bibitem{KLX05}
V.~G. Kac, R.~Longo, and F.~Xu.
\newblock Solitons in affine and permutation orbifolds.
\newblock {\em Comm. Math. Phys.}, 253(3):723--764, 2005.

\bibitem{KLM01}
Y.~Kawahigashi, R.~Longo, and M.~M\"{u}ger.
\newblock Multi-interval subfactors and modularity of representations in
  conformal field theory.
\newblock {\em Comm. Math. Phys.}, 219(3):631--669, 2001.

\bibitem{Kir04}
A.~{Kirillov}, Jr.
\newblock {On $G$-equivariant modular categories}.
\newblock {\em ArXiv Mathematics e-prints}, Jan. 2004.

\bibitem{Liu19}
Z.~Liu.
\newblock Quon language: surface algebras and {F}ourier duality.
\newblock {\em Comm. Math. Phys.}, 366(3):865--894, 2019.

\bibitem{LMWW23a}
Z.~Liu, S.~Ming, Y.~Wang, and J.~Wu.
\newblock {3-Alterfolds and Quantum Invariants}.
\newblock arXiv preprint arXiv:2307.12284.

\bibitem{LMWW23b}
Z.~Liu, S.~Ming, Y.~Wang, and J.~Wu.
\newblock {Alterfold Topological Quantum Field Theory}.
\newblock arXiv preprint arXiv:2312.06477.

\bibitem{LMP20}
Z.~Liu, S.~Morrison, and D.~Penneys.
\newblock Lifting shadings on symmetrically self-dual subfactor planar
  algebras.
\newblock In {\em Topological phases of matter and quantum computation}, volume
  747 of {\em Contemp. Math.}, pages 51--61. Amer. Math. Soc., [Providence],
  RI, [2020] \copyright 2020.

\bibitem{LWJ17}
Z.~Liu, A.~Wozniakowski, and A.~M. Jaffe.
\newblock Quon 3{D} language for quantum information.
\newblock {\em Proc. Natl. Acad. Sci. USA}, 114(10):2497--2502, 2017.

\bibitem{LX19}
Z.~Liu and F.~Xu.
\newblock Jones-{W}assermann subfactors for modular tensor categories.
\newblock {\em Adv. Math.}, 355:106775, 40, 2019.

\bibitem{LX04}
R.~Longo and F.~Xu.
\newblock Topological sectors and a dichotomy in conformal field theory.
\newblock {\em Comm. Math. Phys.}, 251(2):321--364, 2004.

\bibitem{MW21}
D.~Margalit and R.~R. Winarski.
\newblock Braids groups and mapping class groups: the {B}irman-{H}ilden theory.
\newblock {\em Bull. Lond. Math. Soc.}, 53(3):643--659, 2021.

\bibitem{MPS10}
S.~Morrison, E.~Peters, and N.~Snyder.
\newblock Skein theory for the {$D_{2n}$} planar algebras.
\newblock {\em J. Pure Appl. Algebra}, 214(2):117--139, 2010.

\bibitem{muger03}
M.~M\"{u}ger.
\newblock From subfactors to categories and topology. {II}. {T}he quantum
  double of tensor categories and subfactors.
\newblock {\em J. Pure Appl. Algebra}, 180(1-2):159--219, 2003.

\bibitem{Muger05}
M.~M\"{u}ger.
\newblock Conformal orbifold theories and braided crossed {$G$}-categories.
\newblock {\em Comm. Math. Phys.}, 260(3):727--762, 2005.

\bibitem{NR11}
D.~Naidu and E.~C. Rowell.
\newblock A finiteness property for braided fusion categories.
\newblock {\em Algebr. Represent. Theory}, 14(5):837--855, 2011.

\bibitem{RT91}
N.~Reshetikhin and V.~G. Turaev.
\newblock Invariants of {$3$}-manifolds via link polynomials and quantum
  groups.
\newblock {\em Invent. Math.}, 103(3):547--597, 1991.

\bibitem{Row05f}
E.~C. Rowell.
\newblock From quantum groups to unitary modular tensor categories.
\newblock In {\em Representations of algebraic groups, quantum groups, and
  {L}ie algebras}, volume 413 of {\em Contemp. Math.}, pages 215--230. Amer.
  Math. Soc., Providence, RI, 2006.

\bibitem{RW17}
E.~C. Rowell and H.~Wenzl.
\newblock {${\rm SO}(N)_2$} braid group representations are {G}aussian.
\newblock {\em Quantum Topol.}, 8(1):1--33, 2017.

\bibitem{Ruan22}
Y.~Ruan.
\newblock Projective representations of hecke groups from topological quantum
  field theory.
\newblock {\em arXiv:2203.10745}, 2022.

\bibitem{TurHQFT}
V.~Turaev.
\newblock {\em Homotopy quantum field theory}, volume~10 of {\em EMS Tracts in
  Mathematics}.
\newblock European Mathematical Society (EMS), Z\"{u}rich, 2010.
\newblock Appendix 5 by Michael M\"{u}ger and Appendices 6 and 7 by Alexis
  Virelizier.

\bibitem{Tur10}
V.~G. Turaev.
\newblock {\em Quantum invariants of knots and 3-manifolds}, volume~18 of {\em
  De Gruyter Studies in Mathematics}.
\newblock Walter de Gruyter \& Co., Berlin, revised edition, 2010.

\end{thebibliography}
\end{document}